\documentclass[11pt,reqno]{amsart}  
\usepackage{bbm}
\usepackage{amssymb,mathrsfs,mathtools,dsfont,pbox,subfig}
\usepackage[foot]{amsaddr}

\usepackage[usenames,dvipsnames]{xcolor}
\usepackage{graphicx,epsfig}
\usepackage{makecell}
\usepackage[hidelinks]{hyperref}
\usepackage{dsdshorthand}
\usepackage{cite}
\hypersetup{
    colorlinks,
    linkcolor={red!50!black},
    citecolor={blue!50!black},
    urlcolor={blue!80!black}
}
\usepackage{tikz}
\usepackage{caption} 
\renewcommand{\arraystretch}{0.85}
\renewcommand{\leq}{\leqslant}
\renewcommand{\geq}{\geqslant}
\usepackage{braket}
\usepackage[para]{threeparttable}
\usepackage{slashed}
\usepackage{booktabs,multirow}
\usepackage[top=1in, bottom=1.25in, left=1.2in, right=1.2in,footskip=.5in]{geometry}
\usepackage{tikz-cd}
\numberwithin{equation}{section}
\numberwithin{figure}{section}
\makeatletter
\DeclarePairedDelimiter\floor{\lfloor}{\rfloor}

\makeatletter

\theoremstyle{plain}
\newtheorem*{thm*}{Theorem}
\newtheorem{thm}{Theorem}[section]

\newtheorem{prop}[thm]{Proposition}
\newtheorem{conj}[thm]{Conjecture}

\theoremstyle{definition}
\newtheorem{defn}[thm]{Definition}
\newtheorem*{defn*}{Definition}

\newtheorem{rem}[thm]{Remark}

\newtheorem{linpro}{Linear Program}
\newtheorem{SDP}{Semidefinite Program}

\DeclarePairedDelimiterX{\abs}[1]{\lvert}{\rvert}{\ifblank{#1}{{}\cdot{}}{#1}}
\makeatother

\newcommand{\trace}{\text{Tr }}
\newcommand{\sclambda}{\lambda^{(0)}} 
\newcommand{\splambda}{\lambda^{(1/2)}} 
\newcommand{\sct}{t^{(0)}} 
\newcommand{\spt}{t^{(1/2)}} 
\newtheorem*{thm:main}{Theorem \ref{thm:main}}
\newtheorem*{thm:prop}{Proposition \ref{thm:prop}}
\newcommand{\coloneq}{:~\!\!=}
\newcommand{\strangefont}[1]{{\mathbf{#1}}}

\newcommand{\DSD}[1]{{\color{red}[DSD: #1]}}

\begin{document}
\title{Bounds on spectral gaps of Hyperbolic spin surfaces} 

\author{Elliott Gesteau\textsuperscript{1}}
\address[1]{Division of Physics, Mathematics, and Astronomy, California Institute of Technology, Pasadena, CA 91125, USA}
\address[2]{Walter Burke Institute for Theoretical Physics,  California Institute of Technology,  Pasadena, CA 91125, USA}
\curraddr{}
\email{egesteau@caltech.edu}

\author{Sridip Pal\textsuperscript{1,2}}
\email{sridip@caltech.edu}

\author{ David Simmons--Duffin\textsuperscript{1,2}}
\email{dsd@caltech.edu}

\author{Yixin Xu\textsuperscript{1,2}}
\email{yixinxu@caltech.edu}

%
\begin{abstract}
We describe a method for constraining Laplacian and Dirac spectra of two dimensional compact orientable hyperbolic spin manifolds and orbifolds. The key ingredient is an infinite family of identities satisfied by the spectra. These spectral identities follow from the consistency between 1) the spectral decomposition of functions on the spin bundle into irreducible representations of $\mathrm{SL}(2,\mathbb{R})$ and 2) associativity of pointwise multiplication of functions. Applying semidefinite programming methods to our identities produces rigorous upper bounds on the Laplacian spectral gap as well as on the Dirac spectral gap conditioned on the former. In several examples, our bounds are nearly sharp; a numerical algorithm based on the Selberg trace formula shows that the $[0;3,3,5]$ orbifold, a particular surface with signature $[1;3]$, and the Bolza surface nearly saturate the bounds 
at genus $0$, $1$ and $2$ respectively. Under additional assumptions on the number of harmonic spinors carried by the spin-surface, we obtain more restrictive bounds on the Laplacian spectral gap. In particular, these bounds apply to hyperelliptic surfaces. We also determine the set of Laplacian spectral gaps attained by all compact orientable two-dimensional hyperbolic spin orbifolds. We show that  this set is upper bounded by $12.13798$; this bound is nearly saturated by the $[0;3,3,5]$ orbifold, whose first non-zero Laplacian eigenvalue is $\lambda^{(0)}_1\approx 12.13623$.

\end{abstract}

\date{}

\makeatletter
\gdef\@fpheader{}
\makeatother


\maketitle

\newpage
\setcounter{tocdepth}{1}
\tableofcontents
\section{Introduction}


In this paper, we derive new upper bounds on the first nonzero eigenvalue of Laplace and Dirac operators on compact orientable hyperbolic surfaces and orbifolds with spin. Our approach is inspired by a technique known in high energy physics as the conformal bootstrap \cite{Rattazzi:2008pe,Simmons-Duffin:2016gjk,Kos_2016,PolandNature,Poland:2018epd}. Bootstrap methods have already been successfully applied to bound the first nonzero eigenvalue of the Laplace operator on hyperbolic surfaces \cite{Kravchuk:2021akc,Bonifacio:2021msa,Bonifacio:2021aqf} and hyperbolic 3-manifolds \cite{Bonifacio:2023ban}, see also \cite{Bonifacio:2019ioc,Bonifacio:2020xoc}. A remarkable achievement of \cite{Kravchuk:2021akc,Bonifacio:2021msa} is that for several low genera, it improves on the Yang--Yau bounds \cite{yang1980eigenvalues} and its successors \cite{ilias,10.2969/jmsj/85898589,karpukhin2022first}, as well as the bounds of \cite{Huber_1980}, which collectively had been the strongest bounds available for decades. (In the large genus limit, \cite{Huber_1980} gives the optimal bound, as was recently shown in \cite{hide2023near}.)

Recently, \cite{bourque2023linear} used linear programming methods and the Selberg trace formula (STF) to obtain new bounds on Laplacian eigenvalues and other geometric quantities for hyperbolic surfaces. These bounds are stronger than the ones of \cite{Kravchuk:2021akc} for genus $g>6$. However, the technique of \cite{bourque2023linear} cannot immediately be extended to the Dirac operator due to a lack of positivity in the STF with spin. Furthemore, \cite{bourque2023linear}  treats orbifolds and manifolds on a separate footing, while \cite{Kravchuk:2021akc} treats them uniformly. In this paper, we follow the methods of \cite{Kravchuk:2021akc}. 




The basic idea of \cite{Kravchuk:2021akc} is to leverage the fact that the orthonormal frame bundle of any compact orientable hyperbolic surface or orbifold can be written as $\Gamma\backslash \mathrm{PSL}(2,\mathbb{R})$, where $\Gamma$ is a cocompact Fuchsian group. Constraints on spectra come from combining the associativity of pointwise multiplication of functions on $\Gamma\backslash \mathrm{PSL}(2,\mathbb{R})$ with the representation theory of $\mathrm{PSL}(2,\mathbb{R})$. By decomposing functions into irreducible representations of $\mathrm{PSL}(2,\mathbb{R})$, and imposing associativity of their pointwise products, one obtains an infinite set of constraint equations on the eigenvalues of the Laplace operator. Linear programming techniques applied to these equations yield bounds on the spectrum.

We adapt this method to study surfaces and orbifolds with spin structure as follows. Two operators are relevant in this context: the Laplace operator, which is insensitive to the spin structure, and the Dirac operator, which depends on the spin structure. In group-theoretic language, adding a spin structure amounts to considering a quotient of the form  $\widetilde{\Gamma}\backslash \mathrm{SL}(2,\mathbb{R})$, where $\widetilde{\Gamma}$ is a lift of the cocompact Fuchsian group $\Gamma$ to $\mathrm{SL}(2,\mathbb{R})$. Each lift corresponds to a spin structure. We can then study the interplay between associativity of pointwise multiplication of functions on $\widetilde{\Gamma}\backslash \mathrm{SL}(2,\mathbb{R})$ and the representation theory of $\mathrm{SL}(2,\mathbb{R})$. By decomposing functions into irreducible representations of $\mathrm{SL}(2,\mathbb{R})$, and imposing associativity of their pointwise products, we derive spectral identities, which are amenable to linear/semidefinite programming.

Our bounds depend on the number of linearly-independent modular forms of various weights possessed by a surface. For forms of weight $2$ or higher, the Riemann--Roch theorem determines their number from purely topological information: the genus and the orders of orbifold points. The situation is different for weight 1 forms (which correspond to harmonic spinors): the number of weight 1 forms can be different for surfaces with the same topology but different geometries. For example, except in genus 4 and 6, a surface carries the maximal possible number of weight 1 modular forms if and only if it is hyperelliptic. This allows us to obtain stronger bounds for hyperelliptic surfaces.\\

Among our results, we prove the following theorems:

\begin{thm}[Universal upper bound on Laplacian eigenvalue: I]\label{thm:1*}
    Given a compact orientable hyperbolic spin orbifold $X$, the first non-zero eigenvalue of the Laplacian operator, $\lambda_1^{(0)}(X)$ satisfies 
    \begin{equation*}
        \lambda_1^{(0)}(X)< 12.137980 \,.
    \end{equation*}
\end{thm}

\begin{rem}
    The above bound is nearly saturated by $[0;3,3,5]$, a genus $0$ orbifold with three orbifold singularities of order $3$, $3$ and $5$. We have $\lambda_1^{(0)} ([0;3,3,5])\simeq 12.13623$.
\end{rem}
\begin{thm}[Universal upper bound on Laplacian eigenvalue: II]\label{thm:2*}
    Given a compact orientable hyperbolic spin orbifold $X$ admitting a harmonic spinor, the first non-zero eigenvalue of the Laplacian operator, $\lambda_1^{(0)}(X)$ satisfies 
    \begin{equation*}
        \lambda_1^{(0)}(X)< 4.763782\,.
    \end{equation*}
\end{thm}

\begin{rem}\label{remark:[1,3]}
  The above bound is nearly saturated by the most symmetric point (see fig.~\ref{fig:1,3}) of the moduli space of $[1;3]$ i.e a 
 torus with one orbifold singularity of order $3$. Let us call the orbifold $[1;3]_{sym}$. Using \texttt{FreeFEM++}, we learn that  $\lambda_1^{(0)} ([1;3]_{sym})\simeq 4.7609$. 
\end{rem}
\begin{figure}[!ht]
    \centering
\includegraphics[scale=0.5]{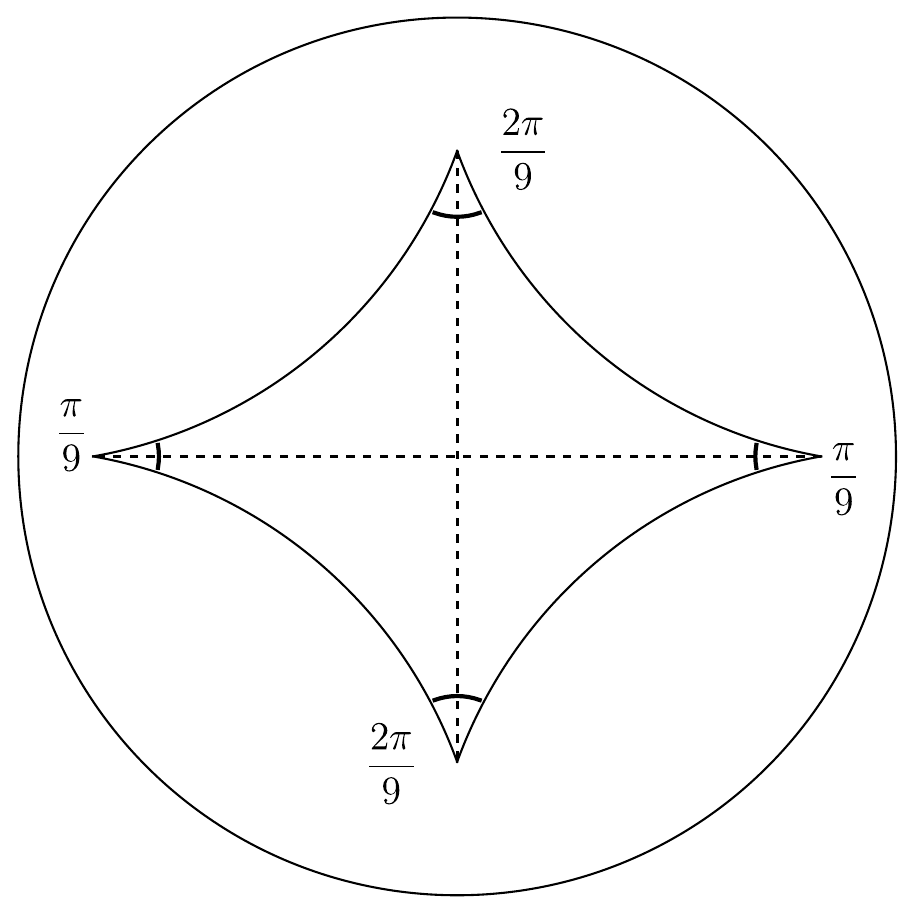}
    \caption{A fundamental domain of $\Gamma$ drawn on the Poincaré disc where $\Gamma$ is a subgroup of $\mathrm{PSL}(2,\mathbb{R})$ isomorphic to the fundamental group of the most symmetric point in the moduli space of $[1;3]$. This fundamental domain is a hyperbolic quadrilateral symmetric under reflections against the two dashed lines.} 
    \label{fig:1,3}
\end{figure}



More generally, we obtain two types of bounds. The first type are \textit{exclusion plots} in a two-dimensional parameter space, labelled by the first nonzero eigenvalues of the Laplace and Dirac operators. We use $\lambda_1^{(0)}$ to denote the first non-zero Laplacian eigenvalue. For the first non-zero Dirac eigenvalue $t^{(1/2)}_1$, we use the variable $\lambda_1^{(1/2)}:=1/4+(t_1^{(1/2)})^2$ in the plots. Examples of such plots are given on Figures \ref{fig:3,3/2g}, \ref{fig:3,3/2} and \ref{fig:1/2,1}.  Let us focus on Figure \ref{fig:3,3/2g}. For any hyperbolic spin orbifold $X$, equipped with a spin structure such that there is no harmonic spinor, the pair $(\lambda_1^{(0)}(X),\lambda_1^{(1/2)}(X))$ must lie in the union of the pink-shaded and yellow-shaded regions. If the hyperbolic spin orbifold $X$ has genus $1$ or more and is equipped with a spin structure admitting one or more harmonic spinors, then $(\lambda_1^{(0)}(X),\lambda_1^{(1/2)}(X))$ must lie in the pink shaded region. The yellow shaded region as well as the pink one has a kink at rightmost corner. Exclusion plots with kinks are abundant in the conformal bootstrap literature. Often, one can identify the object that lives near the kink. Here a similar phenomenon is true. Indeed, the corner point in the yellow-shaded region is very close to $(\lambda_1^{(0)}([0;3,3,5]),\lambda_1^{(1/2)}([0;3,3,5]))$. See Figure~\ref{fig:3,3/2zoomWithoutRes} for a zoomed-in version of the exclusion plot. For the pink shaded region, the corner point is close to $(\lambda_1^{(0)}(X),\lambda_1^{(1/2)}(X))$, where $X$ is the most symmetric point of the moduli space of $[1;3]$, equipped with the odd spin structure. See Table \ref{app:tableOfLambdaOne} in Appendix \ref{app:tableOfLambdaOne} for the numerical estimates of the eigenvalues corresponding to these nearly saturating examples.


\begin{figure}[!h]
    \centering
 \includegraphics[scale=0.9]{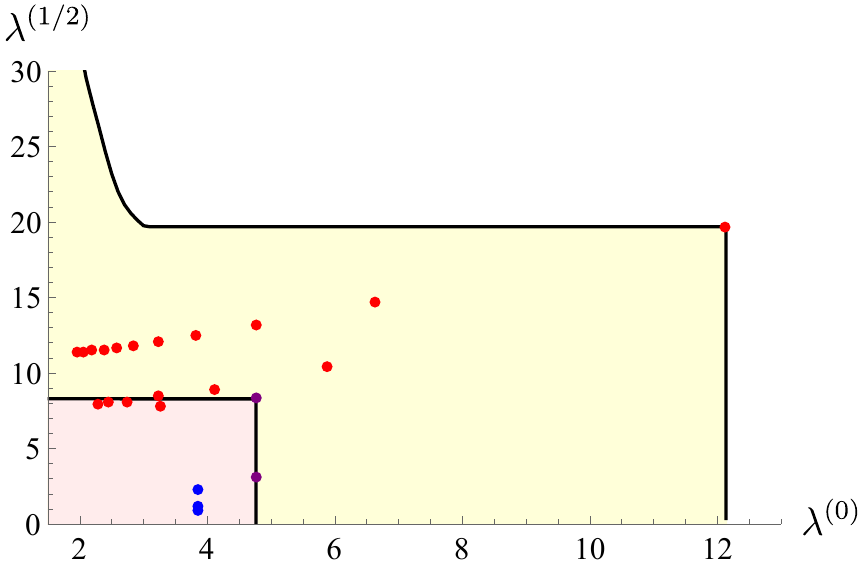}
    \caption{For any compact orientable hyperbolic spin orbifold $X$, equipped with a spin structure such that there is no harmonic spinor, $(\lambda_1^{(0)}(X),\lambda_1^{(1/2)}(X))$ must lie in the union of the pink-shaded and yellow-shaded regions.  
    Note that, even though it is not shown explicitly, we have $\lambda_{1}^{1/2}>1/4$. If the hyperbolic spin orbifold $X$ has genus $1$ or more and is equipped with a spin structure admitting one or more harmonic spinors, then $(\lambda_1^{(0)}(X),\lambda_1^{(1/2)}(X))$ must lie in the pink shaded region. 
    The red dot in the corner corresponds to $(\lambda_1^{(0)}([0;3,3,5]),\lambda_1^{(1/2)}([0;3,3,5]))$.
    The other red dots come from the eigenvalues corresponding to various hyperbolic triangles while the blue dots come from the eigenvalues corresponding to the Bolza surface, equipped with various spin structures. The purple dot in the corner of the pink shaded region corresponds to $(\lambda_1^{(0)}([1;3]_{sym}),\lambda_1^{(1/2)}([1;3]_{sym}))$, where $[1;3]_{sym}$ is equipped with an odd spin structure. See Remark~\ref{remark:[1,3]}. The other purple dot lying on the boundary of the pink shaded region corresponds to $[1;3]_{sym}$, equipped with an even spin structure.  See the table \ref{app:tableOfLambdaOne} in Appendix \ref{app:tableOfLambdaOne} for the numerical estimates of these eigenvalues. Here the origin is at $(1.5,0)$. }
    \label{fig:3,3/2g}
\end{figure}

We can derive more specific theorems from the exclusion plot. For example, we prove

\begin{thm}[Conditional upper bound on Dirac eigenvalue]\label{thm:30*}
    Given any compact orientable hyperbolic spin orbifold $X$ such that the first non-zero eigenvalue of the Laplacian operator satisfies $\lambda_1^{(0)}(X)>1.5$, we have 
    \begin{equation*}
        \lambda_1^{(1/2)}(X)< 55.9 \,.
    \end{equation*}
    Here $\lambda_1^{(1/2)}(X)=1/4+t^2$ and $|t|$ is the lowest non-zero positive eigenvalue of Dirac operator on $X$.
\end{thm}

 To estimate the Laplace and Dirac eigenvalues of a given surface and compare them with our bounds, we use a numerical method based on the STF, described in \cite{lin2022seiberg}.
Using this method, we can populate the exclusion plots as shown on Figure \ref{fig:3,3/2g}. We also derive more specific exclusion plots, for example Figures~\ref{fig:1/2,1} and \ref{fig:1/2,1zoom}, which describe the bound applicable to all compact orientable hyperbolic genus $2$ spin orbifolds. As we can see, $(\lambda_1^{(0)}(X),\lambda_1^{(1/2)}(X))$, where $X$ is the Bolza surface with odd spin strcuture, sit very close to the boundary, see the table \ref{app:tableOfLambdaOne} in Appendix \ref{app:tableOfLambdaOne} for the numerical estimates.

The second type of bound is obtained from the first type, by noticing that the first nonzero eigenvalue of the Dirac operator corresponds to $\lambda_1^{(1/2)}(X)>1/4$. 
This leads to upper bounds on the first nonzero eigenvalue of the Laplace operator for a given choice of spin structure, admitting a given number of weight $1$ modular form. These upper bounds are collected in Table \ref{table:1}. Interestingly, when the spin structure allows for enough harmonic spinors, the upper bounds become more stringent. In particular, this leads to improved upper bounds on the first Laplace eigenvalue for hyperelliptic surfaces compared to \cite{Kravchuk:2021akc}. As a warmup, we can use a small number of spectral identities to derive a very simple analytical bound for genus $g$ hyperelliptic surfaces by using the fact that they can be equipped with a spin structure, admitting $\lfloor \frac{g+1}{2}\rfloor$ harmonic spinors.

\renewcommand{\arraystretch}{0.9}
\begin{table}[!h]
	\begin{tabular}[t]{ccccc}\toprule
$g$  & $\ell_{1/2}$ & $\lambda^{(0)}_{1}(g,\ell_{1/2})$ & KMP \cite{Kravchuk:2021akc} & FP \cite{bourque2023linear}\\
\midrule
\multirow{1}{*}{2} & 0,1 & 3.8388976481 & 3.8388976481 & 4.625307\\
\cmidrule(lr){1-5}
\multirow{2}{*}{3} & 0,1 & 2.6784823893 & \multirow{2}{*}{2.6784823893} & \multirow{2}{*}{2.816427}\\
  & 2 &  1.9497673318$^{\color{red}*}$ \\
\cmidrule(lr){1-5}
\multirow{2}{*}{4}  & 0,1 & 2.1545041334 & \multirow{2}{*}{2.1545041334} & \multirow{2}{*}{2.173806}\\
  & 2 & 1.9497673318$^{\color{red}*}$\\
\cmidrule(lr){1-5}
\multirow{2}{*}{5}  & 0,1,2 & 1.8526509456 & \multirow{2}{*}{1.8526509456} & \multirow{2}{*}{1.836766}\\
   & 3 & 1.18275751${\color{red}^*}$ \\
\cmidrule(lr){1-5}
\multirow{2}{*}{6} & 0,1,2 & 1.654468363 & \multirow{2}{*}{1.654468363} & \multirow{2}{*}{1.625596}\\
& 3 & 1.386265630$^{\color{red}*}$ \\
   \cmidrule(lr){1-5}
\multirow{3}{*}{7} & 0,1,2 & 1.51326783 & \multirow{3}{*}{1.51326783} & \multirow{3}{*}{1.480008}\\
   & 3 & 1.38626563$^{\color{red}*}$ \\
   & 4 & 0.9160143\! $^{\color{red}*}$ \\
   \cmidrule(lr){1-5}
   \multirow{3}{*}{8} & 0,1,2 & 1.40690466 & \multirow{3}{*}{1.40690466} & \multirow{3}{*}{1.372804}\\
   & 3 &1.38626563$^{\color{red}*}$ \\
   & 4 & 1.0443114\! $^{\color{red}*}$ \\
   \cmidrule(lr){1-5}
   \multirow{3}{*}{9} & 0,1,2,3 & 1.32348160 & \multirow{3}{*}{1.32348160} & \multirow{3}{*}{1.289024}\\
   & 4 &  1.148493243\! $^{\color{red}*}$ \\
   & 5 &  0.78690\! $^{\color{red}*}$ \\
   \cmidrule(lr){1-5}
   \multirow{3}{*}{10} & 0,1,2,3 & 1.25602193 & \multirow{3}{*}{1.25602193} & \multirow{3}{*}{1.222189}\\
   & 4 & 1.148493243\! $^{\color{red}*}$ \\
   & 5 & 0.85292\! $^{\color{red}*}$ \\
\bottomrule
\end{tabular}
	\caption{Table of upper bounds on the Scalar-Laplacian gap as a function of genus and the number of harmonic spinors. For comparison, we also tabulate the spinor-independent bounds from \cite{Kravchuk:2021akc} and \cite{bourque2023linear}. We note that bounds from \cite{bourque2023linear} are only applicable to manifolds while the bounds from \cite{Kravchuk:2021akc} and the bounds in this paper are applicable to both manifolds and orbifolds. For $2\leqslant g\leqslant 10$, the bounds here coincide with the ones from \cite{Kravchuk:2021akc}, except when, given the genus $g$, the surface admits a spin structure such that number of harmonic spinors is strictly more than $\lceil\ell_{Max}/2\rceil$, where $\ell_{Max}=\lfloor (g+1)/2\rfloor$ is the maximal number of harmonic spinors allowed for that genus. In that case, the bounds are more restrictive; we mark them with {\color{red}*}.}
 \label{table:1}
\end{table}
\normalsize

\begin{thm}\label{thm:analytic}
    Given a genus $g\geq 3$ compact hyperbolic hyperelliptic surface $X$, we must have
    \begin{equation}
        \lambda_1^{(0)}(X) \leq \frac{\lfloor \frac{g+1}{2}\rfloor}{\lfloor \frac{g+1}{2}\rfloor-1}\,.
    \end{equation}
\end{thm}

\begin{rem}
    For a given numerical value of $g$, it is possible to achieve a much stronger bound using computer assistance. This is evident from table  \ref{table:1}, upon using the fact that the hyperelliptic surfaces can be equipped with a spin structure which admits $\ell_{1/2}=\lfloor \frac{g+1}{2}\rfloor$.
\end{rem}

Using the STF, we can compute the location of $\lambda_1^{(0)}(X)$ for many orbifolds. Combining these numerical data, we also put forward the following conjecture:
\begin{conj}\label{conj}
Let $\lambda^{(0)}_1(X)$ be the first non-zero eigenvalue of the Laplacian on an hyperbolic orbifold $X$.
  Let $E_{\mathtt{spin}} \subset \mathbb{R}_{>0}$ be the set of $\lambda^{(0)}_1(X)$ as $X$ runs over all compact hyperbolic spin orbifolds. Then $E_{\mathtt{spin}}$ is the union of the interval $\left(0, \lambda_1^{[3,3,9]}\right]$ with the finite set $\left\{\lambda_1^{[3,3,5]}\right\} \cup\left\{\lambda_1^{[3,5,5]}\right\} \cup\left\{\lambda_1^{[3,3,7]}\right\}$. When $X$ is an orbifold with signature $\left[0 ; k_1, k_2, k_3\right]$, we write $\lambda_1^{\left[k_1, k_2, k_3\right]}$ to denote $\lambda^{(0)}_1(\left[0 ; k_1, k_2, k_3\right])$. See Figure~\ref{fig:conjecture}.
\end{conj} 

\begin{figure}
    \centering
\includegraphics[scale=0.30]{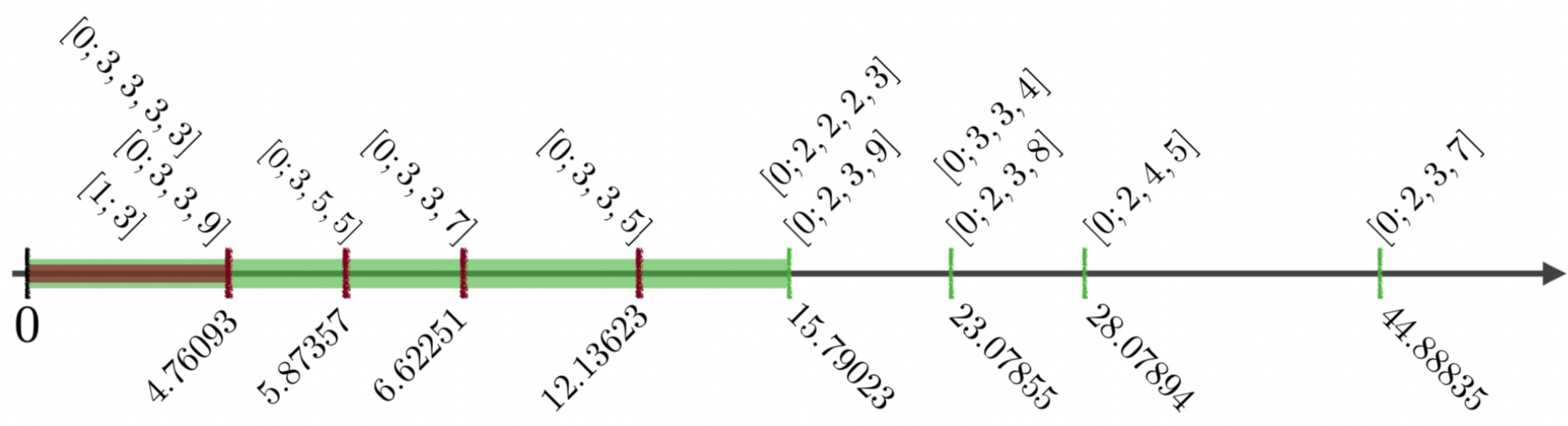}
    \caption{Pictorial description of Conjecture \ref{conj}: The green color denotes the set of $\lambda^{(0)}_1(X)$, as $X$ runs over all compact orientable hyperbolic orbifolds. The red color denotes the set of $\lambda^{(0)}_1(X)$, as $X$ runs over all compact orientable hyperbolic orbifolds equipped with a spin structure. The red continuum comes from the moduli space of $[1;3]$ and $[0;3,3,3,3]$ while the green continuum comes from the moduli space of $[0;2,2,2,3]$. This conjecture is a refinement of Conjecture $4.2$ in \cite{Kravchuk:2021akc} as one can see that the red continuum along with $3$ discrete red points lie entirely within the green continuum.}
\label{fig:conjecture}
\end{figure}
\begin{rem}
    The above conjecture is a refinement of Conjecture $4.2$ in \cite{Kravchuk:2021akc}, where the relevant set $E \subset \mathbb{R}_{>0}$ was the set of $\lambda^{(0)}_1(X)$ as $X$ runs over all orbifolds. We believe that this conjecture can partially be proven by bootstrapping the spectral identities as done for Conjecture $4.2$ in \cite{Kravchuk:2021akc}.
\end{rem}

To put the above results in context, let us make some historical remarks. The Dirac operator (in particular, the twisted Dirac operator, i.e.\ Dirac operator in presence of a gauge connection) plays numerous interesting roles in physics. An upper bound on the spectral gap of such a twisted Dirac operator was found for the first time in the context of QCD in a paper by Vafa and Witten \cite{cite-key-vafa}, followed by a paper by Atiyah \cite{ 10.1007/BFb0084593}. These two papers prove a \textit{uniform upper bound} (uniform with respect to the gauge connection) on the $n$th  nontrivial eigenvalue of the twisted Dirac operator on a compact Riemannian spin manifold in any dimension $D$. Similar methods were used to arrive at further specialized bounds in \cite{cite-key-baum} in terms of a smooth map from the manifold to the sphere. They further derived bounds for surfaces with positive sectional curvature. Ref.\cite{AGRICOLA19991} derived extrinsic bounds on Dirac eigenvalues for isometrically immersed surfaces inside $\mathbb{R}^3$  in terms of mean curvature and a smooth function from the surface to $\mathbb{R}$, as well as intrinsic bounds for genus $0$ and genus $1$ surfaces. Finally, lower bounds for the first nonzero Dirac eigenvalue have also been studied for various kinds of spin manifolds \cite{cite-key-bar}, for general manifolds of positive curvature \cite{Friedrich_1980,bar2004first}, as well as in the case of surfaces depending on some geometric parameters including the choice of spin structure \cite{bar99spin,ammann2002dirac}. For a general discussion on the spectrum of the Dirac operator on various kinds of Riemannian manifolds, we refer the readers to \cite{Friedrich_Nestke_2012} for a review, or \cite{bar2000dirac} for details on the hyperbolic case. 


Our bounds have a different flavor from these results, and it is not straightforward to directly compare them. In particular, our method encodes interesting interplays between eigenvalues of Laplacian operator and eigenvalues of Dirac operator on a hyperbolic spin orbifold, for example, see Theorem~\ref{thm:30*} or Figure \ref{fig:3,3/2g}.




\subsection{Organization of the paper}

The rest of this paper is organized as follows. In Section \ref{sec:specgeo}, we introduce the basic notions of spectral geometry that we will need. In Section \ref{sec:constraints}, we derive our constraint equations from associativity and harmonic analysis. In Section \ref{sec:selberg}, we use the Selberg trace formula to develop a numerical technique to estimate the Dirac and Laplace spectra of an orbifold of interest. In Section \ref{sec:results}, we derive our main bounds and exclusion plots.  In Section \ref{sec:discuss}, we discuss a few possible future directions that our results suggest. In Appendix \ref{app:harmonicspinors}, we relate the number of harmonic spinors a surface of low genus can carry to its geometrical properties. In Appendix \ref{app:tableOfLambdaOne} we tabulate the numerical estimates of $\lambda^{(0)}_1$ and $\lambda_1^{(1/2)}$ for various surfaces and orbifolds. 

We emphasize that while the search for functionals in our linear/semidefinite programming algorithm necessitates computer assistance, the output can be rationalized in such a way that one can explicitly exhibit the required properties of the rationalized functionals leading to the bounds (such verifications of bounds coming from linear/semidefinite programming using \textit{rational arithmetic} appear in many contexts, for example in sphere packing \cite{cohn2003new} and in ref. \cite{Kravchuk:2021akc}).  Consequently, the proofs of our theorems are completely rigorous and the computer is only used as an oracle. It is only in Section \ref{sec:selberg} that our individual estimates, which are independent from the theorems, are numerical and not proven at a mathematical level, although it should also be possible to do so.\\

\textbf{Technical note:} Unless specified otherwise, all surfaces or orbifolds considered in this paper are connected, compact and orientable.

\subsection*{Acknowledgments}
We would like to thank Anshul Adve, James Bonifacio, Vikram Giri, Petr Kravchuk, Matilde Marcolli, Dalimil Mazac, Paul Nelson, Peter Sarnak and Antoine Song for useful discussions.  SP and YX acknowledge the support of the U.S. Department of Energy, Office of Science, Office of High Energy Physics, under Award Number DE-SC0011632 and of the Walter Burke Institute for Theoretical Physics. YX
and DSD are supported by Simons Foundation grant no.\ 488657 (Simons Collaboration on the Nonperturbative Bootstrap) and a DOE Early Career Award under grant no.\ 
DE-SC0019085. The computations in this work were run on the Resnick High Performance Computing Center, a facility supported by Resnick Sustainability Institute at the California Institute of Technology.

\section{Spectral geometry on a spin-surface}
\label{sec:specgeo}
\subsection{Dirac operator and Automorphic Laplacian}
Let us consider a $d$-dimensional orientable connected Riemannian manifold $(X,g)$ and the orthonormal frame bundle over $X$, $(Q,\pi,X; SO(d))$; here $Q$ is the total space, $\pi: Q\to X$ is the projection map. We will often suppress $\pi$ and write the shorthand $Q(X)$ to mean the orthonormal frame bundle over X. 
A spin structure on this principal bundle $Q$ is a pair $(P,\Lambda)$ such that $P$ is $\text{Spin}(d)$-principal bundle and there exists a principal-bundle morphism $\Lambda:P\to Q$, that is equivariant with respect to the double covering $\text{Spin}(d)\to SO(d)$.  A spin structure on $Q$ exists if and only if the second Stiefel-Whitney class of the $SO(d)$ principal bundle vanishes.

To set the stage, we restrict our attention to upper-half plane $\mathbb{H}\coloneq \{(x,y)\in \mathbb{R}^2: y>0\}$. The metric $g$ is given by Poincare metric i.e
\begin{equation}
ds^2= \frac{dx^2+dy^2}{y^2}\,.
\end{equation}
There exists a unique spin structure on the orthonormal frame bundle over $\mathbb{H}$ (which is an $SO(2)$ principal bundle). The corresponding $\mathrm{Spin}(2)$ principal bundle is trivial and given by $P=\mathbb{H}\times \mathrm{Spin}(2)$.  If $\rho$ is the representation of $\mathrm{Spin}(2)$ on its Clifford module $\mathbb{C}^2$,  then the bundle $S$ associated with $\rho$,
\begin{equation}
S\coloneq P\times_{\!\rho}\mathbb{C}^2,
\end{equation}
is the spinor bundle over the upper half-plane. The sections of this spinor bundle are spinor fields defined on $\mathbb{H}$. In particular, we will be interested in $\mathfrak{C}^{\infty}(S)$,  the space of smooth sections of the spinor bundle $S$.  

\begin{defn}[Dirac operator on $\mathbb{H}$]
The Dirac operator $\slashed{D}:\mathfrak{C}^{\infty}(S)\to \mathfrak{C}^{\infty}(S)$ is given by\footnote{For more details on defining the Dirac operator on a $d$-dimensional oriented connected Riemannian manifold $(X,g)$, see Section $3$ of \cite{belgun2007singularity}.}
\begin{align}
    \slashed{D}\coloneq i\begin{bmatrix}0&iy\partial_x+y\partial_y-\frac{1}{2}\\-iy\partial_x+y\partial_y-\frac{1}{2}&0\end{bmatrix}.
\end{align}
\end{defn}

\begin{rem}
By restricting the domain of $\slashed{D}$ to $\mathfrak{C}_c^\infty(S)$, the space of smooth compactly supported sections of the spinor bundle $S$, one can construct a pre-Hilbert space, which can then be promoted to a Hilbert space in the standard way: first,  we define the inner product of $f,g\in\mathfrak{C}_c^\infty(S)$ by
\begin{align}
    \braket{f,g}\coloneq\int_{\mathbb{H}}\braket{f(z),g(z)}_{\mathbb{C}^2}d\mu(z),
\end{align}
where $\mu$ is the hyperbolic measure on $\mathbb{H}$ and then complete $\mathfrak{C}_c^\infty(S)$ into a Hilbert space $\mathfrak{H}(S)$ using the induced norm. The space $\mathfrak{C}_c^\infty(S)$ is dense inside $\mathfrak{H}(S)$.
\end{rem}

\begin{rem}
 The Dirac operator is elliptic and essentially self-adjoint on $\mathfrak{C}_c^\infty(S)$,    which in particular implies that its spectrum is real \cite{bolte2006selberg}.
 \end{rem}

A compact connected orientable hyperbolic surface $\Sigma$ can be thought of as $\Gamma\backslash \mathbb{H}$, where $\Gamma$ is a cocompact Fuchsian group, a discrete subgroup of $\mathrm{PSL}(2,\mathbb{R})$. 
More precisely, the cocompact Fuchsian group $\Gamma$ acts properly discontinuously on $\mathbb{H}$. If $\gamma=\pm\begin{pmatrix}a&b\\c&d\end{pmatrix}\in\Gamma$, the action is defined by
\begin{align}
\begin{array}{ccc}
\Gamma\times \mathbb{H}&\longrightarrow& \mathbb{H}\\(\gamma,z)&\longmapsto &\frac{az+b}{cz+d}.
\end{array}
\end{align}
The above induces an action on the orthonormal frame bundle $Q(\mathbb{H})\cong \mathrm{PSL}(2,\mathbb{R})$, which is simply the action by left multiplication:
\begin{align}
\begin{array}{ccc}
\Gamma\times \mathrm{PSL}(2,\mathbb{R})&\longrightarrow& \mathrm{PSL}(2,\mathbb{R})\\(\gamma,g)&\longmapsto &\gamma\cdot g.
\end{array}
\end{align}
The double cover of the bundle $Q$ is $P(\mathbb{H})\cong \mathrm{SL}(2,\mathbb{R})$.

\begin{defn}
Let $\Gamma\subset \mathrm{PSL}(2,\mathbb{R})$ be a cocompact Fuchsian group. Let $p:\mathrm{SL}(2,\mathbb{R})\rightarrow \mathrm{PSL}(2,\mathbb{R})$ be the canonical projection. We define the $\bar{\Gamma}:=p^{-1}(\Gamma)$.
\end{defn}
In order to describe spin structures more directly in terms of the action of $\Gamma$, it is convenient to apply the following result to our context:
\begin{prop}[Ref.\cite{AIF_1999__49_5_1637_0}]
\label{prop:splitseq}
Spin structures on $\Sigma$ are in one-to-one correspondence with lifts to $\mathrm{SL}(2,\mathbb{R})$ of the action by left multiplication of $\Gamma$ on $\mathrm{PSL}(2,\mathbb{R})$ on itself. In other words, the spin structures on $\Gamma \backslash \mathbb{H}$ are in one-to-one correspondence with the possible right splittings of the short exact sequence:
\begin{equation}
\label{exactseq}
\begin{tikzcd}
0\arrow[r] & \mathbb{Z}_2 \arrow[r,"\iota"] &\bar{\Gamma}\arrow[l,dashed,bend left,"\chi"]\arrow[r,"p"]&\Gamma\arrow{r}\arrow[l,dashed,bend left,"\rho"]&0.
\end{tikzcd}
\end{equation}
\end{prop}
The number of right splittings $\rho$ (i.e. $\rho$ for which $p\circ\rho=Id$) of this exact sequence is equal to the number of left splittings $\chi$, for which $\chi\circ\iota=Id$, which amounts to imposing the constraint $\chi(-Id)=-1$. A natural way to establish a one-to-one correspondence between the two is to characterize $\chi$ by $\mathrm{Ker}\,\chi=\mathrm{Im}\,\rho$. As a result, we obtain:

\begin{prop}\label{prop:multiplier}
Spin structures on a compact orientable smooth Riemann surface $\Sigma$ can be labelled by homomorphisms $\chi:\bar{\Gamma}\rightarrow\mathbb{Z}_2$ satisfying $\chi(-Id)=-1$. 
\end{prop}

\begin{rem}
In the language of automorphic forms, $\chi$ is a multiplier which allows the existence of non-trivial odd weight automorphic forms.
\end{rem}

\begin{rem}
There is also a geometric interpretation of the morphism $\chi$. If $\tilde{\gamma}\in \bar{\Gamma}$, $\chi(\tilde{\gamma})$ gives the holonomy of the spin bundle along the corresponding closed geodesic.
\end{rem}

This new characterization of spin structure in terms of a homomorphism $\chi:\bar{\Gamma}\rightarrow\mathbb{Z}_2$ such that $\chi(-Id)=-1$ has the very convenient feature of easily generalizing to the case in which the cocompact Fuchsian group $\Gamma$ has some elliptic elements, in particular, to the case of orbifolds which will be of particular interest to us. We hence define:

\begin{defn}
Let $\Gamma$ be a general cocompact Fuchsian group. A spin structure on $\Gamma\backslash \mathrm{PSL}(2,\mathbb{R})$ is a group homomorphism $\chi:\bar{\Gamma}\rightarrow\mathbb{Z}_2$ such that $\chi(-Id)=-1$. 
\end{defn}

\begin{rem}
Unlike Riemann surfaces, not every orbifold admits a spin structure. In particular, if there exists an order-2 element in $\Gamma$, injectivity of $\rho$ implies that this element must be mapped to $-\text{Id}\in\bar{\Gamma}$ which lies in the image of $\iota$ hence the kernel of $p$. This is inconsistent with the fact that $p\circ\rho=\text{Id}$, thus there does not exist any right splitting of the exact sequence \eqref{exactseq}, which, by splitting lemma, implies that there does not exist any left splitting $\chi$. Hence even order orbifold points are obstructions to the existence of a spin structure.\end{rem}

\begin{rem}
The geometric interpretation of $\chi$ naturally extends to the case of orbifolds. In the case of elliptic elements, it gives the holonomy around orbifold singularities.
\end{rem}

\begin{prop}[Ref.~\cite{geiges2012generalised}, Theorem 3]
On a smooth Riemann surface of genus $g$, there are $2^{2g}$ spin structures. An orbifold admits a spin structure if and only if all its orbifold points have odd order, and then the number of spin structures is also given by $2^{2g}$.
\end{prop}

To define the Dirac operator for $\Sigma$,  we need to restrict the domain of the $\slashed{D}$,  the Dirac operator on the upper half plane to special sets of sections (since the bundle is trivial,  they are actually  functions on $\mathbb{H}$ to $\mathbb{C}^2$) that satisfy the right automorphic properties under $\overline{\Gamma}$.   

\begin{defn}[Automorphy factor]
For $\gamma=\begin{pmatrix}
   a&b\\c&d
\end{pmatrix}\in\overline{\Gamma}\,,$ and $k\in \mathbb{N}$,  let us define autmorphy factor $j_\gamma:\mathbb{H}\times\overline{\Gamma}\to \mathbb{C}$ and $J_\gamma:\mathbb{H}\times\overline{\Gamma}\to \mathbb{C}^2$ as
\begin{equation}
\begin{split}
j_\gamma(z,k)\coloneq \frac{(cz+d)^k}{(c\bar{z}+d)^k}\,,\quad \quad J_\gamma(z,k)\coloneq\begin{pmatrix}
    j_\gamma(z,k)&0\\0&j_\gamma(z,k-1)
\end{pmatrix}\,.
\end{split}
\end{equation}
\end{defn}

\begin{defn}
Let $\Gamma$ be a cocompact Fuchsian group, and let $\chi$ be a multiplier of weight $2k\in\mathbb{N}$ on $\bar{\Gamma}$ ($\chi(
-\mathrm{Id})=(-1)^{2k}$).  We define the space $\mathfrak{F}(\Gamma,k,\chi)$ as the space of functions $\Psi:\mathbb{H}\rightarrow \mathbb{C}^2$ that satisfy
\begin{align}
\Psi(\gamma z)=\chi(\gamma)J_\gamma(z,k)\Psi(z)\,, \quad \forall\ \gamma=\begin{pmatrix}
   a&b\\c&d
\end{pmatrix}\in\overline{\Gamma}\,.
\end{align}
If these functions are asked to be smooth (resp. $L^2$ over the fundamental domain of $\Gamma$), the corresponding space is denoted $\mathfrak{C}^\infty(\Gamma,k,\chi)$ (resp. $\mathfrak{L}^2(\Gamma,k,\chi)$).  
\end{defn}

It can be checked that $\slashed{D}$ leaves the space $\mathfrak{C}^\infty(\Gamma,1/2,\chi)$ invariant and thus we have 

\begin{defn}[Dirac operator on $\Sigma$]
    The Dirac operator associated to $\Sigma$ is defined to be the restriction of $\slashed{D}$ to $\mathfrak{C}^\infty(\Gamma,1/2,\chi)$. 
\end{defn}
Under a slight abuse of notation, we will continue calling $\slashed{D}$ this Dirac operator with restricted domain. $\mathfrak{C}^\infty(\Gamma,1/2,\chi)$ is dense inside $\mathfrak{L}^2(\Gamma,1/2,\chi)$.  Compactness of the surface implies that the spectrum of $\slashed{D}$ is discrete.

\begin{defn}
Let $\Gamma$ be a cocompact Fuchsian group, and let $\chi$ be a multiplier of weight $2k$ on $\bar{\Gamma}$ ($\chi(
-\mathrm{Id})=(-1)^{2k}$).  We define the space $\mathcal{F}(\Gamma,k,\chi)$ as the space of functions $\psi:\mathbb{H}\rightarrow \mathbb{C}$ that satisfy
\begin{align}
\psi(\gamma z)=\chi(\gamma)j_\gamma(z,k)\psi(z)\,, \quad \forall\ \gamma=\begin{pmatrix}
   a&b\\c&d
\end{pmatrix}\in\overline{\Gamma}\,.
\end{align}
If these functions are asked to be smooth (resp. $L^2$ over the fundamental domain of $\Gamma$), the corresponding space is denoted $\mathcal{C}^\infty(\Gamma,k,\chi)$ (resp. $L^2(\Gamma,k,\chi)$).
\end{defn}

\begin{defn}\label{def:weightKautomorphicLaplacian}
Define the weight $2k$ automorphic Laplacian $\Delta_k:\mathcal{C}^\infty(\Gamma,k,\chi)\to \mathcal{C}^\infty(\Gamma,k,\chi)$ as
\begin{equation}
\Delta_{k}\coloneq y^2(\partial^2_x+\partial^2_y)-2 iky\partial_x\,.
\end{equation}
\end{defn}

\begin{rem}\label{remark:1-Laplacian}
The square of the Dirac operator,  $\slashed{D}^2:\mathfrak{C}^\infty(\Gamma,1/2,\chi)\to \mathfrak{C}^\infty(\Gamma,1/2,\chi)$ is given by
\begin{align}
\slashed{D}^2=\begin{pmatrix}-\Delta_{1/2}-\frac{1}{4}&0\\0&-\Delta_{-1/2}-\frac{1}{4}\end{pmatrix}\,.
\end{align}
$\Delta_{1/2}$ is also called $1$-Laplacian. 
\end{rem}

Now the following key proposition relates the spectrum of $\slashed{D}$ to that of $\Delta_{1/2}$.

\begin{prop}
If $\Psi=\begin{pmatrix}\psi_1\\\psi_2\end{pmatrix}$ is an eigenform of $-\slashed{D}$ with eigenvalue $t$, then $\psi_1$ is an eigenform of $-\Delta_{1/2}$ with eigenvalue $\lambda=1/4+t^2$. Conversely, if $\psi$ is an eigenform of $-\Delta_{1/2}$ with eigenvalue $\lambda=1/4+t^2$, then $\Psi=\begin{pmatrix}t\psi\\i\left(iy\partial_x-y\partial_y+\frac{1}{2}\right)\psi\end{pmatrix}$ is an eigenform of $-\slashed{D}$ with eigenvalue $t$.
\end{prop}

\begin{proof}
See Proposition $1$ of \cite{bolte2006selberg} and its proof.
\end{proof}

Thus bounding the Dirac spectrum associated to a given surface $\Sigma$ amounts to bounding the spectra of $\Delta_{\pm1/2}$ on $\Sigma$. Another important feature of the Dirac spectrum is that it has two important symmetries. First, the spectrum of the Dirac operator is symmetric with respect to the origin:

\begin{prop}
If $\begin{pmatrix}\psi_1\\ \psi_2\end{pmatrix}$ is an eigenvector of $\slashed{D}$ with eigenvalue $t$, then $\begin{pmatrix}\psi_1\\-\psi_2\end{pmatrix}$ is an eigenvector of $\slashed{D}$ with eigenvalue $-t$. In particular, the spectrum of $\slashed{D}$ is symmetric with respect to the origin.
\end{prop}

The Dirac operator enjoys one more symmetry leading to what is called \textit{Kramers degeneracy}. This implies that all non-zero eigenvalues of the Dirac operator are at least doubly degenerate,

\begin{prop}[Kramers degeneracy]
\label{prop:kramers}
Let $C$ be the complex conjugation operator, and $\sigma_2:=\begin{pmatrix}0&-i\\i&0\end{pmatrix}$. If $\Psi$ is an eigenvector of $\slashed{D}$ with eigenvalue $t$, then so is $(i\sigma_2 C)\Psi$.
\end{prop}

In quantum field theory, the operator implementing Kramers degeneracy is usually interpreted as the time reversal operator. 
In Section \ref{sec:step3}, we used Kramers degeneracy in our algorithm to estimate the spectra of Dirac operator on orbifolds and surfaces. More precisely, we noted that since all the non-zero eigenvalues of the Dirac operator are degenerate, we could use Proposition \ref{prop:jt2} instead of Proposition \ref{prop:jt1}, and look for the first nonzero value of $t$ for which $I_r(\lambda)=2$, instead of the first nonzero value of $t$ for which $I_r(\lambda)=1$.

\subsection{Holomorphic Modular forms and spin structure}

A spin structure determines the definition of holomorphic modular forms of odd weight on the surface. Let us recall that for even weights, holomorphic modular forms on a Riemann surface $\Sigma$ are defined independently from any choice of spin structure. More precisely: 

\begin{defn}
Let $\Gamma$ be a cocompact Fuchsian group. A holomorphic $\Gamma$-modular form of weight $2k\in2\mathbb{N}$ is a holomorphic function $f:\mathbb{H}\longrightarrow\mathbb{C}$,
such that for all $\gamma:z\mapsto\frac{az+b}{cz+d}$
in $\Gamma$,
\begin{align}
    f(\gamma(z))=(cz+d)^{2k}f(z),
\end{align}
\end{defn}
Note that since $2k$ is an even number, $(cz+d)^{2k}$ has the same value on the two possible lifts of $\gamma$ inside $\mathrm{SL}(2,\mathbb{R})$, which ensures that the above definition is unambiguous. This changes for odd powers of $cz+d$, because the automorphy factor differs by a sign depending on the choice of lift. Instead, one must use the spin structure $\chi$ as an input to define a modular form of odd weight:

\begin{defn}
Let $\chi:\bar{\Gamma}\rightarrow\mathbb{Z}_2$ be a group homomorphism (or equivalently, a choice of spin structure on $\Sigma$). A holomorphic modular form of weight $k$ for $\bar{\Gamma}$ is a holomorphic function $f:\mathbb{H}\longrightarrow\mathbb{C}$, such that for all $\gamma:z\mapsto\frac{az+b}{cz+d}$ in $\Gamma$,
\begin{align}
    f(\gamma(z))=\chi(\gamma)^k(cz+d)^kf(z).
\end{align}
\end{defn}

In the case of $k$ even, this definition reduces to the previous one and does not depend on the choice of $\chi$, but this is no longer true for $k$ odd.

Our approach will use the number of holomorphic modular forms of a given (even or odd) weight as an input to the set of consistency conditions i.e \textit{bootstrap} constraints. Therefore, it is important to understand what moduli space of surfaces and orbifolds supports a given number of holomorphic modular forms of a given (even or odd) weight. Apart from the case of modular forms of weight 1, this is determined by purely topological data, as shown by the Riemann--Roch theorem, which we now review.

The Riemann--Roch theorem comes from algebraic geometry, and allows to count the number of holomorphic modular forms of weight $\geqslant 2$ on $\Sigma$ only from the information about its topology. 


\begin{thm}[Riemann--Roch for orbifolds]
    Let $n\geqslant 2$ and $n\in\mathbb{Z}$. If an orbifold of genus $g$ with elliptic points $p_1,\dots,p_r$ of orders $k_1,\dots,k_r$ admits a spin structure, then it possesses \begin{align}\ell_{n/2}=(n-1)(g-1)+\sum_{i=1}^r\left\lfloor\frac{n}{2}\left(1-\frac{1}{k_i}\right)\right\rfloor+\delta_{n,2}\end{align}holomorphic modular forms of weight $n$ (regardless of the choice of multiplier), and $g$ holomorphic modular forms of weight 2.
\end{thm}

\begin{rem}
Note that when $n$ is odd, there can be cases in which there is no spin structure on the space under consideration (if it is an orbifold). The Riemann--Roch theorem requires the existence of a spin structure for odd $n$.
\end{rem}

\begin{rem}
The Riemann--Roch theorem will be the central tool that we will use in order to apply our bootstrap techniques. The number of modular forms of a certain weight will be used as an input in our symmetry constraints, and the Riemann--Roch theorem will enable us to characterize the class of $\Sigma$ supporting enough modular forms in order for our constraints to apply.
\end{rem}

\subsection{Harmonic spinors}

Calculating the number of modular forms of a given weight is straightforward thanks to the Riemann--Roch formula for $n\geqslant 2$: it is a topological invariant of the surfaces or orbifolds under consideration. However, the case $n=1$ is particularly tricky. In particular, for compact orientable smooth surfaces of genus $2$ or more, this number is not topologically invariant, and actually carries some nontrivial information about the surface's geometry and the spin structure on it.

\begin{rem}
The fact that the $n=1$ case behaves differently is most easily understandable in terms of sheaf cohomology. In this language, the Riemann--Roch formula involves the dimension of the zeroth and first cohomology groups of the sheaf under consideration, which is a power of the canonical line bundle determined by $n$ (see for example Theorem 16.9 of \cite{Forster_Gilligan_1999}). It is the dimension of the zeroth cohomology group that gives the number of holomorphic modular forms of weight $n$. If $n>1$, the first sheaf cohomology group automatically vanishes, which allows to directly compute the number of holomorphic modular forms from the Riemann--Roch formula. It is no longer right to ignore the term coming from the first sheaf cohomology group for $n=1$, which explains why the situation is more complicated.
\end{rem}

What remains true in the $n=1$ case is that the genus gives an upper bound on the number of harmonic spinors:
\begin{prop}[Ref.~\cite{HITCHIN19741}]
The number $\ell_{1/2}$ of harmonic spinors on a compact orientable smooth Riemann surface of genus $g$ with a spin structure satisfies
\begin{align}
\label{eq:harmonicspinorbound}
    \ell_{1/2}\leqslant\left\lfloor\frac{g+1}{2}\right\rfloor.
\end{align}
\end{prop}

The natural next step is to ask when this upper bound is saturated. Here the answer is interesting, and requires the following definition:

\begin{defn}
Let $\Sigma$ be a (smooth) Riemann surface. $\Sigma$ is \textit{hyperelliptic} if it can be realized as a branched double cover $\Sigma\rightarrow \mathbb{C}P^1$.
\end{defn}

\begin{rem}
Hyperelliptic surfaces enjoy particularly nice properties, for example, they can always be described by an algebraic equation of the form 
\begin{align}
y^2=F(x),
\end{align}
where $F$ is a polynomial with distinct roots. At genus $1$ and $2$, every surface is hyperelliptic. However, at genus higher than or equal to $3$, the moduli space of hyperelliptic surfaces is not the entire moduli space and hence is interesting to study.
\end{rem}

For our purposes, the following result is relevant:

\begin{prop}[Ref.~\cite{Martens1968}]
Given $\Sigma$, a smooth surface of genus different from $4$ or $6$, there exists a spin structure on $\Sigma$ for which the bound \eqref{eq:harmonicspinorbound} is saturated if and only if $\Sigma$ is hyperelliptic.
\end{prop}

This shows that apart from the case of genus $4$ and the case of genus $6$,\footnote{It is intriguing to note that these genera are precisely the ones for which bootstrap bounds\cite{Kravchuk:2021akc} for Laplace spectra are weaker than the Yang--Yau bounds \cite{ASNSP_1980_4_7_1_55_0}.} we can use the existence of a maximal number of harmonic spinors to write down the set of consistency conditions, leading to bounds on Laplacian spectra for a genus $g$ hyperelliptic surface. These bounds are refined (and consequently, as we will see later, they are stronger as well), in contrast with the bounds obtained in \cite{Kravchuk:2021akc}, which are valid for all surfaces of genus $g$.


The upper bound on the gap of the Laplacian on some appropriately chosen moduli space of surfaces is likely to be saturated on surfaces with a large automorphism group (i.e. a locally maximal automorphism group). The classification of hyperelliptic surfaces with large automorphism groups is achieved in Table 1 of \cite{https://doi.org/10.48550/arxiv.1711.06599}. More generally, it is interesting to ask for a geometric description of the space of surfaces of a given genus that can carry a given number of harmonic spinors. As we saw, all surfaces of genus $1$ and $2$ are hyperelliptic so such an endeavor is trivial there, but from genus $3$ onwards, the answer starts to give nontrivial information about the geometry of the surfaces. The classification of Riemann surfaces in terms of the possible dimensions of their spaces of harmonic spinors is still open in general, however, for genus $3\leq g\leq 6$, precise results are available. In Appendix \ref{app:harmonicspinors}, we summarize these explicit classifications in low genus. The results collected in this appendix are condensed in Table \ref{tab:mostsymmetric}.

\begin{table}[!ht]
    \centering
    \begingroup
    \renewcommand{\arraystretch}{2.2}
    \scalebox{.8}{\begin{tabular}{ccccc}
    \toprule
$g$  & $\ell_{1/2}^{Max}$ & Moduli space & Most symmetric surface & Automorphism group \\
\midrule
2 & 1 & All & Bolza surface &$GL(2,\mathbb{F}_3)$\\
\cmidrule(lr){1-5}
\multirow{2}{*}{3} & 1 &  Non-hyperelliptic & Klein quartic & $\mathrm{PSL}(2,\mathbb{F}_7)$\\
  & 2 & Hyperelliptic & $y^2=x^8+14x^4+1$ & $\mathbb{Z}_2\times S_4$\\
\cmidrule(lr){1-5}
\multirow{2}{*}{4}  & 1 &
        \pbox{15cm}{Non-hyperelliptic with\\ smooth canonical quadric} & Bring surface & $S_5$\\[+3pt]& 2 & \pbox{15cm}{Hyperelliptic or non-hyperelliptic \\ with singular canonical quadric}& $z^3y^2=x(x^4+y^4)$ & Order $72$\\
\cmidrule(lr){1-5}
\multirow{2}{*}{5}  & 1 or 2 & Non-hyperelliptic & $x_1^2+x_2^2+2x_4^2=x_4^2(x_2^2+x_4^2)-x_3^4=0$ & Order 192\\&3& Hyperelliptic & $y^2=x(x^{10}+11x^5-1)$& $\mathbb{Z}_2\times A_5$\\
\cmidrule(lr){1-5}
\multirow{2}{*}{6}& 1 or 2 & \pbox{15cm}{Non-hyperelliptic and \\ not a smooth quintic} & Wiman sextic & $S_5$\\[+3pt]& 3 & \pbox{15cm}{Hyperelliptic or\\smooth quintic} & Fermat quintic & Order $150$\\
\bottomrule
\end{tabular}}\endgroup
\caption{Table of the moduli spaces of surfaces, satisfying each genus/spin constraint, and of the most symmetric surfaces in each of these moduli spaces. Here $g$ is the genus and $\ell_{1/2}^{Max}$ is the maximal number of harmonic spinors that a surface can carry.}
    \label{tab:mostsymmetric}
\end{table}

\section{Spectral identities for hyperbolic surfaces from associativity}
\label{sec:constraints}
The aim of this section is to derive the spectral identities which we will eventually \textit{bootstrap} to bound the Laplacian and Dirac spectra on a compact connected orientable hyperbolic surfaces and orbifolds. To set the stage, let us consider the upper half plane $\mathbb{H}=\{(x,y):x,y\in\mathbb{R}, y>0\}$, equipped with the usual hyperbolic metric with sectional curvature $-1$ i.e $ds^2=y^{-2}(dx^2+dy^2)$. The orientation preserving isometry group of $\mathbb{H}$ is $\mathrm{PSL}(2,\mathbb{R})$. Given $z\in\mathbb{H}$, $\mathrm{PSL}(2,\mathbb{R})$ acts as follows:
\begin{equation}
    z \mapsto \frac{a z+b}{cz+d}\,, \quad \text{for all}\quad \pm \begin{pmatrix}
        a & b\\
        c & d
    \end{pmatrix}\in \mathrm{PSL}(2,\mathbb{R})\,,\quad a,b,c,d\in \mathbb{R}\,,\quad ad-bc=1\,.
\end{equation}
The upper half plane $\mathbb{H}$ can be identified with $\mathrm{PSL}(2,\mathbb{R})/SO(2)$. The double cover of $\mathrm{PSL}(2,\mathbb{R})$ is $\mathrm{SL}(2,\mathbb{R})$. There is a natural action of $\mathrm{SL}(2,\mathbb{R})$ on $z\in\mathbb{H}$ given by
\begin{equation}
  z\mapsto \frac{a z+b}{cz+d}\,, \quad \text{for all}\quad \begin{pmatrix}
        a & b\\
        c & d
    \end{pmatrix}\in \mathrm{SL}(2,\mathbb{R})\,,\quad a,b,c,d\in \mathbb{R}\,,\quad ad-bc=1\,.
\end{equation}
Note that the center of $\mathrm{SL}(2,\mathbb{R})$ acts trivially on $z\in\mathbb{H}$. 

A compact connected orientable hyperbolic orbifold\footnote{We are following the convention that the manifolds are orbifolds with no orbifold-singularity unless otherwise mentioned.} can be thought of as $\Gamma\backslash (\mathrm{PSL}(2,\mathbb{R})/SO(2))$, where $\Gamma$ is a cocompact Fuchsian group. For a spin orbifold with a given spin structure, there exists $\bar{\Gamma}$, a subgroup of $\mathrm{SL}(2,\mathbb{R})$, such that $\Gamma\simeq \bar{\Gamma}/\mathbb{Z}_2$. This means that $\Gamma$ can be consistently embedded inside $\bar\Gamma$ such that the embedding does not have $-Id$. We will denote this embedding as $\widetilde{\Gamma}$. As a set, we have 
\begin{equation}\label{split}
    \bar\Gamma= \widetilde{\Gamma} \sqcup \left(-\widetilde{\Gamma}\right)\,,
\end{equation}
and $\widetilde{\Gamma}$ is isomorphic to $\Gamma$. Proposition \ref{prop:multiplier} implies that $\chi(\widetilde{\gamma})=\pm1$ for 
$\widetilde{\gamma}\in \pm\widetilde{\Gamma}$ respectively. In other words, different spin structures correspond to different homomorphisms $\chi:\bar{\Gamma}\to \mathbb{Z}_2$, 
which amounts to saying that \eqref{split}, and in particular $\widetilde{\Gamma}$, depends on the spin structure. 
We stress that even though $\Gamma$ is isomorphic to $\widetilde{\Gamma}$, by choosing the embedding $\widetilde{\Gamma}$, we have already committed to a spin structure. Thus it is important to distinguish between $\Gamma$ and $\widetilde{\Gamma}$ even though they are isomorphic.

The upshot of the above is that a compact connected orientable hyperbolic spin orbifold $X$, with a given spin structure can be thought of as $\widetilde{\Gamma}\backslash (\mathrm{SL}(2,\mathbb{R})/\text{Spin}(2))$. The total space for the spin bundle, i.e. $P(X)$, is given by $\widetilde{\Gamma}\backslash \mathrm{SL}(2,\mathbb{R})$.
We formulate the spectral problem for the Laplacian and Dirac operator on these hyperbolic spin orbifolds with a spin structure using the harmonic analysis on $L^2(\widetilde{\Gamma}\backslash G)$, where $G:=\mathrm{SL}(2,\mathbb{R})$. Evidently, a key role is played by irreducible unitary representations of $\mathrm{SL}(2,\mathbb{R})$.

\subsection{Representation theory of $\mathrm{SL}(2,\mathbb{R})$}

\begin{theorem}[See~\cite{Knapp}]\label{thm:SL2Rirreps}
	The unitary irreducible representations of $\mathrm{SL}(2,\R)$ are given up to equivalence by:
\begin{enumerate}
	\item the trivial representation,
 \item the complementary series $\cC_s$ for $s\in (0,\thalf)$,
 \item the principal series $\cP^\pm_{i\nu}$ for $\nu\in \R\setminus\{0\}$,
	\item the holomorphic discrete series $\cD_{n/2}$ and anti-holomorphic discrete series $\bar\cD_{n/2}$, $n\geq 2,\, n\in \Z$,
	\item the limits of discrete series $\cD_{\half}$, $\bar\cD_{\half}$.
\end{enumerate}
	The only equivalence between the representations listed above is the following: $\cP^\pm_{i\nu}\simeq \cP^\pm_{-i\nu}$. 
\end{theorem}

\begin{prop}
    The complexified Lie algebra $\mathfrak{sl}_2(\mathbb{C})$ is generated by $L_0,L_{\pm}$ and we have $[L_m,L_n]=(m-n)L_{m+n}$. For unitary representations, we have $L_n^\dagger=L_{-n}$. 
\end{prop}

\begin{prop}
    The quadratic Casimir $C_2$ takes the form $$C_2:=L_0^2-\frac{1}{2}\left(L_{-1}L_1+L_1L_{-1}\right)\,.$$
\end{prop}

To explicitly construct the irreps, we find it convenient to recall that $\mathrm{SL}(2,\mathbb{R})$ can be mapped to $SU(1,1)$ by conjugation within $\mathrm{SL}(2,\mathbb{C})$. If $g=\begin{pmatrix}
        a & b\\
        c & d
    \end{pmatrix}\in \mathrm{SL}(2,\mathbb{R})$, it gets mapped to $u\in SU(1,1)$; explicitly we have
\begin{equation}
    u=\begin{pmatrix}
        1 & i\\
        i & 1
    \end{pmatrix}^{-1} \begin{pmatrix}
        a & b\\
        c & d
    \end{pmatrix}\begin{pmatrix}
        1 & i\\
        i & 1
    \end{pmatrix}\in SU(1,1)\quad \text{if}\quad \begin{pmatrix}
        a & b\\
        c & d
    \end{pmatrix}\in \mathrm{SL}(2,\mathbb{R})\,.
\end{equation}
Under this map, the upper half-plane gets mapped to the unit disk $\mathbb{D}=\{z:|z|<1\}$. In what follows, we denote the elements of $SU(1,1)$ by $u$, where
\begin{equation}
    u:=\begin{pmatrix}
        \alpha & \beta\\
        \bar\beta &\bar\alpha 
    \end{pmatrix}\,,\quad \alpha\bar\alpha-\beta\bar\beta=1\,,\quad \quad \alpha,\beta\in\mathbb{C}\,.
\end{equation}

Then, $SU(1,1)$ acts on $z\in\mathbb{C}$ as
\begin{equation}
    u\cdot z:= \frac{\alpha z+\beta}{\bar\beta z+\bar\alpha}\,,
\end{equation}

\begin{rem}
The irreps of $\mathrm{PSL}(2,\mathbb{R})$ are already  reviewed in detail in \cite{Kravchuk:2021akc}.  For $n\in\mathbb{Z}$, $\cD_n$ can be realized as antiholomorphic functions living on the unit disk $\mathbb{D}=\{z:|z|<1\}$ or holomorphic functions on $\mathbb{D}':=\{z: |z|>1\} \cup \{\infty\}$. $\bar\cD_n$ can be realized as holomorphic functions living on the unit disk $\mathbb{D}$. The complementary series $\mathcal{C}_s$ and the principal series $P^{+}_{i\nu}$ can be realized as functions living on $\partial\mathbb{D}=\{z:|z|=1\}$. All these functions belong to an $L^2$ space with respect to an appropriate $G$-invariant measure. The new ingredients that we are going to use are $\cP^-_{i\nu}$ and $\cD_{n}$, $\bar\cD_{n}$ with $n\in 1/2+\mathbb{Z}$. Hence, we will be brief in explaining the construction of irreps, often referring the readers to section $3$ of \cite{Kravchuk:2021akc} for details, and only highlight the parts that are new to this paper.
\end{rem}


\begin{defn}[Antiholomorphic discrete series $\bar\cD_{n/2}$]
  The explicit realization of antiholomorphic discrete series $\bar\cD_{n/2}$ ($n\in\mathbb{Z}_+$) is achieved by holomorphic functions $f(z)$ on $\mathbb{D}$ such that for all $u\in SU(1,1)$, we have
\begin{equation}\label{eq:Anti}
    u\cdot f(z) =  \left(-\bar \beta z+ \alpha \right)^{-n}f\left(u^{-1}\cdot z\right)\,,
\end{equation}
In particular, the center of $\mathrm{SL}(2,\mathbb{R})$, given by $\alpha=-1$, $\beta=0$, acts non-trivially iff $n$ is odd. The norm of a vector $f$ inside $\bar\cD_{n/2}$ is given by
\begin{equation}
\begin{split}\label{eq:norm}
       ||f||_{\bar\cD_{1/2}}=&\underset{0\leq r<1}{\text{sup}}\int_0^{2\pi}d\theta\ |f(re^{i\theta})|^2\,,\\
      ||f||_{\bar\cD_{n/2}}= &\int_{\mathbb{D}}dz\ (1-|z|^2)^{n-1} |f(z)|^2\,,\quad n>1\,.
\end{split}
\end{equation}  
\end{defn}

\begin{rem}
       We note that $$\bar\cD_{n/2}\simeq\underset{k\in -n/2-\mathbb{Z}_{\geq 0}}{\bigoplus} V_k\,,$$ where the $V_k$ are the one-dimensional irreducible representations of $\text{Spin}(2)$. This amounts to saying that the $L_0$ spectrum of such an irrep is given by $\{-n/2-k:k\in\mathbb{Z}_{\geq 0}\}$. The quadratic Casimir coorsponding to the irrep $\bar\cD_{n/2}$ is $-n/2(1+n/2)$.
\end{rem}

\begin{defn}[Holomorphic discrete series $\cD_{n/2}$]
The holomorphic discrete series $\cD_{n/2}$ with $n\in\mathbb{Z}_+$ is realized by the space of antiholormorphic functions $f$ on $\mathbb{D}$ such that for all $u\in SU(1,1)$, we have
\begin{equation}\label{eq:Anti}
    u\cdot f(\bar z) =  \left(-\beta \bar z+ \bar \alpha \right)^{-n}f\left(u^{-1}\cdot \bar z\right)\,,
\end{equation}
In particular, the center of $\mathrm{SL}(2,\mathbb{R})$ given by $\alpha=-1$, $\beta=0$ acts non-trivially iff $n$ is odd. 
The norm of a vector $f$ inside $\cD_{n/2}$ is given by 
\begin{equation}
\begin{split}\label{eq:norm}
       ||f||_{0,\cD_{1/2}}=&\underset{0\leq r<1}{\text{sup}}\int_0^{2\pi}d\theta\ |f(re^{i\theta})|^2\,,\\
      ||f||_{0,\cD_{n/2}}= &\int_{\mathbb{D}}dz\ (1-|z|^2)^{n-1} |f(z)|^2\,,\quad n>1\,.
\end{split}
\end{equation}

However as explicitly worked out in \cite{Kravchuk:2021akc}, it is easier to consider another realization by holomorphic functions $F$ on $\mathbb{D}'=\{z:|z|>1\} \cup \{\infty\}$, where 
\begin{equation}
    F(z):= z^{-n} f(\bar{z}^{-1})\,,
 \end{equation}
and the norm of $F$ is given by
\begin{equation}
    ||F||_{\cD_{n/2}}=||f||_{0,\cD_{n/2}}.
\end{equation}
The functions $F$ transform exactly in the same way as given by eq.~\!\eqref{eq:Anti}. Again the action of the center is nontrivial iff $n$ is odd.
\end{defn}

\begin{rem}
     We note that $$\cD_{n/2}\simeq\underset{k\in n/2+\mathbb{Z}_{\geq 0}}{\bigoplus} V_k\,,$$ where the $V_k$ are the one-dimensional irreducible representations of $\text{Spin}(2)$. This amounts to saying that the $L_0$ spectrum of such an irrep is given by $\{n/2+k:k\in\mathbb{Z}_{\geq 0}\}$. The quadratic Casimir corresponding to the irrep $\cD_{n/2}$ is $n/2(1-n/2)$.
\end{rem}

\begin{defn}
[Principal series $\mathcal{P}^{+}_{i\nu}$] They are realized in the space of equivalence classes of square integrable functions on $\partial\mathbb{D}=\{z:|z|=1\}$ such that for all $u\in SU(1,1)$ we have
\begin{equation}
    (u\cdot f)(z)= |-\beta\bar z+\bar\alpha |^{-1-2 i\nu} f(u^{-1}\cdot z).
\end{equation}
In particular, the center of $\mathrm{SL}(2,\mathbb{R})$, given by $\alpha=-1$, $\beta=0$, acts trivially. The norm of $f$ is given by 
\begin{equation}
    ||f||_{\mathcal{P}^{+}_{i\nu}} =\int_{0}^{2\pi}d\theta\  |f(e^{i\theta})|^2.
\end{equation}
    
\end{defn}

\begin{rem}
We have
$$\mathcal{P}^{+}_{i\nu}\simeq\underset{k\in\mathbb{Z}}{\bigoplus}\ V_k\,,$$
where the $V_k$ are the one-dimensional irreducible representations of $\text{Spin}(2)$. The spectrum of $L_0$ for vectors inside $\mathcal{P}^{+}_{i\nu}$ consists of all integers and each integer appear exactly once. A basis for functions living inside $\mathcal{P}^{+}_{i\nu}$ is given by $f_j(z)=z^{j}$ with $j\in \mathbb{Z}$. The quadratic Casimir corresponding to the irrep $\mathcal{P}^{+}_{i\nu}$ is $1/4+\nu^2$.
\end{rem}

\begin{defn}[Principal series $\mathcal{P}^{-}_{i\nu}$]
    They are realized in the space of  equivalence classes of square integrable functions $f$ on $\partial\mathbb{D}$ such that for all $u\in SU(1,1)$ we have
\begin{equation}
    (u\cdot f)(z)= \text{sgn}\left(-\bar\beta z+\alpha\right)|-\beta\bar z+\bar\alpha |^{-1-2 i\nu} f(u^{-1}\cdot z).
\end{equation}
In particular, the center of $\mathrm{SL}(2,\mathbb{R})$, given by $\alpha=-1$, $\beta=0$, acts non-trivially and gives a minus sign. The norm of $f$ is given by ($z=e^{i\theta}$)
\begin{equation}
    ||f||_{\mathcal{P}^{-}_{i\nu}} =\int_{0}^{2\pi}d\theta\  |f(e^{i\theta})|^2.
\end{equation}
\end{defn}

\begin{rem}
We have
$$\mathcal{P}^{-}_{i\nu}\simeq \underset{k\in1/2+\mathbb{Z}}{\bigoplus} V_k\,,$$
where $V_k$ are the one-dimensional irreducible representations of $\text{Spin}(2)$.
    The spectrum of $L_0$ for vectors inside $\mathcal{P}^{-}_{i\nu}$ consists of all half-integers, and each half-integer appears exactly once. A basis for functions inside $\mathcal{P}^{-}_{i\nu}$ is given by $f_j(z)=z^{j}$ with $j\in 1/2+\mathbb{Z}$. Strictly speaking, they are actually functions on the double cover of $\partial\mathbb{D}$. The quadratic Casimir corresponding to the irrep $\mathcal{P}^{-}_{i\nu}$ is $1/4+\nu^2$. 
\end{rem}

\begin{defn}[Complementary series $\mathcal{C}_{s}$]
    They are realized in the space of  equivalence classes of square integrable functions on $\partial\mathbb{D}$ such that for all $u\in SU(1,1)$ we have
\begin{equation}
    (u\cdot f)(z)= |-\beta\bar z+\bar\alpha |^{-1-2 i\nu} f(u^{-1}\cdot z).
\end{equation}
In particular, the center of $\mathrm{SL}(2,\mathbb{R})$, given by $\alpha=-1$, $\beta=0$, acts trivially. The norm of $f$ is given by ($z=e^{i\theta}$)
\begin{equation}
    ||f||_{\mathcal{C}_{s}} =\int_{0}^{2\pi}\ d\theta\ \int_{0}^{2\pi}\ d\phi\  \frac{|f(e^{i\theta})|^2}{|e^{i\theta}-e^{i\phi}|^{2-2\Delta}}\,,\quad \Delta=1/2+s.
\end{equation}
\end{defn}

\begin{rem}
We have
$$\mathcal{C}_{s}\simeq\underset{k\in\mathbb{Z}}{\bigoplus}\ V_k\,,$$
where $V_k$ are the one-dimensional irreducible representations of $\text{Spin}(2)$. The spectrum of $L_0$ for vectors inside $\mathcal{C}_{s}$ consists of all integers and each integer appears exactly once. The quadratic Casimir corresponding to the irrep $\mathcal{C}_{s}$ is $1/4-s^2$.
\end{rem}

\subsection{Coherent States}
We now construct the coherent states, these are functions living inside the irrep $\cD_{n/2}$ or $\bar\cD_{n/2}$. The construction of these states mimics the one presented in \cite{Kravchuk:2021akc}, extended in the present paper to the case when $n$ is an odd integer.

\begin{defn}
    The coherent state $\mathcal{O}_{n/2}(z)$ for $z\in\mathbb{D}$ is identified with a holomorphic function $f$ on $\mathbb{D}'$, i.e.
\begin{equation}
\begin{split}
    \mathcal{O}_{1/2}(z)(w)&:=\sqrt{\frac{1}{2\pi}}(z-w)^{-1}\,,\\
    \mathcal{O}_{n/2}(z)(w)&=\sqrt{\frac{n-1}{\pi}}(z-w)^{-n}\,,\quad n>1\,,
    \end{split}
\end{equation}
where $w\in\mathbb{D}'$.
\end{defn}

    \begin{defn}
  The coherent state $\widetilde{\mathcal{O}}_{n/2}(z)$ for $z\in\mathbb{D}'$ is identified with a holomorphic function $f$ on $\mathbb{D}$. Explicitly we have, for $z\in\mathbb{D}'$ and $w\in\mathbb{D}$,
\begin{equation}
\begin{split}
    \widetilde{\mathcal{O}}_{1/2}(z)(w)&:=\sqrt{\frac{1}{2\pi}}(z-w)^{-1}\,,\\
   \widetilde{ \mathcal{O}}_{n/2}(z)(w)&=\sqrt{\frac{n-1}{\pi}}(z-w)^{-n}\,,\quad n>1\,.
    \end{split}
\end{equation}
\end{defn}

\begin{prop}\label{innerproductofcoherentstates}
    The inner products of coherent states are given by 
    \begin{equation}
    \begin{split}
    &\left( \mathcal{O}_{n/2}(z_1),\mathcal{O}_{n/2}(z_2)\right)_{\cD_{n/2}}=\left(1-\bar{z}_1z_2\right)^{-n}\,,\\
   & \left( \widetilde{\mathcal{O}}_{n/2}(z_1),\widetilde{\mathcal{O}}_{n/2}(z_2)\right)_{\bar\cD_{n/2}}=\left(\bar{z}_1z_2-1\right)^{-n}\,.
    \end{split}
\end{equation}
\end{prop}

\begin{proof}
    See \cite{Kravchuk:2021akc} for the case when $n$ is an even integer. The extended version for $n\in\mathbb{Z}_+$ follows from the definition of the inner product for $\cD_{n/2}$ and $\bar\cD_{n/2}$,
\end{proof}

For a unitary representation $R$ of $G$ we write $R^\infty$ for its Fréchet space of smooth vectors. We borrow the following propositions from \cite{Kravchuk:2021akc}, which will be useful for our purposes:
\begin{prop}[Proposition $3.3$ \cite{Kravchuk:2021akc}]
    The coherent states $\mathcal{O}$ and $\widetilde{\mathcal{O}}$ take values in $\cD_{n/2}^\infty$ and $\bar\cD_{n/2}^{\infty}$ respectively. They are holomorphic as functions $\mathbb{D}\to \cD_{n/2}^\infty$ and $\mathbb{D}'\to\bar\cD_{n/2}^\infty$. The span of the coherent states $\mathcal{O}(z)\in \cD_{n/2}$ for $z\in\mathbb{D}$ is dense in $\cD_{n/2}$. The span of the coherent states $\widetilde{\mathcal{O}}(z)\in \bar\cD_{n/2}$ for $z\in\mathbb{D}'$ is dense in $\bar \cD_{n/2}$. 
\end{prop}

\begin{prop}[Propostion $3.4$ \cite{Kravchuk:2021akc}]
   We have $L^2(\widetilde{\Gamma}\backslash G)^{\infty}=C^{\infty}(\widetilde{\Gamma}\backslash G)$.
\end{prop}

Finally we define coherent states inside continuous series irreps.
\begin{defn}[See \cite{Kravchuk:2021akc}]
    The coherent states inside $R_k$ where $R_k=\mathcal{P}_{i\nu}^{+}$ with $\lambda_k^{(0)}=1/4+\nu^2$ or $R_k=\mathcal{C}_s$ with $\lambda_k^{(0)}=1/4-s^2$ are defined as $R_k^{\infty}$-valued distributions on $C^\infty(\partial\mathbb{D})$ by
    \begin{equation}
        \cO(f):= N^{(0)}_{\Delta_k} f\,,\quad f\in C^\infty(\partial\mathbb{D}),
    \end{equation}
    where $N^{(0)}_{\Delta_k}>0$ is chosen to ensure $N_{\Delta_k}||1||_{R_k}=1$.
\end{defn}

\begin{defn}
    The coherent states inside $\mathcal{P}_{i\nu}^{-}$ with $\lambda_k^{(1/2)}=1/4+\nu^2$ are defined as $R_k^{\infty}$-valued distributions on $C^\infty(\partial\mathbb{D})$ by
    \begin{equation}
        \cO(f)= N^{(1/2)}_{\Delta_k} f\,,\quad f\in C^\infty(\partial\mathbb{D}),
    \end{equation}
    where $N^{(1/2)}_{\Delta_k}>0$ is chosen to ensure $N^{(1/2)}_{\Delta_k}||z^{1/2}||_{R_k}=1$.
\end{defn}

\subsection{Spectrum of $L^2(\widetilde{\Gamma}\backslash G)$}

The space $\widetilde{\Gamma}\backslash G$ is the spin bundle over the compact orbifold $X$. We would like to study the space $L^2(\widetilde{\Gamma} \backslash G)$, consisting of equivalence class of square integrable functions $F:G\to\mathbb{C}$ such that $F(\widetilde{\gamma}g)=F(g)$ for all $\widetilde{\gamma}\in\widetilde{\Gamma}$ and $g\in G$. We normalize the Haar measure on $G$ in a way such that $\mu(\widetilde{\Gamma}\backslash G)=1$. 
Subsequently, we define the inner product as
\begin{equation}
    (F_1,F_2)=\int_{\widetilde{\Gamma}\backslash G} dg\ F_1^*(g) F_2(g)\,.
\end{equation}
The inner product induces the following norm on $F\in L^2(\widetilde{\Gamma}\backslash G)$:
\begin{equation}
    ||F||^2:=(F,F)\,.
\end{equation}

We can turn $L^2(\widetilde{\Gamma} \backslash G)$ into a representation of $\mathrm{SL}(2,\mathbb{R})$ by defining the following $G$-action: $\tilde{g}\in G$ acts on elements of $L^2(\widetilde{\Gamma}\backslash G)$ as
\begin{equation}\label{groupTr}
    \left[\tilde{g}\cdot F\right](g):=F(g\tilde{g})\,.
\end{equation}
It is easy to verify that the norm of $F$, $||F||\equiv (F, F)$ is $G$ invariant and thus $L^2(\widetilde{\Gamma} \backslash G)$ indeed becomes a representation of $G$. Recalling the $NAK$ decomposition of $\mathrm{SL}(2,\mathbb{R})$:
$$g(x,y,\theta)=\begin{pmatrix}
    1 & x\\
    0 & 1
\end{pmatrix}
\begin{pmatrix}
    \sqrt{y} & 0\\
    0 & \frac{1}{\sqrt{y}}
\end{pmatrix}
\begin{pmatrix}
    \cos\frac{\theta}{2} & -\sin\frac{\theta}{2}\\
    \sin\frac{\theta}{2} & \cos\frac{\theta}{2}
\end{pmatrix}\,,$$
we will often write the functions $F(g)$ as $F(x,y,\theta)$.

\begin{prop}
    The elements of the complexified Lie algebra corresponding to $\mathrm{SL}(2,\mathbb{R})$ act on $F\in L^2(\widetilde{\Gamma}\backslash G)$ as follows:
    \begin{equation}
        \begin{split}
            L_0 F(x,y,\theta)&= i\partial_\theta F(x,y,\theta)\,,\\
            L_1F(x,y,\theta)&=e^{-i\theta}\left[y(\partial_x+i\partial_y)+\partial_\theta\right]F(x,y,\theta)\,,\\
            L_{-1}F(x,y,\theta)&=-e^{i\theta}\left[y(\partial_x-i\partial_y)+\partial_\theta\right] F(x,y,\theta).
        \end{split}
    \end{equation}
    The quadratic Casimir $C_2:=L_0^2-\frac{1}{2}\left(L_{-1}L_1+L_1L_{-1}\right)$ acts as
    \begin{equation}
C_2 F(x,y,\theta)=\left[y^2(\partial_x^2+\partial_y^2)-2y\partial_x\partial_\theta\right]F(x,y,\theta)\,.
    \end{equation}
\end{prop}
\begin{proof}
    The above follows the transformation law, given by eq.~\!\eqref{groupTr} and $NAK$ decomposition of $\mathrm{SL}(2,\mathbb{R})$.
\end{proof}

\begin{prop}\label{decomp}
For a cocompact $\widetilde{\Gamma}\subset \mathrm{SL}(2,\mathbb{R})$, such that $-Id\notin\widetilde{\Gamma}$, we have a discrete decomposition of $L^2(\widetilde{\Gamma} \backslash G)$ in terms of irreducible representations. In particular we have
\begin{equation}
    L^2(\widetilde{\Gamma} \backslash G)=\mathbb{C} \oplus \bigoplus_{n=1}^{\infty}\left(\mathscr{D}_{n/2} \oplus \overline{\mathscr{D}}_{n/2}\right) \oplus \bigoplus_{k=1}^{\infty} \mathscr{C}_{\lambda^{(0)}_k} \oplus \bigoplus_{k=1}^{\infty} \mathscr{P}^{-}_{\lambda^{(1/2)}_k}\,.
\end{equation}
\end{prop}

\begin{rem}
    The group action is transitive, hence the trivial irrep $\mathbb{C}$ appears exactly once. This corresponds to the constant functions on $X$.
\end{rem}

\begin{rem}\label{rem:coherent}
    $\mathscr{D}_{n/2}$ is unitarily isomorphic to $\mathbb{C}^{\ell_{n/2}}\otimes\cD_{n/2}$, where $\ell_{n/2}$ is the number of times $\cD_{n/2}$ appears inside $L^2(\widetilde{\Gamma}\backslash G)$.  The $\overline{\mathscr{D}}_{n/2}$ is unitarily isomorphic to $\mathbb{C}^{\ell_{n/2}}\otimes\bar\cD_{n/2}$,  where $\ell_{n/2}$ is the number of times $\bar \cD_{n/2}$ appears inside $L^2(\widetilde{\Gamma}\backslash G)$.
\end{rem}

\begin{prop}
    $\ell_{n/2}$ is the number of independent normalized holomorphic modular forms of weight $n$. The number of independent normalized antiholomorphic modular forms of weight $n$ is also given by $\ell_{n/2}$. 
\end{prop}

\begin{proof}
    Consider a lowest weight vector $F_{n/2,a}$ inside $\mathscr{D}_{n/2}$. Clearly, $F_{n/2,a}\in V_{n/2}$ and $e^{in\theta/2}F_{n/2,a}$ is independent of $\theta$. Hence, one can define
    \begin{equation}
       h_{n,a}(x,y):= y^{-n/2}e^{in\theta/2} F_{n/2,a}(x,y,\theta).
    \end{equation}
    The $\widetilde{\Gamma}$ invariance of $F$ implies that $\tilde{h}_{n,a}(z):=h_{n,a}(\text{Re}(z),\text{Im}(z))$ transforms like a holomorphoic modular form of weight $n$ for $\gamma\in \widetilde{\Gamma}$. Recalling eq.~\!\eqref{split} and Proposition \ref{prop:multiplier}, this action can be naturally extended to $\bar\Gamma$ by using $\chi$. The fact that $F_{n/2,a}$ is a lowest weight vector translates to $L_{1} F_{n/2,a}=0$, which implies $\tilde{h}_{n,a}(z)$ is holomorphic. 
A similar proof holds for $\overline{\mathscr{D}}_{n/2}$ by considering the highest weight vector, belonging to $V_{-n/2}$. Hence the proposition follows. 
\end{proof}
\begin{rem}
    $\mathscr{C}_{\lambda^{(0)}_k}$ is unitarily isomorphic to 
    
    \begin{enumerate}
        \item $\mathbb{C}^{d_k}\otimes\mathcal{C}_{s}$ with $s:=\sqrt{1/4-\lambda_{k}^{(0)}}$ if $\lambda_{k}^{(0)}<1/4$,
        \item  $\mathbb{C}^{d_k}\otimes\mathcal{P}^{+}_{i \nu}$ with $\nu:=\sqrt{\lambda_{k}^{(0)}-1/4}$ if $\lambda_{k}^{(0)}\geq 1/4$,
    \end{enumerate} 
     where $d_k$ is the number of times $\mathcal{C}_{s}$ or  $\mathcal{P}^{+}_{i \nu}$ appears inside $L^2(\widetilde{\Gamma}\backslash G)$.
\end{rem}

\begin{prop}
    $d_k$ is the degeneracy of $\lambda_k^{(0)}$, an eigenvalue of the Laplace operator on $X$.
\end{prop}

\begin{proof}
    Recall $$\mathscr{C}_{\lambda^{(0)}_k}\simeq \mathbb{C}^{d_k}\otimes \underset{k\in\mathbb{Z}}{\bigoplus}\ V_{k}.$$
    Consider a vector inside $V_0$, it gets mapped to a vector $F$ inside $\mathscr{C}_{\lambda^{(0)}_k}$. Since $i\partial_\theta F=L_0F=0$, we can define $\phi(x,y):=F(x,y,\theta)$ and it follows $\phi(x,y)$ is a $\widetilde{\Gamma}$ invariant function on $X$. Furthermore, we have
    \begin{equation}
        C_2 \phi(x,y)= y^2(\partial_x^2+\partial_y^2) \phi(x,y).
    \end{equation}
On the other hand, $C_2$ of $\mathscr{C}_{\lambda^{(0)}_k}$ is $\lambda^{(0)}_k$. 
Hence, $\phi(x,y)$ is a $\widetilde{\Gamma}$ invariant eigenfunction of the Laplace operator, i.e. $y^2(\partial_x^2+\partial_y^2)$, with eigenvalue $\lambda_k^{(0)}$, and the proposition follows.
\end{proof}

\begin{rem}
    $\mathscr{P}^{-}_{\lambda^{(1/2)}_k}$ is unitarily isomorphic to $\mathbb{C}^{d'_k}\otimes\mathcal{P}^{-}_{i \nu}$ with $\nu:=\sqrt{\lambda_{k}^{(1/2)}-1/4}$ and $\lambda_{k}^{(1/2)}>1/4$, where $d'_k$ is the number of times $\mathcal{P}^{-}_{i \nu}$ appears inside $L^2(\widetilde{\Gamma}\backslash G)$.
\end{rem}

\begin{prop}
    $d'_k$ is the degeneracy of $\lambda_k^{(1/2)}$, an eigenvalue of the weight-$1$ automorphic Laplace operator (see defn.~\!\ref{def:weightKautomorphicLaplacian}) on $X$.
\end{prop}

\begin{proof}
    Recall $$\mathscr{P}^{-}_{\lambda^{(1/2)}_k}\simeq \mathbb{C}^{d_k}\otimes \underset{k\in1/2+\mathbb{Z}}{\bigoplus}\ V_{k}.$$
    Consider a vector inside $V_{1/2}$, it gets mapped to a vector $\Psi(x,y,\theta)$ inside $\mathscr{P}^{-}_{\lambda^{(1/2)}_k}$. Define $\psi(x,y):=e^{i\theta/2}\Psi(x,y,\theta)$. Since $i\partial_\theta \Psi=L_0\Psi=1/2$, it follows that $\psi(x,y)$ is a $\widetilde{\Gamma}$-equivariant form on $X$. Furthermore, we have
    \begin{equation}
        C_2 \psi(x,y)= \left[y^2(\partial_x^2+\partial_y^2)-iy\partial_x\right] \psi(x,y).
    \end{equation}
On the other hand, $C_2$ of $\mathscr{P}^{-}_{\lambda^{(1/2)}_k}$ is $\lambda^{(1/2)}_k$. 
Hence, $\psi(x,y)$ is a $\widetilde{\Gamma}$-equivariant eigenform of the weight-$1$ automorphic Laplacian operator, i.e $y^2(\partial_x^2+\partial_y^2)-iy\partial_x$. Therefore, the proposition follows.
\end{proof}

Now let us identify the coherent states inside $L^2(\widetilde{\Gamma}\backslash G)$. 

\begin{defn}[Coherent States inside $\mathscr{D}_{n/2}^{\infty}$ and $\overline{\mathscr{D}}^{\infty}_{n/2}$ ]
    Recall Proposition~\!\eqref{decomp} and Remark \ref{rem:coherent} that $\mathscr{D}_{n/2}$ and $\overline{\mathscr{D}}_{n/2}$ are unitarily isomorphic to $\mathbb{C}^{\ell_{n/2}}\otimes \cD_{n/2}$ and $\mathbb{C}^{\ell_{n/2}}\otimes \bar\cD_{n/2}$ respectively.
    
    Following \cite{Kravchuk:2021akc}, we define the isomorphism $\tau_{n/2}:\mathbb{C}^{\ell_{n/2}}\otimes \cD_{n/2}\to\mathscr{D}_{n/2}$ and define 
    \begin{equation}
        \mathscr{O}_{n/2,a}(z):= \tau_{n/2}\left(e_a\otimes \mathcal{O}(z)\right)\in\mathscr{D}_{n/2}^{\infty}\subset\mathcal{C}^{\infty}(\widetilde{\Gamma}\backslash G) \,.
    \end{equation}
    Here $e_a$ for $a=1,2,\cdots \ell_{n/2}$ is the standard basis for $\mathbb{C}^{\ell_{n/2}}$. We further define $\bar\tau_{n/2}:\mathbb{C}^{\ell_{n/2}}\otimes\bar\cD_{n/2}\to\overline{\mathscr{D}}_{n/2}$ as
    \begin{equation}
        \bar\tau_{n/2}(v\otimes f) =\overline{\tau_{n/2}(\bar v\otimes \widetilde{f})}\,,\quad \widetilde{f}(z):=z^{-n}\overline{f(\bar z^{-1})}\,,
    \end{equation}
    and 
     \begin{equation}
        \widetilde{\mathscr{O}}_{n/2,a}(z):= \bar\tau_{n/2}\left(e_a\otimes \widetilde{\mathcal{O}}(z)\right)\in\overline{\mathscr{D}}_{n/2}^{\infty}\subset\mathcal{C}^{\infty}(\widetilde{\Gamma}\backslash G) \,,
    \end{equation}
\end{defn}

\begin{prop}[\cite{Kravchuk:2021akc}]
    The generators of the complexified Lie algebra $\mathfrak{sl}_2(\mathbb{C})$ act on coherent states as follows:
    \begin{equation}
\begin{aligned}
L_{-1} \cdot \mathscr{O}_{n/2, a}(z) & =\partial_z \mathscr{O}_{n/2, a}(z), & L_{-1} \cdot \widetilde{\mathscr{O}}_{n/2, a}(z) & =\partial_z \widetilde{\mathscr{O}}_{n/2, a}(z), \\
L_0 \cdot \mathscr{O}_{n/2, a}(z) & =\left(z \partial_z+n/2\right) \mathscr{O}_{n/2, a}(z), & L_0 \cdot \widetilde{\mathscr{O}}_{n/2, a}(z) & =\left(z \partial_z+n/2\right) \widetilde{\mathscr{O}}_{n/2, a}(z), \\
L_1 \cdot \mathscr{O}_{n/2, a}(z) & =\left(z^2 \partial_z+ n z\right) \mathscr{O}_{n/2, a}(z), & L_1 \cdot \widetilde{\mathscr{O}}_{n/2, a}(z) & =\left(z^2 \partial_z+ n z\right) \widetilde{\mathscr{O}}_{n/2, a}(z).
\end{aligned}
\end{equation}
\end{prop}

\begin{prop}[\cite{Kravchuk:2021akc}]\label{eq:coherentIdentity}
We have the following identity:
    $$\overline{\left(\mathscr{O}_{n/2, a}(z)\right)}=(\bar{z})^{-n} \widetilde{\mathscr{O}}_{n/2, a}\left((\bar{z})^{-1}\right).$$
\end{prop}

\subsection{Correlators of smooth functions on $G$}
\begin{defn}
   We consider smooth functions $F_1,F_2,\cdots F_n$ on $\widetilde{\Gamma}\backslash G$. Their correlator is defined as 
   \begin{equation}
       \langle F_1 F_2\cdots F_n\rangle := \int_{\widetilde{\Gamma}\backslash G}\ dg\ F_1(g)F_2(g)\cdots F_n(g)\,.
   \end{equation}
\end{defn}

Since the functions $F_i$ are smooth and $\widetilde{\Gamma}\backslash G$ is compact, the correlator is well-defined and finite. 

\begin{prop}
    The correlator $\langle \cdot \ldots \cdot \rangle: \left(\mathcal{C}^{\infty}(\widetilde{\Gamma}\backslash G)\right)^N\to \mathbb{C}$ is a $G$ invariant functional.
\end{prop}

\begin{prop}
    The two-function correlator is given by 
    \begin{equation}
    \langle \mathscr{O}_{m,i}(z_1)\widetilde{\mathscr{O}}_{n,j}(z_2)\rangle=\frac{\delta_{m,n}\delta_{i,j}}{(z_2-z_1)^{2n}}\,,\quad \text{where}\quad\  2m,2n\in \mathbb{Z}_+
\end{equation}
\end{prop}

\begin{proof}
Using Proposition \ref{eq:coherentIdentity} alongside with orthogonality of the decomposition given in Proposition \ref{decomp}, we find
$$    \langle \mathscr{O}_{m,i}(z_1)\widetilde{\mathscr{O}}_{n,j}(z_2)\rangle 
=z_2^{-2n} \left(\mathscr{O}_{n,j}( \overline{z}_2^{-1}) ,\mathscr{O}_{m,i}(z_1)\right)=\delta_{m,n}\delta_{i,j}z_2^{-2n} \left(\cO_n( \overline{z}_2^{-1}) ,\cO_n(z_1)\right)\,,$$ 
and then we use Proposition \ref{innerproductofcoherentstates}.
\end{proof}

\begin{prop}\label{prop:3pointDDDbar}
Define $z_{ij}=z_i-z_j$. The three-function correlator between vectors inside $\mathscr{D}_k$,$\mathscr{D}_l$ and $\overline{\mathscr{D}}_{m}$ with $2k,2l,2m\in\mathbb{Z}_+$ is given by 
    \begin{equation}
\left\langle\mathscr{O}_{k, a}\left(z_1\right) \mathscr{O}_{l, b}\left(z_2\right) \widetilde{\mathscr{O}}_{m, c}\left(z_3\right)\right\rangle=\frac{f_{(k;a)(l;b)}^{m;c}}{z_{21}^{k+l-m} z_{31}^{k+m-l} z_{32}^{l+m-k}}\,.
\end{equation}
\end{prop}

\begin{proof}
    The result follows from using the same steps involving $G$-invariance to prove a similar result in \cite{Kravchuk:2021akc} for $k,l,m\in\mathbb{Z}$. 
\end{proof}

\begin{rem}
    We recall that 
    \begin{equation}
        \mathscr{O}_{k, a}(0)=F_{k,a}\,,\quad  \mathscr{O}_{l, b}(0)=F_{l,b}\,,\quad \lim_{z\to\infty}z^{2m}\mathscr{O}_{m, c}(z)=\bar{F}_{m,c}\,.
    \end{equation}
    Thus we have
    \begin{equation}
        f^{k+l,c}_{(k;a)(l;b)}=\int_{\widetilde{\Gamma}\backslash G} dg\  F_{k,a}F_{l,b}\bar{F}_{k+l,c}=\frac{1}{\text{vol}(X)}\int_{X}dx\ dy\ y^{2(k+l-1)}h_{k,a}h_{l,b}\bar{h}_{m,c}\,.
    \end{equation}
\end{rem}

\begin{defn}
    The $G$-invariant cross-ratio for four points $z_1,z_2,z_3,z_4$ is given by 
    \begin{equation}
      z:=\frac{z_{12}z_{34}}{z_{13}z_{24}}\,,\quad z_{ij}:=z_i-z_j\,.
    \end{equation}
\end{defn}
\begin{prop}
    Let $2n_1,2n_2,2n_3,2n_4\in\mathbb{Z}_+$. 
\begin{equation}\label{eq:4point}
    \langle \mathscr{O}_{n_1} (z_1)\mathscr{O}_{n_2} (z_2)\widetilde{\mathscr{O}}_{n_3} (z_3)\widetilde{\mathscr{O}}_{n_4} (z_4)\rangle =\frac{1}{z_{12}^{n_1+n_2}z_{34}^{n_3+n_4}}\left(\frac{z_{24}}{z_{14}}\right)^{n_1-n_2}\left(\frac{z_{14}}{z_{13}}\right)^{n_3-n_4}g_{\{n_1,n_2,n_3,n_4\}}(z).
\end{equation}
\end{prop}
\begin{proof}
    The proof mimics the one appearing in \cite{Kravchuk:2021akc} for $n_1=n_2=n\in\mathbb{Z}$. Basically, $G$ invariance implies that the $4$-function correlator can be written in terms of a single variable function of the cross-ratio $z$. 
\end{proof}

\subsection{Product expansion}

\begin{equation}
    F_1F_2= P_{\mathbb{C}}(F_1F_2)+\sum_{n=1}^{\infty}\left[P_{\mathscr{D}_{n/2}}(F_1F_2)+ P_{\overline{\mathscr{D}}_{n/2}}(F_1F_2)\right]+\sum_{k=1}^{\infty}P_{\mathscr{C}_{\lambda_{k}^{(0)}}}(F_1F_2)+\sum_{k=1}^{\infty}P_{\mathscr{P}^{-}_{\lambda_{k}^{(1/2)}}}(F_1F_2)\,.
\end{equation}
Here $P_H$ refers to the orthogonal projection onto the irrep $H$. 

We choose $F_i$ from the irrep $H_i$. Now the key point is that given an irrep $H_m$, $G$-invariance constrains the dependence of $P_{H_m}(F_1F_2)$ on $F_1$ and $F_2$ upto finitely many constants. And those finitely many constants are related to triple product integrals between elements of $H_1$, $H_2$ and $H_m$. The spectral identities come from the associativity constriants i.e. 
\begin{equation}
    ((F_1F_2) F_3 )=(F_1(F_2F_3))\,.
\end{equation}

In what follows we will take $F_1=\mathscr{O}_{n_1}$, $F_2=\mathscr{O}_{n_2}$, $F_3=\widetilde{\mathscr{O}}_{n_3}$. A convenient way to encode these constraints is to consider $4$ point functions of the form $\langle \mathscr{O}_{n_1}\mathscr{O}_{n_2}\widetilde{\mathscr{O}}_{n_3}\widetilde{\mathscr{O}}_{n_4}\rangle$.




Now we introduce some lemmas which are relevant to the product expansion of coherent states.

\begin{lemma}
   $P_H\left(\mathscr{O}_{n_1}\left(z_1\right) \mathscr{O}_{n_2}\left(z_2\right)\right)=0$ unless $H=\mathscr{D}_p$ with $p-n_1-n_2\in\mathbb{Z}_{\geq 0}$. If $\mathscr{O}_{n_1}=\mathscr{O}_{n_2}$ with $n_1=n_2$, we further have  $p-2n=0(\text{mod}\ 2)$. 
\end{lemma}

\begin{proof}
    The proof is similar to that of Lemma $3.9$ of \cite{Kravchuk:2021akc} with minor modifications to include the cases when $n_1$ or $n_2$ are half-integers and possibly different.
\end{proof}

\begin{prop}
\begin{equation}
    \mathscr{O}_{n_1,a_1}(z_1)\mathscr{O}_{n_2,a_2}(z_2)= \sum_{p-n_1-n_2\in\mathbb{Z}_{\geq 0}}\sum_{a=1}^{\ell_p} f_{(n_1;a_1)(n_2;a_2)}^{p;a}\tau_p\left(e_a\otimes C_p(z_1,z_2) \right).
\end{equation}
Here $C_p(z_1,z_2)\in\cD_p$ is defined as 
\begin{equation}
    C_p(z_1,z_2)(z):= \sqrt{\frac{2p-1}{\pi}} \frac{1}{z_{12}^{n_1+n_2-p} z_{13}^{n_1+p-n_2} z_{23}^{n_2+p-n_1}}.
\end{equation}
If $n_1=n_2=n$, $f_{(n_1;a_1)(n_2;a_2)}^{p;a}$ vanishes for $p=2n+1$.
\end{prop}

\begin{lemma}
   $P_H\left(\mathscr{O}_{n_1}\left(z_1\right) \widetilde{\mathscr{O}}_{n_1}\left(z_2\right)\right)=0$ unless $H=\mathbb{C}$ or  $H=\mathscr{C}_{\lambda^{0}_k}$.
\end{lemma}

\begin{proof}
    The proof is similar to that of Lemma $3.11$ of \cite{Kravchuk:2021akc} with minor modifications. $\mathscr{P}^{-}_{i\nu}$ does not appear since the center of $\mathrm{SL}(2,\mathbb{R})$ acts trivially on $\mathscr{O}_{n_1}(z_1) \widetilde{\mathscr{O}}_{n_1}(z_2)$ while the center acts nontrivially on vectors inside $\mathscr{P}^{-}_{i\nu}$.
\end{proof}
\begin{prop}[See Lemma 3.12 of \cite{Kravchuk:2021akc}]
Let $n_2-n_1\in\mathbb{Z}$, we have
\begin{equation}
P_{\mathscr{C}_{\lambda_{k}^{(0)}}}\left(\mathscr{O}_{n_1,a}(z_1)\widetilde{\mathscr{O}}_{n_2,b}(z_2)\right)=\frac{\delta_{n_1,n_2}\delta_{a,b}}{(z_2-z_1)^{2n}}+ \sum_{r=1}^{d_k} c_{(n_1;a)(n_2;b)}^{k;r}\kappa^{(0)}_k\left(e_r \otimes\widetilde{C}^{(0)}_{k}(z_1,z_2)\right).
\end{equation}
   Here  $\kappa^{(0)}_k$ is the unitary isomorphism between $\mathbb{C}^{d_k}\otimes R_k$ and $\mathscr{C}_{\lambda_{k}^{(0)}}$, where $R_k=\mathcal{P}^{+}_{i\nu}$ with $\lambda_{k}^{(0)}=1/4+\nu^2$ or $R_k=\mathcal{C}_{s}$ with $\lambda_{k}^{(0)}=1/4-s^2$. Furthermore, $\widetilde{C}^{(0)}_{k}(z_1,z_2)\in R_k$, and is given by 
   \begin{equation*}
       \widetilde{C}^{(0)}_{k}(z_1,z_2)(z_0):=\frac{N_{\Delta_k}z_2^{-2n_2}z_0^{n_2-n_1}}{\left(1-z_1z_2^{-1}\right)^{n_1+n_2-\Delta_k}\left(1-z_1z_0^{-1}\right)^{n_1-n_2+\Delta_k}\left(1-z_2^{-1}z_0\right)^{n_2-n_1+\Delta_k}}\,,\quad |z_0|=1.
   \end{equation*}
\end{prop}

\begin{lemma}
Let $n_1>n_2$ such that $2(n_1+n_2)=0\,(\text{mod}\ 2)$.
   $P_H\left(\mathscr{O}_{n_1}\left(z_1\right) \widetilde{\mathscr{O}}_{n_2}\left(z_2\right)\right)=0$ unless $H=\mathscr{C}_{\lambda^{(0)}_k}$ or $H=\mathscr{D}_{m}$ with $1\leq m\leq  n_1-n_2$ and $m\in\mathbb{Z}$.
\end{lemma}
\begin{proof}
    The proof is similar to that of Lemma $3.11$ of \cite{Kravchuk:2021akc}, with minor modifications like in the previous lemma. 
\end{proof}

\begin{lemma}
Let $n_1<n_2$ such that $2(n_1+n_2)=0(\text{mod}\ 2)$.
   $P_H\left(\mathscr{O}_{n_1}\left(z_1\right) \widetilde{\mathscr{O}}_{n_2}\left(z_2\right)\right)=0$ unless $H=\mathscr{C}_{\lambda^{(0)}_k}$ or $H=\bar{\mathscr{D}}_{m}$ with $1\leq m\leq  n_2-n_1$ and $m\in\mathbb{Z}$.
\end{lemma}
\begin{proof}
    The proof is similar to the above. 
\end{proof}

\begin{prop}\label{prop:discrete-t}
    Let $n_1>n_2$ such that $2(n_1+n_2)=0(\text{mod}\ 2)$. We have for $p\leq n_1-n_2$ and $p\in\mathbb{Z}_{\geq 0}$,
        \begin{equation}
  P_{\mathscr{D}_p} \left( \mathscr{O}_{n_1,a_1}(z_1)\widetilde{\mathscr{O}}_{n_2,a_2}(z_2)\right)= \sum_{a=1}^{\ell_p} f_{(n_1;a_1)(p;a)}^{n_2;a_2}\tau_p\left(e_a\otimes \widetilde{C}_p(z_1,z_2) \right).
\end{equation}
Here $\widetilde{C}_p(z_,z_2)\in\cD_p$ is defined as 
\begin{equation}
    \widetilde{C}_p(z_1,z_2)(z_3):= \sqrt{\frac{2p-1}{\pi}} \frac{1}{z_{13}^{n_1+n_2-p} z_{12}^{n_1+p-n_2} z_{32}^{n_2+p-n_1}}.
\end{equation}
\end{prop}

\begin{rem}
    The structure constant $ f_{(n_1;a_1)(p;a)}^{n_2;a_2}$ appearing here is the same as the one appearing in $\langle\mathscr{O}_{n_1,a_1}(z_1)\mathscr{O}_{p,a}(z_3)\widetilde{\mathscr{O}}_{n_2,a_2}(z_2)\rangle$.
\end{rem}
\begin{rem}
    A similar statement can be made for $n_1<n_2$.
\end{rem}

\begin{lemma}
Let $n_1>n_2$ such that $2(n_1+n_2)=1(\text{mod}\ 2)$.
   $P_H\left(\mathscr{O}_{n_1}\left(z_1\right) \widetilde{\mathscr{O}}_{n_2}\left(z_2\right)\right)=0$ unless $H=\mathscr{P}^{-}_{\lambda^{(1/2)}_k}$  or $H=\mathscr{D}_{m}$ with $1/2\leq m\leq  n_1-n_2$ and $m\in 1/2+ \mathbb{Z}$.
\end{lemma}

\begin{proof}
    $\mathbb{C}$ and $\mathscr{C}_{\lambda_k^{(0)}}$ do not appear because the center acts nontrivially on $\mathscr{O}_{n_1}\left(z_1\right) \widetilde{\mathscr{O}}_{n_2}\left(z_2\right)$ for $2(n_1+n_2)=1(\text{mod}\ 2)$. The rest of the proof is similar to that of Lemma $3.11$ of \cite{Kravchuk:2021akc}.
 \end{proof}
 
\begin{lemma}
Let $n_1<n_2$ such that $2(n_1+n_2)=1(\text{mod}\ 2)$.
   $P_H\left(\mathscr{O}_{n_1}\left(z_1\right) \widetilde{\mathscr{O}}_{n_2}\left(z_2\right)\right)=0$ unless $H=\mathscr{P}^{-}_{\lambda^{(1/2)}_k}$  or $H=\bar{\mathscr{D}}_{m}$ with $1/2\leq m\leq  n_2-n_1$ and $m\in 1/2+ \mathbb{Z}$.
\end{lemma}

\begin{proof}
    The proof is similar to the above one.
 \end{proof}

\begin{prop}
\begin{equation}
P_{\mathscr{P}^{-}_{\lambda_{k}^{(1/2)}}}\left(\mathscr{O}_{n_1,a}(z_1)\widetilde{\mathscr{O}}_{n_2,b}(z_2)\right)= \sum_{r=1}^{d'_k} s_{(n_1;a)(n_2;b)}^{k;r}\kappa^{(1/2)}_k\left(e_r \otimes\widetilde{C}^{(1/2)}_{k}(z_1,z_2)\right).
\end{equation}
   Here  $\kappa^{(1/2)}_k$ is the unitary isomorphism between $\mathbb{C}^{d'_k}\otimes \mathcal{P}^{-}_{i\nu}$ and $\mathscr{P}^{-}_{\lambda_{k}^{(1/2)}}$, where $\lambda_{k}^{(1/2)}=1/4+\nu^2$ and $\widetilde{C}^{(1/2)}_{k}(z_1,z_2)\in \mathcal{P}^{-}_{i\nu}$ is given by 
   \begin{equation}
       \widetilde{C}^{(1/2)}_{k}(z_1,z_2)(z):=\frac{N_{\Delta_k}z_2^{-2n}z_0^{n_2-n_1}}{\left(1-z_1z_2^{-1}\right)^{n_1+n_2-\Delta_k}\left(1-z_1z_0^{-1}\right)^{n_1-n_2+\Delta_k}\left(1-z_2^{-1}z_0\right)^{n_2-n_1+\Delta_k}}.
   \end{equation}
\end{prop}

\begin{prop}\label{prop:discrete-t-half}
    Let $n_1>n_2$ such that $2(n_1+n_2)=1(\text{mod}\ 2)$. We have for $p\leq n_1-n_2$ and $p\in1/2+\mathbb{Z}_{\geq 0}$
        \begin{equation}
  P_{\mathscr{D}_p} \left( \mathscr{O}_{n_1,a_1}(z_1)\widetilde{\mathscr{O}}_{n_2,a_2}(z_2)\right)= \sum_{a=1}^{\ell_p} f_{(n_1;a_1)(p;a)}^{n_2;a_2}\tau_p\left(e_a\otimes \widetilde{C}_p(z_1,z_2) \right).
\end{equation}
Here $\widetilde{C}_p(z_,z_2)\in\cD_p$ is defined as 
\begin{equation}
\begin{split}
      \widetilde{C}_p(z_1,z_2)(z)&:= \sqrt{\frac{2p-1}{\pi}} \frac{1}{z_{13}^{n_1+n_2-p} z_{12}^{n_1+p-n_2} z_{32}^{n_2+p-n_1}}\,,\quad\ p>1/2\\
       \widetilde{C}_{1/2}(z_1,z_2)(z)&:= \sqrt{\frac{1}{2\pi}} \frac{1}{z_{13}^{n_1+n_2-1/2} z_{12}^{n_1+1/2-n_2} z_{32}^{n_2+1/2-n_1}}\,.
\end{split}
\end{equation}
\end{prop}


Equipped with the above lemmas, we can study 
correlators of the form $\langle \mathscr{O}_{n_1} \mathscr{O}_{n_2} \widetilde{\mathscr{O}}_{n_2} \widetilde{\mathscr{O}}_{n_1} \rangle$. They are studied in \cite{Kravchuk:2021akc} for $n\in\mathbb{Z}_+$. They can easily be extended for half-integer $n_i$. In particular we have

\begin{thm}[Theorem $2.3$ \cite{Kravchuk:2021akc}: Extended]\label{thm:scalar}
Let $g_0(z):=g_{\{n_1,n_2,n_2,n_1\}}(z)$ be as in eq.~\eqref{eq:4point}. 
\begin{enumerate}
\item $g_0(z)$ has the following expansion, known as $s$-channel expansion:
    \begin{equation}
    g_0(z)=\sum_{p=n_1+n_2} \sum_{a=1}^{\ell_p}|f^{p,a}_{n_1,n_2}|^2 \mathcal{G}_p(n_1,n_2,n_2,n_1;z)\,,
\end{equation}
where
\begin{equation}
\mathcal{G}_p(n_1,n_2,n_2,n_1;z):=z^{p}{}_2F_{1}\left(p-n_{12},p-n_{12},2p,z\right).
\end{equation}

The $s$-channel sum and its derivatives converge uniformly on compact subsets of $\mathbb{C}\setminus(1,\infty)$. Furthermore, $f^{p,a}_{n_1,n_2}\in\mathbb{C}$ and satisfies 
\begin{equation}
    f^{p}_{n,n}=0\,,\ \text{if}\ p-2n=1(\text{mod}\ 2).
\end{equation}
\item $g_0(z)$ has the following expansion, known as $t$-channel expansion:
\begin{equation}
    g_0(z)=\frac{z^{n_1+n_2}}{(1-z)^{2n_2}}\left(1+\sum_{k=1}^{\infty}\sum_{a=1}^{d_k} c_{n_1,n_1}^{k;a}c_{n_2,n_2}^{k;a}\mathcal{H}_{\Delta_k^{(0)}}\left(n_1,n_2,n_2,n_1;z\right)\right),
\end{equation}
where
\begin{equation}
\mathcal{H}_{\Delta}\left(n_1,n_2,n_2,n_1;z\right):={ }_2 F_1\left(\Delta,1-\Delta, 1, \frac{z}{z-1}\right),
\end{equation}
and 
 $c_{n_i,n_i}^{k;a}\in\mathbb{R}$, $\Delta_k^{(0)}=\frac{1}{2}+i\sqrt{\lambda_k^{(0)}-1/4}$.
The $t$-channel sum and its derivatives converge uniformly on compact subsets of $\mathbb{C}\setminus(1,\infty)$.
\end{enumerate}
\end{thm}

\begin{thm}[Theorem $2.3$ \cite{Kravchuk:2021akc}: Extended]\label{thm:spinor}
Let $g_1(z):=g_{\{n_1,n_2,n_1,n_2\}}(z)$ be as in eq.~\!\eqref{eq:4point} with $n_1> n_2$ and $n_1\in\mathbb{Z}_+$, $2n_2=1(\text{mod}\ 2)$.
\begin{enumerate}
\item $g_1(z)$ has the following expansion, known as $s$-channel expansion:
    \begin{equation}
    g_1(z)=\sum_{p=n_1+n_2} \sum_{a=1}^{\ell_p}(-1)^{p-n_1-n_2}|f^{p,a}_{n_1,n_2}|^2 \mathcal{G}_p(n_1,n_2,n_1,n_2;z)\,,
\end{equation}
where
\begin{equation}
\mathcal{G}_p(n_1,n_2,n_1,n_2;z):=z^{p}{}_2F_{1}\left(p-n_{12},p+n_{12},2p,z\right).
\end{equation}

The $s$-channel sum and its derivatives converge uniformly on compact subsets of $\mathbb{C}\setminus(1,\infty)$. Furthermore, $f^{p,a}_{n_1,n_2}\in\mathbb{C}$ and satisfies 
\begin{equation}
    f^{p}_{n,n}=0\,,\ \text{if}\ p-2n=1(\text{mod}\ 2)\,. 
\end{equation}
\item $g_1(z)$ has the following expansion, known as $t$-channel expansion:
\begin{equation}
\begin{split}
    g_1(z)&=\frac{z^{n_1+n_2}}{(1-z)^{n_1+n_2}}\bigg(\sum_{k=1}^{\infty}\sum_{a=1}^{d'_k} |s_{n_1,n_2}^{k;a}|^2\mathcal{H}_{\Delta_k^{(1/2)}}\left(n_1,n_2,n_1,n_2;z\right)+ \\
&+\sum_{0<m\leq n_1-n_2}\ \sum_{a=1}^{\ell_m}\ |f^{n_1}_{n_2,(m;a)}|^2\mathcal{H}_{m}\left(n_1,n_2,n_1,n_2;z\right)\bigg),
    \end{split}
\end{equation}
where
\begin{equation}
\mathcal{H}_{\Delta}\left(n_1,n_2,n_1,n_2;z\right):=(1-z)^{n_{12}}{ }_2 F_1\left(\Delta+n_{21},1-\Delta+n_{21} , 1 ,\frac{z}{z-1}\right),
\end{equation}
and 
 $s_{n_1,n_2}^{k;a}, f^{n_1}_{n_2,(m;a)}\in\mathbb{C}$, $\Delta_k^{(1/2)}=\frac{1}{2}+i\sqrt{\lambda_k^{(1/2)}-1/4}$.
The $t$-channel sum and its derivatives converge uniformly on compact subsets of $\mathbb{C}\setminus(1,\infty)$.
\end{enumerate}
\end{thm}


\subsection{Spectral Identities and Linear Programming}

\begin{defn}\label{def:kernel}
Define $r:=1-m-n_1-n_2$ and $\Delta:=1/2+i\sqrt{\lambda-1/4}$. We further define 
     \begin{equation}
     \begin{split}
&\mathcal{F}_m^{(n_1,n_2,n_3,n_4)}(\lambda):=\\
&\oint \frac{dz}{2\pi i} z^{-2}\mathcal{G}_{r}(-n_{1},-n_2,-n_3,-n_4;z)\frac{z^{n_1+n_2}}{(1-z)^{n_2+n_3}}\mathcal{H}_{\Delta}(n_1,n_2,n_3,n_4;z) \,.
     \end{split}
 \end{equation}
\end{defn}

\begin{defn} 
We define $S_{p;(n_1,n_2)}:=\sum_{a=1}^{\ell_p}|f^{p,a}_{n_1,n_2}|^2$.
\end{defn}

\begin{prop}\label{prop:scalar}
Let $m\in\mathbb{Z}_{\geq 0}$. The consistency condition coming from the four function correlator, $\langle \mathscr{O}_{n_1} (z_1)\mathscr{O}_{n_2} (z_2)\widetilde{\mathscr{O}}_{n_2} (z_3)\widetilde{\mathscr{O}}_{n_1}(z_4) \rangle$ is given by 
 \begin{equation}
 S_{n_1+n_2+m;(n_1,n_2)}= \left(\mathcal{F}_m^{(n_1,n_2,n_2,n_1)}(0)+\sum_{k=1}^{\infty}\left[\sum_{a=1}^{d_k}c_{n_1,n_1}^{k;a}c_{n_2,n_2}^{k;a}\right]\mathcal{F}_m^{(n_1,n_2,n_2,n_1)}(\lambda_k^{(0)})\right).
\end{equation}
\end{prop}
\begin{proof}
We recall Theorem~\!\ref{thm:scalar} and equate the s-channel and t-channel expansions and integrate against $\mathcal{G}_r$ introduced in \ref{def:kernel}, followed by exchanging the sum and integral using uniform convergence of the expansion. Noting that 
$$ 
 \oint \frac{dz}{2\pi i}z^{-2}\mathcal{G}_{1-m_1-n_1-n_2}(-n_1,-n_2,-n_3,-n_4)\mathcal{G}_{m_2+n_1+n_2}(n_1,n_2,n_3,n_4)=\delta_{m_1,m_2}\,,$$ the proposition follows. 
\end{proof}

\begin{prop}\label{prop:spinor}
Let $m\in\mathbb{Z}_{\geq 0}$, $n_1\in\mathbb{N}$ and $2n_2=1(\text{mod}\ 2)$. The consistency condition coming from $\langle \mathscr{O}_{n_1} (z_1)\mathscr{O}_{n_2} (z_2)\widetilde{\mathscr{O}}_{n_1}(z_3) \widetilde{\mathscr{O}}_{n_2}(z_4) \rangle$ is given by 
 \begin{equation}
 \begin{split}
     (-1)^{m} S_{n_1+n_2+m;(n_1,n_2)}&= \bigg(\sum_{k=1}^{\infty}\left[\sum_{a=1}^{d'_k}|s_{n_1,n_2}^{k;a}|^2\right]\mathcal{F}_m^{(n_1,n_2,n_1,n_2)}(\lambda_k^{(1/2)}) +\\
&+\sum_{\underset{q\in 1/2+\mathbb{Z}_{\geq 0}}{1/2\leq q\leq n_1-n_2}}\ \sum_{a=1}^{\ell_q}\ |f^{n_1}_{n_2,(q;a)}|^2\mathcal{F}_m^{(n_1,n_2,n_1,n_2)}(q(1-q))\bigg).
 \end{split}
\end{equation}
\end{prop}

\begin{proof}
     It follows from Theorem \ref{thm:spinor}. The proof is similar to that of Proposition \ref{prop:scalar}.
\end{proof}

 In what follows, we consider a hyperbolic spin orbifold supporting homolomorphic modular forms of weight $n$ and $2n$ such that $n=1(\text{mod } 2)$. If $\ell_{n/2}>1$ and/or $\ell_n>1$, we choose a particular $\mathscr{O}_{n} $ and a particular $\mathscr{O}_{n/2} $  such that $f_{n/2,n/2}^n\neq 0$. One can always consider a bigger system correlator using explicitly the value of $\ell_n$ and $\ell_{n/2}$, however, in order to find out a universal bound irrespective of $\ell_n$ and $\ell_{n/2}$, we choose not to do so. This leads us to consider the system of following correlators $\langle \mathscr{O}_{n} \mathscr{O}_{n} \widetilde{\mathscr{O}}_{n}\widetilde{\mathscr{O}}_{n}\rangle$ , $\langle \mathscr{O}_{\frac{n}{2}} \mathscr{O}_{\frac{n}{2}} \widetilde{\mathscr{O}}_{\frac{n}{2}}\widetilde{\mathscr{O}}_{\frac{n}{2}} \rangle$ , $\langle \mathscr{O}_{n} \mathscr{O}_{\frac{n}{2}} \widetilde{\mathscr{O}}_{\frac{n}{2}}\widetilde{\mathscr{O}}_{n}\rangle$ and $\langle \mathscr{O}_{n} \mathscr{O}_{\frac{n}{2}} \widetilde{\mathscr{O}}_{n}\widetilde{\mathscr{O}}_{\frac{n}{2}} \rangle$, where we have suppressed the degeneracy index $a$ in $(n,a)$ and so on. The explicit crossing equations read
 \begin{prop}\label{crossingEqs}
 Let $m\in \mathbb{Z}_{\geq 0}$, $n\in\mathbb{N}$. We have
      \begin{equation*}
      \begin{split}
           S_{2n+2m;(n,n)}&= \left(\mathcal{F}_{2m}^{(n,n,n,n)}(0)+\sum_{k=1}^{\infty}\left[\sum_{a=1}^{d_k}c_{n,n}^{k;a}c_{n,n}^{k;a}\right]\mathcal{F}_{2m}^{(n,n,n,n)}(\lambda_k^{(0)})\right),\\
           0&= \left(\mathcal{F}_{2m+1}^{(n,n,n,n)}(0)+\sum_{k=1}^{\infty}\left[\sum_{a=1}^{d_k}c_{n,n}^{k;a}c_{n,n}^{k;a}\right]\mathcal{F}_{2m+1}^{(n,n,n,n)}(\lambda_k^{(0)})\right),\\
           S_{n+2m;(n/2,n/2)}&= \left(\mathcal{F}_{2m}^{(n/2,n/2,n/2,n/2)}(0)+\sum_{k=1}^{\infty}\left[\sum_{a=1}^{d_k}c_{n/2,n/2}^{k;a}c_{n/2,n/2}^{k;a}\right]\mathcal{F}_{2m}^{(n/2,n/2,n/2,n/2)}(\lambda_k^{(0)})\right),\\ 
           0&= \left(\mathcal{F}_{2m+1}^{(n/2,n/2,n/2,n/2)}(0)+\sum_{k=1}^{\infty}\left[\sum_{a=1}^{d_k}c_{n/2,n/2}^{k;a}c_{n/2,n/2}^{k;a}\right]\mathcal{F}_{2m+1}^{(n/2,n/2,n/2,n/2)}(\lambda_k^{(0)})\right),\\ 
              S_{3n/2+m;(n,n/2)}&= \left(\mathcal{F}_m^{(n,n/2,n/2,n)}(0)+\sum_{k=1}^{\infty}\left[\sum_{a=1}^{d_k}c_{n/2,n/2}^{k;a}c_{n,n}^{k;a}\right]\mathcal{F}_m^{(n,n/2,n/2,n)}(\lambda_k^{(0)})\right),\\
           (-1)^{m} S_{3n/2+m;(n,n/2)}&= \bigg(\sum_{k=1}^{\infty}\left[\sum_{a=1}^{d'_k}|s_{n,n/2}^{k;a}|^2\right]\mathcal{F}_m^{(n,n/2,n,n/2)}(\lambda_k^{(1/2)}) +\\
&+\sum_{0<q\leq n/2}\ \sum_{a=1}^{\ell_q}\ |f^{n}_{n/2,(q;a)}|^2\mathcal{F}_m^{(n,n/2,n,n/2)}(q(1-q))\bigg).
      \end{split}
\end{equation*}
 \end{prop}

 \begin{proof}
     This follows from Propositions \ref{prop:scalar} and \ref{prop:spinor} by choosing appropriate values of $n_1$ and $n_2$.
 \end{proof}

Based on the above crossing equations, let us formulate the following linear programming problems.\\

\textbf{Notation:} Given a hyperbolic spin orbifold $X$, we consider the Laplacian spectra $\{0\}\cup\{\lambda_{k}^{(0)}(X): \lambda_{k}^{(0)}(X)>0,\ k\in\mathbb{N}\}$, and the weight-$1$ automorphic Laplacian spectrum (not including the eigenvalue corresponding to a harmonic spinor if there is any) $\{\lambda_{k}^{(1/2)}(X): \lambda_{k}^{(1/2)}(X)>1/4,\ k\in\mathbb{N}\}$. Furthermore, without loss of generality, let us assume $\lambda_{k+1}^{(0)}(x)>\lambda_k^{(0)}(X)$ and $\lambda_{k+1}^{(1/2)}(X)>\lambda_k^{(1/2)}(X)$ for all $k\in\mathbb{N}$. 

\begin{linpro}[Spectral bound on Laplacian]\label{linpro0}
Given $\Lambda\in\mathbb{N}$, $\lambda_*^{(0)}\in\mathbb{R}_+$ and a hyperbolic spin orbifold 
with a holomorphic modular form of weight $n$ such that $n=1(\text{mod} 2)$, if there exists $\alpha\in \mathbb{R}^{\Lambda+1}$ with such that 
\begin{enumerate}
     \item 
$\alpha_{2m}\leq 0\,,\ \forall\ m\in\mathbb{Z}_{\geq 0}$,   \\ 
\item $\sum_{m=0}^\Lambda \alpha_{m}\mathcal{F}_m^{(n/2,n/2,n/2,n/2)}(0)=1$,\\

\item $\forall\  \lambda_k^{(0)}\geq \lambda_*^{(0)}$ we have 
\begin{equation}\label{eq:0}
    \sum_{m=0}^\Lambda \alpha_{m}\mathcal{F}_m^{(n/2,n/2,n/2,n/2)}(\lambda_k^{(0)}) \geq 0\,,
\end{equation}
\end{enumerate}
we must have $\lambda_1^{(0)}<\lambda_*^{(0)}$.
\end{linpro}
\begin{proof}
    By taking an appropriate linear combination of 3rd and 4th equation of \ref{crossingEqs}, we obtain
    \begin{equation*}
\sum_{k=1}^{\infty}\sum_{m=0}^\Lambda \alpha_{m}\mathcal{F}_m^{(n/2,n/2,n/2,n/2)}(\lambda_k^{(0)}) =-1+\sum_{m=0}^\Lambda \alpha_{2m}S_{n+2m;(n/2,n/2)}.
    \end{equation*}
     Now using the properties of $\alpha$, we show that the R.H.S of the above is strictly negative, hence so is the L.H.S. Using \eqref{eq:0}, the proposition follows.
\end{proof}
 \begin{SDP}[Spectral bound on Laplacian and Dirac: I]\label{linpro1}
Let us consider a hyperbolic spin orbifold 
, for which, $\exists\  n\in\mathbb{N}$ with $n=1(\text{mod}\ 2)$ such that the orbifold does not support any holomorphic modular form of weight strictly below $n$.\\

Given $\Lambda\in\mathbb{N}$, $\lambda_*^{(0)},\lambda_*^{(1/2)}\in\mathbb{R}_+$, if there exists $\alpha_i\in \mathbb{R}^{\Lambda+1}$ with $i=1,2,3,4$ such that 
 \begin{enumerate}
     \item 
$\alpha_{1;2m}\leq 0\,,\ \forall\ m\in\mathbb{Z}_{\geq 0}$,     \\
\item $\sum_{m=0}^{\Lambda} \left(\alpha_{1;m}\mathcal{F}_{m}^{(n,n,n,n)}(0)+\alpha_{2;m}\mathcal{F}_{m}^{(n/2,n/2,n/2,n/2)}(0)+\alpha_{3;m}\mathcal{F}_{m}^{(n,n/2,n/2,n)}(0)\right)=1$,\\
\item $\alpha_{2;2m}\leq 0\,,\ \forall\ m\in\mathbb{Z}_{\geq 0}\,,$\\
\item $\alpha_{3;m}+(-1)^m\alpha_{4;m} \leq 0\,,\ \forall\ m\in\mathbb{Z}_{\geq 0}\,,$\\
\item $\sum_{m'=0}^{\Lambda}\alpha_{4;m'} \mathcal{F}_{m'}^{(n,n/2,n,n/2)}\left(\frac{n}{2}(1-\frac{n}{2})\right)\geq 0\,,\ \forall\ m\in\mathbb{Z}_{\geq 0}\,,$\\
\item $\forall\  \lambda_k^{(1/2)}\geq \lambda_*^{(1/2)}$, we have \begin{equation}\label{eq:1p}
    \sum_{m=0}^\Lambda \alpha_{4;m}\mathcal{F}_m^{(n,n/2,n,n/2)}(\lambda_k^{(1/2)})\geq 0, \ \ 
\end{equation}
\item $\forall\  \lambda_k^{(0)}\geq \lambda_*^{(0)}$ we have 
\begin{equation}\label{eq:2p}
M(\lambda_k^{(0)}):=\begin{pmatrix}\sum_{m=0}^\Lambda \alpha_{1;m}\mathcal{F}_m^{(n,n,n,n)}(\lambda_k^{(0)})&\frac{1}{2}\sum_{m=0}^\Lambda \alpha_{3;m}\mathcal{F}_m^{(n,n/2,n/2,n)}(\lambda_k^{(0)})\\\frac{1}{2}\sum_{m=0}^\Lambda \alpha_{3;m}\mathcal{F}_m^{(n,n/2,n/2,n)}(\lambda_k^{(0)}) & \sum_{m=0}^\Lambda \alpha_{2;m}\mathcal{F}_m^{(n/2,n/2,n/2,n/2)}(\lambda_k^{(0)}) \end{pmatrix}\succeq 0,\ \ 
 \end{equation}
 \end{enumerate}
 we must have either $\lambda_1^{(0)}<\lambda_*^{(0)}$ or $\lambda_1^{(1/2)}<\lambda_*^{(1/2)}$.
 \end{SDP}

\begin{proof}
Taking appropriate linear combinations of \eqref{crossingEqs} and using the nonexistence of holomorphic modular form of weight strictly below $n$ for the orbifold under consideration, we can derive 
 \begin{equation*}
 \begin{split}
   &\sum_{k=1}^{\infty}\sum_{a=1}^{d_k} \begin{pmatrix} c^{k;a}_{n,n}, c^{k;a}_{n/2,n/2}\end{pmatrix} M(\lambda_k^{(0)})\begin{pmatrix} c^{k;a}_{n,n}\\ c^{k;a}_{n/2,n/2}\end{pmatrix}+\sum_{k=1}^{\infty}\sum_{a=1}^{d'_k}|s^{k;a}_{n,n/2}|^2\sum_{m=0}^\Lambda \alpha_{4;m}\mathcal{F}_m^{(n,n/2,n,n/2)}(\lambda_k^{(1/2)})\\
&=-1+ \sum_{m=0}^\Lambda\bigg(\alpha_{1,2m}S_{2n+2m;(n,n)}+ \alpha_{2;2m}S_{n+2m;(n/2,n/2)}-\\
&-\sum_{m'=0}^{\Lambda}\alpha_{4;m'} \mathcal{F}_{m'}^{(n,n/2,n,n/2)}\left(\frac{n}{2}\left(1-\frac{n}{2}\right)\right)\sum_{a=1}^{\ell_{n/2}}\ |f^{n}_{n/2,(n/2;a)}|^2 
+\left(\alpha_{3;m}+(-1)^m\alpha_{4;m}\right)S_{3n/2+m;(n,n/2)}\bigg)\,.\\
    \end{split}
 \end{equation*}
 Now using the properties of $\alpha$, we show that the R.H.S of the above is strictly negative, hence so is the L.H.S. Using \eqref{eq:1p} and \eqref{eq:2p}, the proposition follows.    
\end{proof}

\begin{rem}
    When $\ell_n=\ell_{n/2}=1$, the structure constant $\sum_{a=1}^{\ell_{n/2}} |f^{n}_{n/2,(n/2;a)}|^2$ appearing in the $t$-channel expansion in the last equation of \eqref{crossingEqs} is exactly same as $S_{n;(n/2,/2)}$. Thus we can formulate a stronger \texttt{SDP} problem, which we state below.
\end{rem}

\begin{SDP}[Spectral bound on Laplacian and Dirac: II]\label{linpro2}
Let us consider a hyperbolic spin orbifold 
, for which, $\exists\  n\in\mathbb{N}$ with $n=1(\text{mod}\ 2)$ such that the orbifold does not support any holomorphic modular form of weight strictly below $n$.
Finally, assume that the multiplicity of holomorphic modular forms of weight $n$ and $2n$ are respectively given by $\ell_{n/2}=1$ and $\ell_n=1$.  \\

Given $\Lambda\in\mathbb{N}$, $\lambda_*^{(0)},\lambda_*^{(1/2)}\in\mathbb{R}_+$, if there exists $\alpha_i\in \mathbb{R}^{\Lambda+1}$ with $i=1,2,3,4$ such that 
 \begin{enumerate}
     \item 
$\alpha_{1;2m}\leq 0\,,\ \forall\ m\in\mathbb{Z}_{\geq 0}$     \\
\item $\sum_{m=0}^{\Lambda} \left(\alpha_{1;m}\mathcal{F}_{m}^{(n,n,n,n)}(0)+\alpha_{2;m}\mathcal{F}_{m}^{(n/2,n/2,n/2,n/2)}(0)+\alpha_{3;m}\mathcal{F}_{m}^{(n,n/2,n/2,n)}(0)\right)=1$\\
\item $\alpha_{2;2m}-\delta_{2m,0}\sum_{m'=0}^{\Lambda}\alpha_{4;m'} \mathcal{F}_{m'}^{(n,n/2,n,n/2)}\left(\frac{n}{2}(1-\frac{n}{2})\right)\leq 0\,,\ \forall\ m\in\mathbb{Z}_{\geq 0}\,,$\\
\item $\alpha_{3;m}+(-1)^m\alpha_{4;m} \leq 0\,,\ \forall\ m\in\mathbb{Z}_{\geq 0}\,,$\\
\item $\forall\  \lambda_k^{(1/2)}\geq \lambda_*^{(1/2)}$, we have \begin{equation}\label{eq:1prime}
    \sum_{m=0}^\Lambda \alpha_{4;m}\mathcal{F}_m^{(n,n/2,n,n/2)}(\lambda_k^{(1/2)})\geq 0, \ \ 
\end{equation}
\item $\forall\  \lambda_k^{(0)}\geq \lambda_*^{(0)}$ we have 
\begin{equation}\label{eq:2prime}
M(\lambda_k^{(0)}):=\begin{pmatrix}\sum_{m=0}^\Lambda \alpha_{1;m}\mathcal{F}_m^{(n,n,n,n)}(\lambda_k^{(0)})&\frac{1}{2}\sum_{m=0}^\Lambda \alpha_{3;m}\mathcal{F}_m^{(n,n/2,n/2,n)}(\lambda_k^{(0)})\\\frac{1}{2}\sum_{m=0}^\Lambda \alpha_{3;m}\mathcal{F}_m^{(n,n/2,n/2,n)}(\lambda_k^{(0)}) & \sum_{m=0}^\Lambda \alpha_{2;m}\mathcal{F}_m^{(n/2,n/2,n/2,n/2)}(\lambda_k^{(0)}) \end{pmatrix}\succeq 0,\ \ 
 \end{equation}
 \end{enumerate}
 we must have either $\lambda_1^{(0)}<\lambda_*^{(0)}$ or $\lambda_1^{(1/2)}<\lambda_*^{(1/2)}$.
 \end{SDP}

\begin{proof}
Taking appropriate linear combinations of \eqref{crossingEqs} and using the nonexistence of holomorphic modular forms of weight strictly below $n$ for the orbifold under consideration, we can derive 
 \begin{equation*}
 \begin{split}
   &\sum_{k=1}^{\infty}\sum_{a=1}^{d_k} \begin{pmatrix} c^{k;a}_{n,n}, c^{k;a}_{n/2,n/2}\end{pmatrix} M(\lambda_k^{(0)})\begin{pmatrix} c^{k;a}_{n,n}\\ c^{k;a}_{n/2,n/2}\end{pmatrix}+\sum_{k=1}^{\infty}\sum_{a=1}^{d'_k}|s^{k;a}_{n,n/2}|^2\sum_{m=0}^\Lambda \alpha_{4;m}\mathcal{F}_m^{(n,n/2,n,n/2)}(\lambda_k^{(1/2)})\\
&=-1+ \sum_{m=0}^{\floor{\Lambda/2}}\bigg(\alpha_{1,2m}S_{2n+2m;(n,n)}+ \bigg(\alpha_{2;2m}-\delta_{2m,0}\sum_{m'=0}^{\Lambda}\alpha_{4;m'} \mathcal{F}_{m'}^{(n,n/2,n,n/2)}\left(\frac{n}{2}\left(1-\frac{n}{2}\right)\right)\bigg)\times\\
&\quad\quad\times S_{n+2m;(n/2,n/2)}\bigg)
+\sum_{m=0}^{\Lambda}\left(\alpha_{3;m}+(-1)^m\alpha_{4;m}\right)S_{3n/2+m;(n,n/2)}\,.\\
    \end{split}
 \end{equation*}
 Now using the properties of $\alpha$, we show that the R.H.S of the above is strictly negative, hence so is the L.H.S. Using \eqref{eq:1prime} and \eqref{eq:2prime}, the proposition follows.    
\end{proof}


In what follows, we will use $\mathscr{O}_n$ and $\mathscr{O}_{n/2}$ and assume that there is no holomorphic modular form of weight strictly below $n$ on the orbifold under consideration. This is clearly true for $n=1$. Furthermore, we will use a clever trick introduced in \cite{Kravchuk:2021akc} to reduce the size of the problem. The idea is to consider the correlators and do a Haar integral over $U(\ell_{n})$ and $U(\ell_{n/2})$.

\begin{prop}
For a unitary group $U(\ell)$, we have 
    \begin{equation*}
\begin{split}
\int_{U(\ell)} dU \ U_{a_1}^{a_1'}(\bar{U})_{a_2}^{a_2'}&=\frac{1}{\ell}\delta_{a_1,a_2}\delta_{a_1',a_2'},\\
\int_{U(\ell)} dU \ U_{a_1}^{a_1'}U_{a_2}^{a_2'}(\bar{U})_{a_3}^{a_3'}(\bar{U})_{a_4}^{a_4'}&=\frac{1}{\ell^2-1}\Big(\delta_{a_1,a_3}\delta_{a_2,a_4}\delta^{a_1',a_3'}\delta^{a_2',a_4'}+\delta_{a_1,a_4}\delta_{a_2,a_3}\delta^{a_1',a_4'}\delta^{a_2',a_3'}\Big)\\&\ \ -\frac{1}{\ell(\ell^2-1)}\left(\delta_{a_1,a_3}\delta_{a_2,a_4}\delta^{a_1',a_4'}\delta^{a_2',a_3'}+\delta_{a_1,a_4}\delta_{a_2,a_3}\delta^{a_1',a_3'}\delta^{a_2',a_4'}\right).
\end{split}
\end{equation*}
\end{prop}

Next the idea is to consider the symmetrized set of correlators:

\begin{defn}
The symmetrized  correlators are defined by
\begin{equation}
    \begin{split}
        \langle \mathscr{O}_{n} \mathscr{O}_{n} \widetilde{\mathscr{O}}_{n}\widetilde{\mathscr{O}}_{n}\rangle_{\texttt{SYM}}:=&\int_{U(\ell_n)} dU \ U_{a_1}^{a_1'}U_{a_2}^{a_2'}(\bar{U})_{a_3}^{a_3'}(\bar{U})_{a_4}^{a_4'}\ \langle \mathscr{O}_{n,a_1} \mathscr{O}_{n,a_2} \widetilde{\mathscr{O}}_{n,a_3}\widetilde{\mathscr{O}}_{n,a_4}\rangle,\\ 
        \langle \mathscr{O}_{\frac{n}{2}} \mathscr{O}_{\frac{n}{2}} \widetilde{\mathscr{O}}_{\frac{n}{2}}\widetilde{\mathscr{O}}_{\frac{n}{2}} \rangle_{\texttt{SYM}}:=&\int_{U(\ell_{n/2})} dU \ U_{b_1}^{b_1'}U_{b_2}^{b_2'}(\bar{U})_{b_3}^{b_3'}(\bar{U})_{b_4}^{b_4'}\ \langle \mathscr{O}_{\frac{n}{2},b_1} \mathscr{O}_{\frac{n}{2},b_2} \widetilde{\mathscr{O}}_{\frac{n}{2},b_3}\widetilde{\mathscr{O}}_{\frac{n}{2},b_4} \rangle,\\
        \langle \mathscr{O}_{n} \mathscr{O}_{\frac{n}{2}} \widetilde{\mathscr{O}}_{\frac{n}{2}}\widetilde{\mathscr{O}}_{n}\rangle_{\texttt{SYM}}:=&\int_{U(\ell_n)} dU \ U_{a_1}^{a_1'}(\bar{U})_{a_2}^{a_2'}\int_{U(\ell_{n/2})} dU \ U_{b_1}^{b_1'}(\bar{U})_{b_2}^{b_2'}\ \langle \mathscr{O}_{n,a_1} \mathscr{O}_{\frac{n}{2},b_1} \widetilde{\mathscr{O}}_{\frac{n}{2},b_2}\widetilde{\mathscr{O}}_{n,a_2}\rangle,\\
        \langle \mathscr{O}_{n} \mathscr{O}_{\frac{n}{2}} \widetilde{\mathscr{O}}_{n}\widetilde{\mathscr{O}}_{\frac{n}{2}} \rangle_{\texttt{SYM}}:=&\int_{U(\ell_n)} dU \ U_{a_1}^{a_1'}(\bar{U})_{a_2}^{a_2'}\int_{U(\ell_{n/2})} dU \ U_{b_1}^{b_1'}(\bar{U})_{b_2}^{b_2'}\ \langle \mathscr{O}_{n,a_1} \mathscr{O}_{\frac{n}{2},b_1} \widetilde{\mathscr{O}}_{n,a_2}\widetilde{\mathscr{O}}_{\frac{n}{2},b_2} \rangle.
    \end{split}
\end{equation}
\end{defn}

Following \cite{Kravchuk:2021akc}, we can argue that the symmetrized system of correlators has the same spectral gap as the original one. Thus any bound obtained from the symmetrized system is applicable to the unsymmetrized system of correlators. Furthermore, if we can find a functional for the unsymmetrized system of correlators, that would also work for the symmetrized one. Hence, it suffices to analyze the symmetrized system of correlators.  

\begin{defn}
Let us define the following quantities: 
\begin{equation}
\begin{aligned}
\strangefont{S}_{p;n}&:=\frac{1}{\ell_n(\ell_n+1)}\sum_{i=1}^{\ell_p}\sum_{a_1,a_2=1}^{\ell_n}\frac{1}{4}\left|f^{(p;i)}_{(n;a_1),(n;a_2)}+f^{(p;i)}_{(n;a_2),(n;a_1)}\right|^2,\\
\strangefont{A}_{p;n}&:=\frac{1}{\ell_n(\ell_n-1)}\sum_{i=1}^{\ell_p}\sum_{a_1,a_2=1}^{\ell_n}\frac{1}{4}\left|f^{(p;i)}_{(n;a_1),(n;a_2)}-f^{(p;i)}_{(n;a_2),(n;a_1)}\right|^2,\\
\strangefont{T}_{k;n}&:=\frac{1}{\ell_n^2-1}\sum_{i=1}^{d_k}\sum_{a_1,a_2=1}^{\ell_n}\left[c^{(k;i)}_{(n;a_1),(n;a_2)}(c^{(k;i)}_{(n;a_1),(n;a_2)})^*-\frac{1}{\ell_n}c^{(k;i)}_{(n;a_1),(n;a_1)}(c^{(k;i)}_{(n;a_2),(n;a_2)})^*\right],\\
\strangefont{Q}_{k;n}&:=\frac{1}{\ell_n}\sum_{i=1}^{d_k}c^{(k;i)}_{(n;a),(n;a)},\\
\strangefont{B}_{p;(n_1,n_2)}&:=\frac{1}{\ell_{n_1}\ell_{n_2}}\sum_{i=1}^{\ell_p}\sum_{a=1}^{\ell_{n_1}}\sum_{b=1}^{\ell_{n_2}}|f^{(p;i)}_{(n_1;a),(n_2;b)}|^2,\\
\strangefont{P}_{k;(n_1,n_2)}&:=\frac{1}{\ell_{n_1}\ell_{n_2}}\sum_{i=1}^{d_k}\sum_{a=1}^{\ell_{n_1}}\sum_{b=1}^{\ell_{n_2}}|s^{(p;i)}_{(n_1;a),(n_2;b)}|^2.  
\end{aligned}
\end{equation}
\end{defn}
\begin{rem}
    Except for $\strangefont{Q}_{k;n}$, all quantities defined above are positive semi-definite combinations of the OPE coefficients.
\end{rem}
\begin{prop}\label{crossEqnMultiple}
   Given a hyperbolic spin orbifold, for which, $\exists\ n\in\mathbb{N}$ with $n=1(\text{mod}\ 2)$ such that the orbifold does not support holomorphic modular forms with weight strictly smaller than $n$,
    the spectral identities derived from the set of symmetrized correlators defined above 
    are  given by:
    \newline $\forall m\in \mathbb{Z}_{\geqslant0}:$
    \begin{align*}
    \strangefont{S}_{2m+2n;n}&=\sum_{k=1}^\infty \strangefont{T}_{k;n}\mathcal{F}_{2m}^{(n,n,n,n)}(\lambda_k^{(0)}),\\
    - \strangefont{A}_{2m+1+2n;n}&=\sum_{k=1}^\infty \strangefont{T}_{k;n}\mathcal{F}_{2m+1}^{(n,n,n,n)}(\lambda_k^{(0)}),\\
\strangefont{S}_{2m+2n;n}&=\mathcal{F}_{2m}^{(n,n,n,n)}(0)+\sum_{k=1}^{\infty}\left(\strangefont{Q}_{k;n}^2-\frac{1}{\ell_n}\strangefont{T}_{k;n}\right)\mathcal{F}_{2m}^{(n,n,n,n)}(\lambda_k^{(0)}),\\
\strangefont{A}_{2m+1+2n;n}&=\mathcal{F}_{2m+1}^{(n,n,n,n)}(0)+\sum_{k=1}^{\infty}\left(\strangefont{Q}_{k;n}^2-\frac{1}{\ell_n}\strangefont{T}_{k;n}\right)\mathcal{F}_{2m+1}^{(n,n,n,n)}(\lambda_k^{(0)}),\\
 \strangefont{S}_{2m+n;n/2}&=\sum_{k=1}^\infty \strangefont{T}_{k;n/2}\ \mathcal{F}_{2m}^{(n/2,n/2,n/2,n/2)}(\lambda_k^{(0)}),\\
  - \strangefont{A}_{2m+1+n;n/2}&=\sum_{k=1}^\infty \strangefont{T}_{k;n/2}\ \mathcal{F}_{2m+1}^{(n/2,n/2,n/2,n/2)}(\lambda_k^{(0)}),\\
\strangefont{S}_{2m+n;n/2}&=\mathcal{F}_{2m}^{(n/2,n/2,n/2,n/2)}(0)+\sum_{k=1}^{\infty}\left(\strangefont{Q}_{k;n/2}^2-\frac{1}{\ell_{n/2}}\strangefont{T}_{k;n/2}\right)\mathcal{F}_{2m}^{(n/2,n/2,n/2,n/2)}(\lambda_k^{(0)}),\\
\strangefont{A}_{2m+1+n;n/2}&=\mathcal{F}_{2m+1}^{(n/2,n/2,n/2,n/2)}(0)+\sum_{k=1}^{\infty}\left(\strangefont{Q}_{k;n/2}^2-\frac{1}{\ell_{n/2}}\strangefont{T}_{k;n/2}\right)\mathcal{F}_{2m+1}^{(n/2,n/2,n/2,n/2)}(\lambda_k^{(0)}),\\
\strangefont{B}_{m+3n/2;(n,n/2)}&=\mathcal{F}_{m}^{(n,n/2,n/2,n)}(0)+\sum_{k=1}^{\infty}\strangefont{Q}_{k;n}\strangefont{Q}_{k;n/2}\ \mathcal{F}_{m}^{(n,n/2,n/2,n)}(\lambda_k^{(0)}),\\
(-1)^m\strangefont{B}_{m+3n/2;(n,n/2)}&=\sum_{k=1}^{\infty}\strangefont{P}_{k;(n,n/2)}\ \mathcal{F}_{m}^{(n,n/2,n,n/2)}(\lambda_k^{(1/2)})\\
&\quad\quad +\frac{\ell_{n/2}+1}{\ell_n}\strangefont{S}_{n;n/2}\ \mathcal{F}_{m}^{(n,n/2,n,n/2)}\left(\frac{n}{2}\left(1-\frac{n}{2}\right)\right).\\
\end{align*}
\end{prop}
\begin{SDP}[Spectral bound on Laplacian and Dirac: III]\label{linpro3}
Let us consider a hyperbolic spin orbifold which supports $\ell_n$ different holomorphic modular forms with weight $2n$ and $\ell_{n/2}$ different holomorphic modular forms with weight $n$. Assume that the orbifold does not support any holomorphic modular form with weight strictly smaller than $n$.\\

Given $\Lambda\in\mathbb{N}$, $\lambda_*^{(0)},\lambda_*^{(1/2)}\in\mathbb{R}_{+}$, if there exist $\alpha_i\in\mathbb{R}^{\Lambda+1}$ with $i=1,2,3,4,5,6$ such that
\\
\begin{enumerate}
\item 
$\sum_{m=0}^\Lambda \alpha_{2;m}\ \mathcal{F}_m^{(n,n,n,n)}(0)+\alpha_{4;m}\ \mathcal{F}_m^{(n/2,n/2,n/2,n/2)}(0)+\alpha_{5;m}\ \mathcal{F}_m^{(n,n/2,n/2,n)}(0)=1$,
\\
\item 
$\alpha_{1;2m}+\alpha_{2;2m}\leqslant 0, \ \ \forall m\in \mathbb{Z}_{\geqslant 0},2m\leqslant \Lambda$,
\\
\item 
$\alpha_{3;2m}+\alpha_{4;2m}-\delta_{m,0}\frac{\ell_{n/2}+1}{\ell_n}\sum_{m'=0}^{\Lambda}\alpha_{6;m'}\ \mathcal{F}_{m'}^{(n,n/2,n,n/2)}\left(\frac{n}{2}\left(1-\frac{n}{2}\right)\right)\leqslant 0,\\
\ \ \ \ \ \ \forall m\in \mathbb{Z}_{\geqslant 0},2m\leqslant\Lambda$,
\\
\item 
$\alpha_{2;2m+1}-\alpha_{1;2m+1}\leqslant0,\ \ \forall m\in \mathbb{Z}_{\geqslant 0},2m+1\leqslant \Lambda$,
\\
\item 
$\alpha_{4;2m+1}-\alpha_{3;2m+1}\leqslant0,\ \ \forall m\in \mathbb{Z}_{\geqslant 0},2m+1\leqslant \Lambda$,
\\
\item 
$ \alpha_{5;m}+(-1)^m\alpha_{6;m}\leqslant0,\ \ \forall m\in \mathbb{Z}_{\geqslant 0},m\leqslant \Lambda$,
\\
\item 
\begin{equation}\label{positiveTChannel1}\sum_{m=0}^\Lambda\left(\alpha_{1;m}-\frac{1}{\ell_n}\alpha_{2;m}\right)\ \mathcal{F}_m^{(n,n,n,n)}(\lambda_k^{(0)})\geqslant 0,\ \ \forall \lambda_k^{(0)}\geqslant\lambda_*^{(0)},\end{equation}
\\
\item 
\begin{equation}\label{positiveTChannel2}\sum_{m=0}^\Lambda\left(\alpha_{3;m}-\frac{1}{\ell_{n/2}}\alpha_{4;m}\right)\ \mathcal{F}_m^{(n/2,n/2,n/2,n/2)}(\lambda_k^{(0)})\geqslant 0,\ \ \forall \lambda_k^{(0)}\geqslant\lambda_*^{(0)},\end{equation}
\\
\item 
\begin{equation}\label{positiveTChannel3}\sum_{m=0}^\Lambda\alpha_{6;m}\ \mathcal{F}_m^{(n,n/2,n,n/2)}(\lambda_k^{(1/2)})\geqslant 0, \ \ \forall \lambda_k^{(1/2)}\geqslant\lambda_*^{(1/2)},\end{equation}
\\
\item 

$\forall \lambda_k^{(0)}\geqslant \lambda_*^{(0)},$\\

\begin{equation}\label{positiveTChannel4}
M'(\lambda_k^{(0)}):=\begin{pmatrix} 
\sum_{m=0}^\Lambda \alpha_{2;m} \mathcal{F}_m^{(n,n,n,n)}(\lambda_k^{(0)})&\frac{1}{2}\sum_{m=0}^\Lambda \alpha_{5;m} \mathcal{F}_m^{(n,n/2,n/2,n)}(\lambda_k^{(0)})\
\\
\frac{1}{2}\sum_{m=0}^\Lambda \alpha_{5;m} \mathcal{F}_m^{(n,n/2,n/2,n)}(\lambda_k^{(0)})&
\sum_{m=0}^\Lambda \alpha_{4;m} \mathcal{F}_m^{(n,n,n,n)}(\lambda_k^{(0)})\end{pmatrix} \succeq 0,\end{equation}
\end{enumerate}
then at least one of the following statements must be true:
\\
\begin{enumerate}

    \item  $\lambda_1^{(0)}<\lambda_*^{(0)}$,
    \\
     \item $\lambda_1^{(1/2)}<\lambda_*^{(1/2)}$.
\end{enumerate}
\end{SDP}
\begin{proof}
Taking appropriate linear combinations of \eqref{crossEqnMultiple} and using the non-existence of holomorphic modular forms with weight strictly below $n$, we can derive the follwoing identity: 
\begin{align*}
    \sum_{m=0}^{\lfloor \Lambda/2\rfloor} \Big[&(\alpha_{1;2m}+\alpha_{2;2m})\strangefont{S}_{2m+2n;n}\\
    &+\big(\alpha_{3;2m}+\alpha_{4;2m}-\delta_{2m,0}\frac{\ell_{n/2}+1}{\ell_n}
    \times \sum_{m'=0}^{\Lambda}\alpha_{6;m'}\mathcal{F}_{m'}^{(n,n/2,n,n/2)}\big(\frac{n}{2}(1-\frac{n}{2})\big)\big)\strangefont{S}_{2m+n;n/2}\\
    &+(\alpha_{2;2m+1}-\alpha_{1;2m+1})\strangefont{A}_{2m+1+2n;n}+(\alpha_{4;2m+1}-\alpha_{3;2m+1})\strangefont{A}_{2m+1+n;n/2}\Big]\\
    &+\sum_{m=0}^{\Lambda}\left(\alpha_{5;m}+(-1)^m \alpha_{6;m}\right)\strangefont{B}_{m+3n/2;(n,n/2)}\\
    &-\sum_{m=0}^\Lambda \left( \alpha_{2;m}\ \mathcal{F}_m^{(n,n,n,n)}(0)+\alpha_{4;m}\ \mathcal{F}_m^{(n/2,n/2,n/2,n/2)}(0)+\alpha_{5;m}\ \mathcal{F}_m^{(n,n/2,n/2,n)}(0)\right)\\
&=\sum_{m=0}^\Lambda \sum_{k=1}^{\infty} \Big[\strangefont{T}_{k,n} \big(\alpha_{1;m}-\frac{1}{\ell_n}\alpha_{2;m}\big)\mathcal{F}_m^{(n,n,n,n)}(\lambda_k^{(0)})\\&+\strangefont{T}_{k;n/2}\left(\alpha_{3;m}-\frac{1}{\ell_{n/2}}\alpha_{4;m}\right)\mathcal{F}_m^{(n/2,n/2,n/2,n/2)}(\lambda_k^{(0)})\\
&+\strangefont{P}_{k;(n,n/2)}\ \alpha_{6;m}\ \mathcal{F}_m^{(n,n/2,n,n/2)}(\lambda_k^{(1/2)})+\begin{pmatrix}\strangefont{Q}_{k;n}&\strangefont{Q}_{k;n/2}\end{pmatrix}M'(\lambda_k^{(0)})\begin{pmatrix}\strangefont{Q}_{k;n}\\\strangefont{Q}_{k;n/2}\end{pmatrix}
\Big].
\end{align*}
Now using the properties of $\alpha$, the L.H.S of the above identity is strictly negative, so should be the R.H.S. Using \eqref{positiveTChannel1}
, \eqref{positiveTChannel2}, \eqref{positiveTChannel3}, \eqref{positiveTChannel4}, the proposition follows.
\end{proof}

\section{Estimates of spectral gaps from the Selberg trace formula}
\label{sec:selberg}

In this section, we use the Selberg trace formula to estimate the spectral gap of a given surface $X$. 
The standard approach for numerically computing the operator spectrum on a hyperbolic surface is the finite element method. Thanks to the \texttt{FreeFEM++} package, the spectrum of Laplacian operator, with periodic and twisted periodic boundary conditions (see supplementary material of \cite{HyperbolicBandTheory}), is known for many surfaces. For our purposes, we look at a different approach to compute the gaps of the Laplacian and the Dirac operator of explicit orbifolds and surfaces. This will allow us to compare these spectra to the previously established bootstrap bounds. The key idea is to use the Selberg trace formula, with the geodesic spectrum as an input, to rule out points on the real line that are not in the eigenvalue spectrum. Our algorithm can be divided into three steps:
\begin{itemize}
    \item Find a Dirichlet domain for the surface of interest and explicit generators for the fundamental group. 
    \item Use this Dirichlet domain to produce a complete non-redundant list of conjugacy classes of geodesics on the surface whose length is bounded by a given value $L$.
    \item Choose a compactly supported test function that yields an estimate of the first eigenvalue of the Laplace/Dirac spectrum when substituted into the Selberg trace formula. 
\end{itemize}
Note that the first two steps of this algorithm work in the same way in order to produce the Laplace and Dirac spectra. Only the explicit form of the Selberg trace formula used in the last step differs.

\subsection{Step 1: Constructing the Dirichlet domain}
We start by reviewing some basic definitions.
\begin{defn}(Dirichlet domain of a Fuchsian group) Let $\Gamma$ be a Fuchsian group and let $x\in \mathbb{H}$ be a point not fixed by any element of $\Gamma$ other than the identity. A \textit{Dirichlet domain with base point $x$} is the set
\begin{equation}
D(x):=\left\{ p\in \mathbb{H}\ |\ \forall \gamma\in \Gamma\backslash\{e\},d(x,p)<d(x,\gamma p) \right\}.
\end{equation}
\end{defn}
For a cocompact Fuchsian group, its Dirichlet domain is a hyperbolic polygon with finite area. The boundary of this hyperbolic polygon consists of finitely many geodesic segments. We will denote this set of geodesic segments by $S(D)$ and it satisfies the  property \cite{voight2009computing} that for every $s\in S(D)$, there exists a unique $s'\in S(D)$ and a unique $\gamma_s \in \Gamma$ such that $s'=\gamma_s\ s$. This is known as the \textit{side-pairing} property. The segment $s$ is the perpendicular bisector of the geodesic between $x$ and $\gamma_s\ x$. The set $G(S):=\{\gamma_s|s\in S(D)\}$ is a generating set for $\Gamma$ and we will refer to its elements as \textit{side-paring generators}. See \cite{voight2009computing} for a systematic algorithm computing $G(S)$ and its application to arithmetic Fuchsian groups.
\begin{defn}(Maximin edge distance of the Dirichlet domain)
 Let $\Gamma$ be a cocompact Fuchsian group with Dirichlet domain $D(x)$ where $x$ is the base point. Let $S(D)$ be the set of geodesic segments forming the boundary of $D(x)$. The \textit{maximin edge distance $R(D)$} is defined as the maximum over all the domain's edges of the minimum distance from the edge to the base-point $x$:
    \begin{equation}
   R(D):= \max_{s\in S(D)}\left(\inf_{p\in s} d(x,p)\right).
    \end{equation}
\end{defn}

We briefly summarize the Dirichlet domain and the corresponding side-paring generators for surfaces and orbifolds relevant to us:
\begin{enumerate}
    \item Hyperbolic triangles:
    \\
   As shown on Figure \ref{fig:Dirichlet}, the Dirichlet domain of a genus-0 hyperbolic triangle $[0;p,q,r]$ consists of two geodesic triangles $\Delta(p,q,r)$ with the four vertices located at $\hat{s}(1/p,1/q,1/r;0)$, $\hat{s}(1/p,1/q,1/r;1)$ and $\pm \hat{s}(1/p,1/q,1/r;\infty)$, where $\hat{s}(1/p,1/q,1/r;z)$ is the rescaled Schwarz function \cite{harmer_2005}
\begin{align}
   \hat{s}(\alpha,\beta,\gamma;z)=\nu s(\alpha,\beta,\gamma;z),\ \  s(\alpha,\beta,\gamma;z)=z^\alpha \frac{_2F_1(a^\prime,b^\prime,c^\prime;z)}{_2F_1(a,b,c;z)},
\end{align}
\begin{align}
   \nu=\sqrt{\frac{\cos (\pi \alpha+\pi \beta)+\cos (\pi \gamma)}{\cos (\pi \alpha-\pi \beta)+\cos (\pi \gamma)} \cdot \frac{\cos (\pi \alpha-\pi \beta-\pi \gamma)+1}{\cos (\pi \alpha+\pi \beta+\pi \gamma)+1}} \cdot \frac{\Gamma\left(a^{\prime}\right) \Gamma\left(b^{\prime}\right)}{\Gamma\left(c^{\prime}\right)} \cdot \frac{\Gamma(c)}{\Gamma(a) \Gamma(b)},
\end{align}
with 
\begin{align}
    a=\frac{1-\alpha-\beta-\gamma}{2},\;\;\;b=\frac{1-\alpha+\beta-\gamma}{2},\;\;\;c=1-\alpha,
\end{align}
and
\begin{align}
    a^\prime=\frac{1+\alpha-\beta-\gamma}{2},\;\;\;b^\prime=\frac{1+\alpha+\beta-\gamma}{2},\;\;\;c^\prime=1+\alpha.
\end{align}
\item The Bolza surface:
\\
The Dirichlet domain of the Bolza surface is a regular octagon. The side-paring generators are $\gamma_1,\gamma_2,\gamma_3,\gamma_4,\gamma_1^{-1},\gamma_2^{-1},\gamma_3^{-1},\gamma_4^{-1}$ with 
\\
\begin{align}
\gamma_k=
\pm \begin{pmatrix}
 \cosh \left(l/2\right) & e^{\frac{\mathrm{i} \pi  k}{4}} \sinh \left(l/2\right)
   \\
 e^{-\frac{\mathrm{i} \pi k}{4}} \sinh \left(l/2\right) & \cosh
   \left(l/2\right) \\
\end{pmatrix}
,\ \ \ l=2 \cosh ^{-1}\left(\cot \left(\frac{\pi }{8}\right)\right).
\end{align}
\\
Here we used ``$\pm$'' to stress the fact that $\gamma_k$'s are elements in $\mathrm{PSU}(1,1)$. The sign ambiguity will be removed once the surface is given a spin structure.
\item The most symmetric point in the moduli space with signature [1;3]:
\\
The Dirichlet domain for the one-punctured torus is an octagon with side-paring generators being $\gamma_1,\gamma_2,\gamma_1 \gamma_2,\gamma_2\gamma_1,\gamma_1^{-1},\gamma_2^{-1},\gamma_1^{-1} \gamma_2^{-1},\gamma_2^{-1}\gamma_1^{-1}$ where
\begin{equation}\gamma_1=
\begin{pmatrix}
\frac{1-x^2 y^2}{\sqrt{\left(1-x^2\right) \left(1-y^2\right) \left(x^2 y^2+1\right)}} &
   \frac{-x \left(1-y^2\right)+\mathrm{i} \left(1-x^2\right) y}{\sqrt{\left(1-x^2\right)
   \left(1-y^2\right) \left(x^2 y^2+1\right)}} \\
   \\
 \frac{-x \left(1-y^2\right)-\mathrm{i} \left(1-x^2\right) y}{\sqrt{\left(1-x^2\right)
   \left(1-y^2\right) \left(x^2 y^2+1\right)}} & \frac{1-x^2
   y^2}{\sqrt{\left(1-x^2\right) \left(1-y^2\right) \left(x^2 y^2+1\right)}} \\
\end{pmatrix},
\end{equation}
\\
\begin{equation}\gamma_2=
\begin{pmatrix}
\frac{1-x^2 y^2}{\sqrt{\left(1-x^2\right) \left(1-y^2\right) \left(x^2 y^2+1\right)}} &
   \frac{x \left(1-y^2\right)+\mathrm{i} \left(1-x^2\right) y}{\sqrt{\left(1-x^2\right)
   \left(1-y^2\right) \left(x^2 y^2+1\right)}} \\
   \\
 \frac{x \left(1-y^2\right)-\mathrm{i} \left(1-x^2\right) y}{\sqrt{\left(1-x^2\right)
   \left(1-y^2\right) \left(x^2 y^2+1\right)}} & \frac{1-x^2
   y^2}{\sqrt{\left(1-x^2\right) \left(1-y^2\right) \left(x^2 y^2+1\right)}} \\
\end{pmatrix},
\end{equation}
with 
\begin{equation}
x=\tanh \left[\frac{1}{2} \cosh ^{-1}\left(\frac{\cos \left(\pi /9 \right)}{ \sin
   \left(\pi/18\right)}\right)\right], \ \ y=\tanh \left[\frac{1}{2} \cosh ^{-1}\left(\frac{\cos \left(\pi /18 \right)} {\sin
   \left(\pi/9 \right)}\right)\right].
\end{equation}
\end{enumerate}


We conclude this subsection with a few comments on generalizations to hyperbolic surfaces with higher genus. We will sketch a general method for constructing a fundamental domain and to look for explicit  expressions of the generators. The surfaces or orbifolds that are most likely to saturate the numerical bounds are conjectured to always possess a large number of symmetries, as we already observed in genus 1 and genus 2. For these special surfaces, the following result holds:

\begin{prop}[\cite{https://doi.org/10.48550/arxiv.1711.06599}]
Let $\Sigma$ be a compact hyperbolic Riemann surface of genus $g>1$. $\Sigma$ realizes a local maximum of the number of automorphisms on the moduli space of Riemann surfaces of genus $g$ if and only if it is isomorphic to a quotient of the upper-half plane by a torsion-free normal subgroup of a cocompact Fuchsian triangle group.
\end{prop}

This result means that in the relevant cases for our purposes, all fundamental domains can be constructed from gluing geodesic triangles $\Delta(p,q,r)$\footnote{The geodesic triangles are also the fundamental domains of the triangle group $T(p,q,r)=\left\langle a, b, c \mid a^2=b^2=c^2=(a b)^p=(b c)^q=(c a)^r=1\right\rangle$. The generators $a$, $b$, $c$ are reflections against the three edges.} together.
Hence, the very general problem of finding an explicit fundamental domain for a given Riemann surface maps to a much simpler one in our case: the problem of 

\begin{enumerate}
\item Constructing a tiling of the hyperbolic plane by triangles of given angles, and 
\item Gluing geodesic triangles to obtain fundamental domains of the surface/orbifold of interest.
\end{enumerate}



\begin{figure}[!ht]
    \centering
    \subfloat[\centering ]{{\includegraphics[width=6cm]{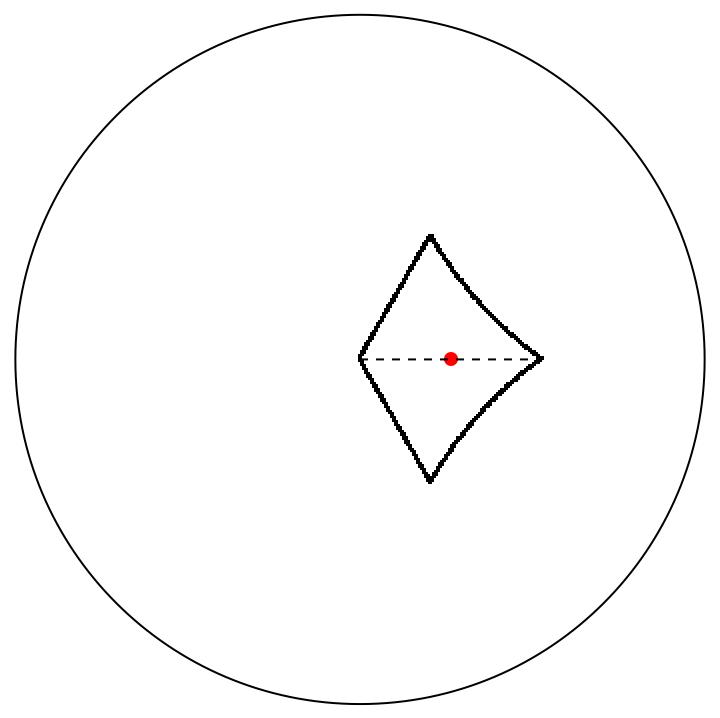} }}%
    \qquad
    \subfloat[\centering ]{{\includegraphics[width=6cm]{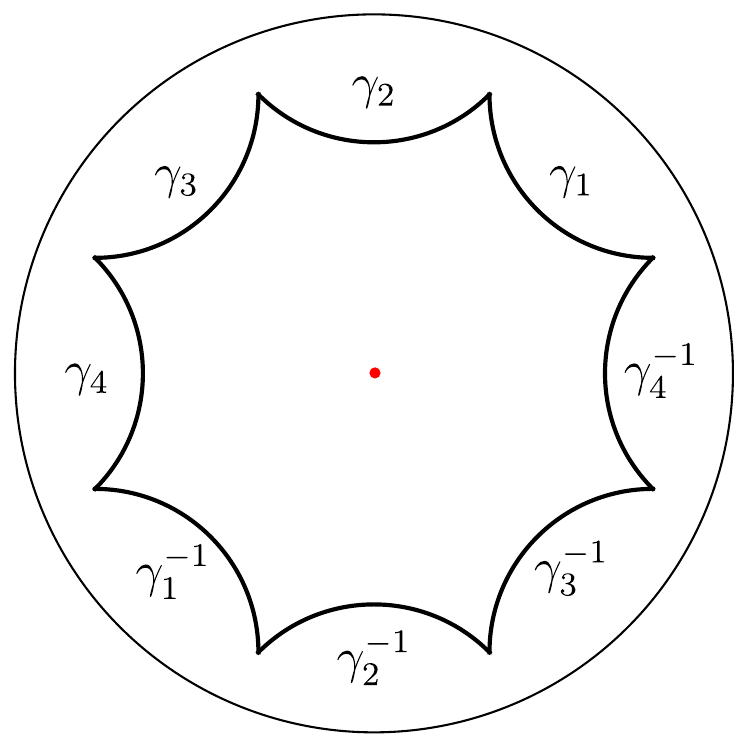} }}%
    \qquad
    \subfloat[\centering ]{{\includegraphics[width=6.5cm]{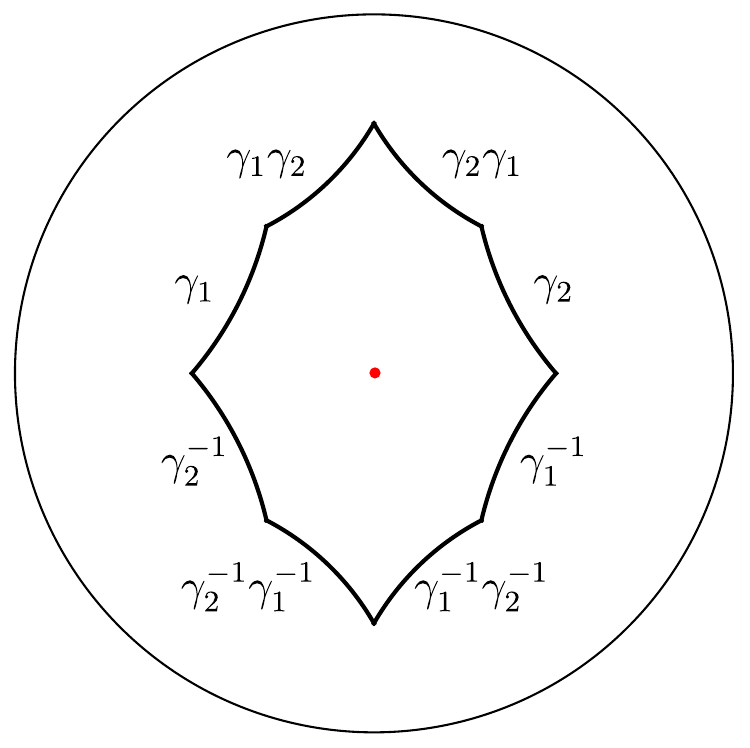} }}%
    \caption{(A) The Dirichlet domain for the hyperbolic triangle $[0;3,3,5]$ consists of two geodesic triangles $\Delta(3,3,5)$ (B) The Dirichlet domain for the Bolza surface. Each side of the Dirichlet domain is labeled by the corresponding side-pairing generator. (C) The Dirichlet domain of the most symmetric point in the moduli space of one punctured torus with signatrue [1;3]. Each side of the Dirichlet domain is labeled by the corresponding side-pairing generator.
     The red dots in the figures are the corresponding base points.}
    \label{fig:Dirichlet}
\end{figure}

To find explicit expressions for the generators, the simplest method would be to directly search for a presentation of a normal subgroup $N$ of the triangle group $T(p,q,r)$ such that 
\begin{align}
\mathrm{Aut}(\Sigma)\cong T(p,q,r)/N.
\end{align}
This can in principle be done using the GAP software, although there could be cases in which there is more than one possible choice. We reserve a more careful examination of this problem to future work. 

Once this subgroup is found, one can construct a fundamental domain by singling out one representative triangle per coset in $T/N$ and taking the union of all such representatives. Note that the obtained fundamental domain is made of exactly $|\mathrm{Aut}(\Sigma)|$ triangles.

\subsection{Step 2: Enumerating the closed geodesics of length bounded by $L$}

For this second step, we will mostly adapt the method proposed in \cite{em/1048515809} to our case. Before stating the algorithm, we first summarize the ideas behind it:
\begin{itemize}
\item In every hyperbolic conjugacy class, there is at least one representative whose axis\footnote{Recall that a hyperbolic element $\gamma$ can be specified by the geodesic it fixes. This geodesic is known as the axis of $\gamma$.} passes within a distance $R(D)$ of the basepoint, where $R(D)$ is the maximin edge distance of the Dirichlet domain. (See Propositions 3.2 and 3.3 from \cite{em/1048515809}.)
\item Working with such representatives, it is intuitive to see that we only need to tile the Poincaré disk $\mathbb{D}$ up to a certain radius in order to exhaust all hyperbolic conjugacy classes below a certain length cutoff. This radius depends both on the length cutoff and the size of the Dirichlet domain $D$. To be more precise, it suffices to find all translates $\gamma D$ such that $d(x,\gamma x)\leqslant 2 \cosh^{-1}(\cosh{R(D)}\cosh{\frac{L}{2}})$. (see Proposition 3.5 of \cite{em/1048515809}.)
\item If two hyperbolic elements with  geodesic length $L$ are conjugate to each other and both of their axes pass within a distance $R(D)$ from the base point, then the conjugating element $\gamma$ must satisfy $d(x,\gamma x)\leqslant 2 \cosh^{-1}(\cosh{R(D)}\cosh{\frac{L}{4}})$ (see Proposition 3.7 of \cite{em/1048515809}.)
\end{itemize}
With these basic ideas in mind, we state the algorithm:
\begin{enumerate}
\item (Initialize) Set $T$ and $T_h$ to be two empty lists. Given a spin structure, write all side-paring generators $\tilde{\gamma}_i$ of $\widetilde{\Gamma}$ as $2\times 2$ matrices .
\item (Enumerate translates) Let $T_n$ be the list of irreducible $n$-letter words formed by the $\tilde{\gamma}_i$. Remove from $T_n$ all elements $\tilde{\tau}$ such that $d(x,\tilde{\tau} x)>2 \cosh^{-1}(\cosh R(D)\cosh \frac{L}{2})$. Here $x$ is the base point of the Dirichilet domain. If $T_n$ becomes an empty list, terminate this step. Otherwise append $T_n$ to $T$. To generate $T_{n+1}$, multiply every element in $T_n$ by side-paring generators $\tilde{\gamma}_i$ and remove the redundant matrices that are already in $T$.\footnote{The enumeration procedure would fail if there exist some translate $\gamma D$ such that $\gamma x$ is closer to $x$ than all of its neighboring tiles, but this won't happen if we work with Dirichlet domain instead of an arbitrary fundamental domain. See Proposition 3.1 of \cite{em/1048515809}.} This procedure terminates after a finite number of steps. 
\item (Count hyperbolic conjugacy classes) Given the list $T$, let $T_h$ be the subset of hyperbolic elements whose axis passes within a distance $R(D)$ from the base point. Group $T_h$ by 
geodesic length (or equivalently, absolute value of the traces of the corresponding $2\times 2$ matrices):
\begin{equation}
T_h=\bigsqcup_l T_h^{(l)},\ \ T_h^{(l)}=\{\tilde{\gamma} \in T_h \ |\ 2\cosh^{-1}\left(|\trace(\tilde{\gamma})|/2\right)=l\}.
\end{equation}
For all sublists of $T_h$ sharing the same geodesic length $l$, select all $\tilde{\tau}\in T$ such that $d(x,\tilde{\tau} x)\leqslant 2 \cosh^{-1}(\cosh R(D) \cosh \frac{l}{4})$:
\begin{equation}
C_l:=\left\{\tilde{\tau}\in T\ \Big |\ d(x,\tilde{\tau} x)\leqslant 2\cosh^{-1}\left(\cosh{R}\cosh \frac{l}{4}\right) \right\},
\end{equation}
and use these elements to partition the sublist into conjugacy classes:
\begin{equation}
\text{Orb}_l(\tilde{\gamma}):=\{\widetilde{\gamma}' \in T_h^{(l)}\ | \  \exists \tilde{\tau} \in C_l, \tilde{\tau} \tilde{\gamma} \tilde{\tau}^{-1}=\pm \widetilde{\gamma}' \},\ \ \ T_h^{(l)}=\bigsqcup_{i}\text{Orb}_l(\tilde{\gamma}_i).
\end{equation}
Note that if the orbifold $\Gamma \backslash \mathbb{H}$ is spin, the traces of all matrices in the same conjugacy class should share the same sign. This sign is exactly the multiplier $\chi$ that enters the sum over hyperbolic elements in the Selberg trace formula for the 1-Laplacian. 
\item (Compute the winding) Compute the winding number of each conjugacy class by checking whether it could be powers of other classes.
\end{enumerate}

\subsection{Step 3: Applying the Selberg trace formula to a well-chosen test function}
\label{sec:step3} 

Now that our algorithm can find all the geodesics up to a certain length, we can plug them into the geometric side of the Selberg trace formula. The Selberg trace formulae for the Laplacian and the squared Dirac operator (which is related to 1-Laplacian, see the remark \ref{remark:1-Laplacian}) are respectively (see Theorem 5.1 of \cite{hejhal2006selberg1} and equation 6.56 of \cite{hejhal2006selberg2}):
\begin{align*}
    \sum_{n=0}^\infty \widehat{h}(\sct_n)=\frac{V(\Sigma)}{4\pi}\int_{-\infty}^\infty t\widehat{h}(t)\mathrm{tanh}(\pi t) dt + \sum_{\{\gamma_h\}}\frac{l(\gamma_{h;0})}{2\,\mathrm{sinh}(l(\gamma_h)/2)}h(l(\gamma_h))\\+\sum_{\{\gamma_e\}}\frac{ie^{-i\theta}}{4M_{\gamma_e}\mathrm{sin}\theta}\int_{-\infty}^\infty\frac{h(u)e^{-\frac{u}{2}}(e^u-e^{2i\theta})}{\mathrm{cosh}u-\mathrm{cosh}(2\theta)}du,
\end{align*}
and
\begin{align*}
    \sum_{n=0}^\infty \widehat{h}(\spt_n)=\frac{V(\Sigma)}{4\pi}\int_{-\infty}^\infty t\widehat{h}(t)\frac{\mathrm{sinh}(2\pi t)}{\mathrm{cosh}(2\pi t)-1} dt + \sum_{\{\gamma_h\}, \mathrm{Tr}\,\gamma_h>2}\frac{\chi(\gamma_h)l(\gamma_{h;0})}{2\,\mathrm{sinh}(l(\gamma_h)/2)}h(l(\gamma_h))\\+\sum_{\{\gamma_e\}, \pi>\theta(\gamma_e)>0}\frac{i\chi(\gamma_e)}{4M_{\gamma_e}\mathrm{sin}\theta}\int_{-\infty}^\infty\frac{h(u)(e^u-e^{2i\theta})}{\mathrm{cosh}u-\mathrm{cosh}(2\theta)}du.
\end{align*}

Let us explain the notations in these formulas:
\begin{itemize}
    \item $h$ is an acceptable compactly supported test function such that $\widehat{h}$ is even.
    \item The $\sct_n$ and $\spt_n$ are equal to the $\sqrt{\sclambda_n-1/4}$ and $\sqrt{\splambda_n-1/4}$, where the $\sclambda_n$ and $\splambda_n$ are the eigenvalues of the Laplacian and 1-Laplacian, respectively.
    \item $V(\Sigma)$ is the volume of $\Sigma$.
    \item The sum over $\{\gamma_h\}$ is a sum over conjugacy classes of hyperbolic elements in $\pi_1(\Sigma)$, or of its double cover in $\mathrm{SL}(2,\mathbb{R})$ in the case of the 1-Laplacian. 
   
    \item For a given hyperbolic element $\gamma_h$, $l(\gamma_h)=2 \cosh^{-1}(|\trace(\gamma_h)|/2)$ is the length of the corresponding geodesic.
     \item For a given hyperbolic element $\gamma_h$, $\gamma_{h;0}$ denotes the primitive hyperbolic element associated with $\gamma_h$, that is, $\gamma_{h;0}$ is the element with shortest geodesic length such that $\gamma_h$ can be expressed as a power of $\gamma_{h;0}$.
    \item The sum over $\{\gamma_e\}$ is a sum over conjugacy classes of elliptic elements in $\pi_1(\Sigma)$, or of its double cover in $\mathrm{SL}(2,\mathbb{R})$ in the case of the 1-Laplacian. 
    \item For a given elliptic element $\gamma_e$, $M_{\gamma_e}$ is the order of $\gamma_e$. 
    \item $\chi$ is a multiplier on the double cover of $\pi_1(\Sigma)$ in $\mathrm{SL}(2,\mathbb{R})$ implementing a choice of spin structure. 
    \item $\theta(\gamma_e)$ is the angle of the rotation matrix $\gamma_e$. 
    \item A hat denotes a Fourier transform.
\end{itemize}

Our strategy will closely follow the reference \cite{lin2022seiberg}. The idea is to apply the Selberg trace formula to functions whose Fourier transform  that appears on the geometric side of the Selberg trace formula has compact support - so that only a finite number of terms contribute to the sum over hyperbolic elements, and that we can find all of them using the procedure outlined in the previous subsection.

Let $h_1,\dots,h_n$ be such test functions with support in $[-r,r]$ and $\lambda_i$ be the eigenvalues of the operator of interest (it can be the Laplacian or the squared Dirac operator). It will also be useful to introduce extra parameters $t_i$, related to the eigenvalues through
\begin{align}
\lambda_i:=\frac{1}{4}+t_i^2.
\end{align}
Define a matrix $A$ whose entries are given by \begin{align}A_{ab}=\sum_j h_a(t_j)h_b(t_j).\end{align}If $c_t$ is the column vector \begin{align} c_t=\begin{bmatrix}h_1(t)\\ \vdots\\h_n(t)\end{bmatrix},\end{align}
let \begin{align}I_r(\lambda):=\underset{\braket{c_t,x}=1}{\mathrm{inf}}\braket{Ax,x}.\end{align}
$I_r(\lambda)$ can be calculated by introducing a Lagrange multiplier, and is equal to
\begin{align}
    I_r(\lambda)=\frac{1}{\braket{A^{-1}c_t,c_t}}.
\end{align}
This quantity can be evaluated explicitly using the geometric side of the Selberg trace formula applied to the functions $h_ah_b$. 
One can then prove (see \cite{lin2022seiberg}): 
\begin{prop}
\label{prop:jt1}
    If $I_r(\lambda)<1$, then $\lambda$ is not an eigenvalue of the operator of interest.
\end{prop}
\begin{proof}
A test function $h$ with support on $[-r,r]$ is called \textit{admissible} if $\widehat{h}\geq 0$ and the Selberg trace formula of interest can be applied to it. For $t\in[-r,r]$, define the quantity
\begin{align}
I^\bullet_r(\lambda):=\underset{\begin{subarray}{c}
  h\;\mathrm{admissible}\\\mathrm{supp}\,h\subset[-r,r]\\\widehat{h}(t)=1\end{subarray}}{\mathrm{inf}}\sum_j \widehat{h}(t_j).
\end{align}
Then, if there exists $j_0$ such that $t=t_{j_0}$, then for any test function $\widehat{h}$, we have \begin{align}1=\widehat{h}(t)\leq\sum_j \widehat{h}(t_j).\end{align}This is just because a sum is larger than or equal to all of its summands if the summands are non-negative. Hence,  
\begin{equation}
I^\bullet_r(\lambda) \geqslant 1.
\end{equation}
By contrapositive, we deduce that if $I^\bullet_r(\lambda)<1$, then $t$ cannot be one of the $t_j$. All that now remains to be shown is that for all $t$, $I^\bullet_r(\lambda)\leq I_r(\lambda)$.
Now, $I_r(\lambda)$ can itself be reexpressed as an infimum: let
\begin{align}
S:=\left\{h\ast h, \mathrm{supp}\; h\subset\left[-\frac{r}{2},\frac{r}{2}\right],h=\sum x_i h_i\; \mathrm{for}\;x_i\in\mathbb{R}\right\}.
\end{align}
We have 
\begin{align}
\underset{\begin{subarray}{c}
  H=h\ast h\in S\\\widehat{h}(t)=1\end{subarray}}{\mathrm{inf}}\sum_j \widehat{h}(t_j)=I_r(\lambda).
\end{align}
But now, $I^\bullet_r(\lambda)$ is also an infimum, over a larger set of functions. We deduce
\begin{align}
I^\bullet_r(\lambda)\leq I_r(\lambda),
\end{align}
which completes the proof of the proposition.
\end{proof}
There is a straightforward generalization of the proof of the above result to the case in which the eigenvalue of the operator of interest is degenerate:
\begin{prop}
\label{prop:jt2}
Let $k\in\mathbb{N}$. If $I_r(\lambda)<k$, then $\lambda$ is not an eigenvalue of multiplicity $k$ or higher of the operator of interest.
\end{prop}

These results provide a method to numerically estimate a lower bound on the gap of the Laplace and Dirac operators on a hyperbolic surface or orbifold. The method works as follows:
\begin{enumerate}
    \item Calculate the matrix $A_{ab}$ for some compactly supported test functions $h_a$ and $h_b$ using the Selberg trace formula and the geodesic spectrum generated in the previous step.
    \item Deduce the value of $I_r(\lambda)$ as a function of $t=\sqrt{\lambda-\frac{1}{4}}$.
    \item For Laplace operators, find the first interval $[a_1,b_1]\subset\mathbb{R}_{+}$ such that $I_r(\lambda)\geq 1$ when $\lambda\in [a_1,b_1]$. This interval $[a_1,b_1]$ provides an estimate for $\lambda_1^{(0)}$ of the surface or orbifold under consideration .
    In the case of the 1-Laplacian, all eigenvalues above $1/4$ must have even multiplicity due to Kramers degeneracy, see Proposition \ref{prop:kramers}. Hence, the interval containing $\lambda_1^{(1/2)}$ is the first interval $[c_1,d_1]\in \mathbb{R}_{+}$ such that $I_r(\lambda)\geq 2$ when $\lambda\in [c_1,d_1]$.
\end{enumerate}

The accuracy of our method is dependent on the choice of test functions $h_k$. In order to get a good bound, we need to slightly modify the choices of test functions made in \cite{lin2022seiberg}. We now explain in more detail how we choose these test functions. 

For our application, we define our test functions $h_a$ in terms of three parameters : $n$, $m$ and $L$. Let
\begin{align}
\delta=\frac{L}{2m+2n},
\end{align}
introduce
\begin{align}
h:=\left(\frac{1}{2\delta}\mathbbm{1}_{[-\delta,\delta]}\right)^{\ast m}.
\end{align}
Then our test functions $h_a$ are defined for $a=1,\dots,n$ by
\begin{align}
h_a(x):=\frac{1}{2}(h(x+a\delta)+h(x-a\delta)).
\end{align}

In this formula:
\begin{itemize}
\item $L$ is the length of the longest geodesic that appears on the right hand side of the Selberg trace formula. Numerically, we observe that when $L$ gets larger, $I_r(\lambda)$ gets more and more relatively sharp peaks, that is, there are more intervals $[a_i,b_i]$ and $[c_i,d_i]$ such that $b_i-a_i<\epsilon$, $d_i-c_i<\epsilon$ where $\epsilon$ is a parameter characterising the precision of the estimation. The more information we have on the geodesic spectrum, the more eigenvalues we can estimate with a reasonably high precision.
\item $n$ is the size of the $A$ matrix. Numerically, we observe that increasing $n$ makes the peaks of $I_r(\lambda)$ sharper. However,  the integral following $V(\Sigma)/4\pi$ on the geometric side of the Selberg trace formulae needs to be evaluated numerically. A very large value of $n$ produces very fast oscillations in the integrand, making the numerical integration difficult to handle. 
\item $m$ is the number of convolutions. Increasing it also makes the peaks of $I_r(\lambda)$ sharper, but we face a similar problem related to fast oscillations.
\end{itemize}

Figure \ref{fig:peaks} shows the results given by this method when applied to the case of the $[0;3,3,5]$ orbifold. In Appendix \ref{app:tableOfLambdaOne} we tabulate the numerical estimates of $\lambda^{(0)}_1$ and $\lambda_1^{(1/2)}$ for various surfaces and orbifolds. 
\begin{figure}[!ht]
    \centering
    \includegraphics[width=12cm]{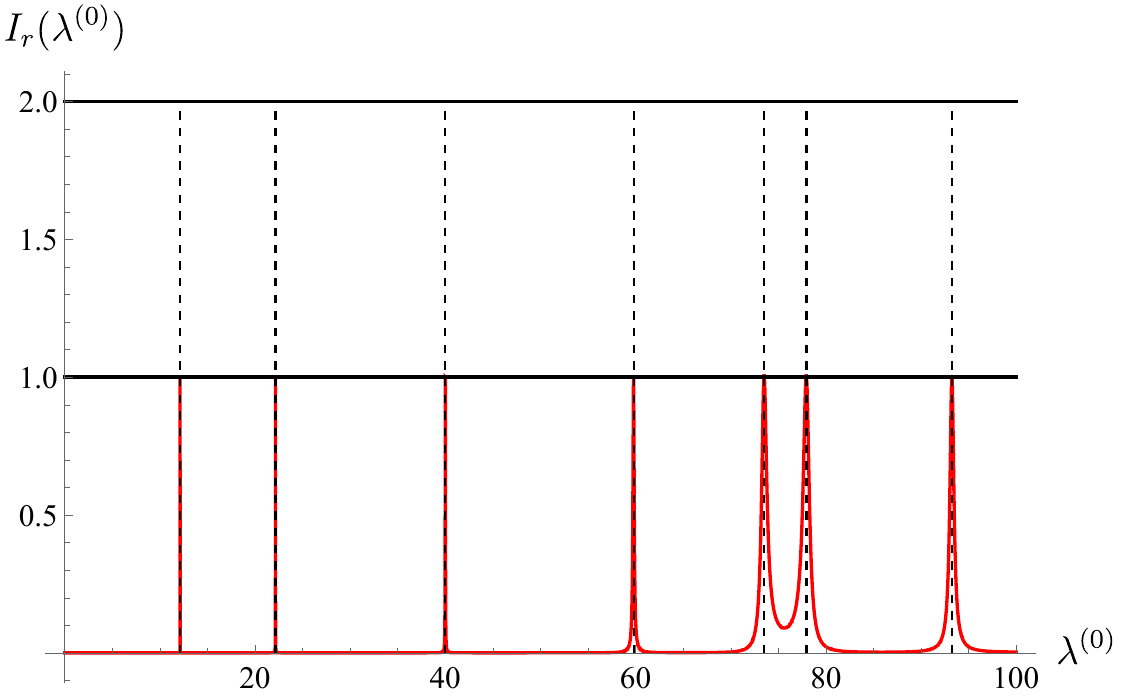}
    \caption{The spectral exclusion plot for the Laplacian spectrum of the hyperbolic triangle [0;3,3,5] with $m=4$, $n=24$ and $L=10.9$. The dashed vertical lines correspond to the first few eigenvalues computed using the finite element method.}
    \label{fig:3,3/2}
\end{figure}
\begin{figure}%
    \centering
    \subfloat[\centering ]{{\includegraphics[width=12cm]{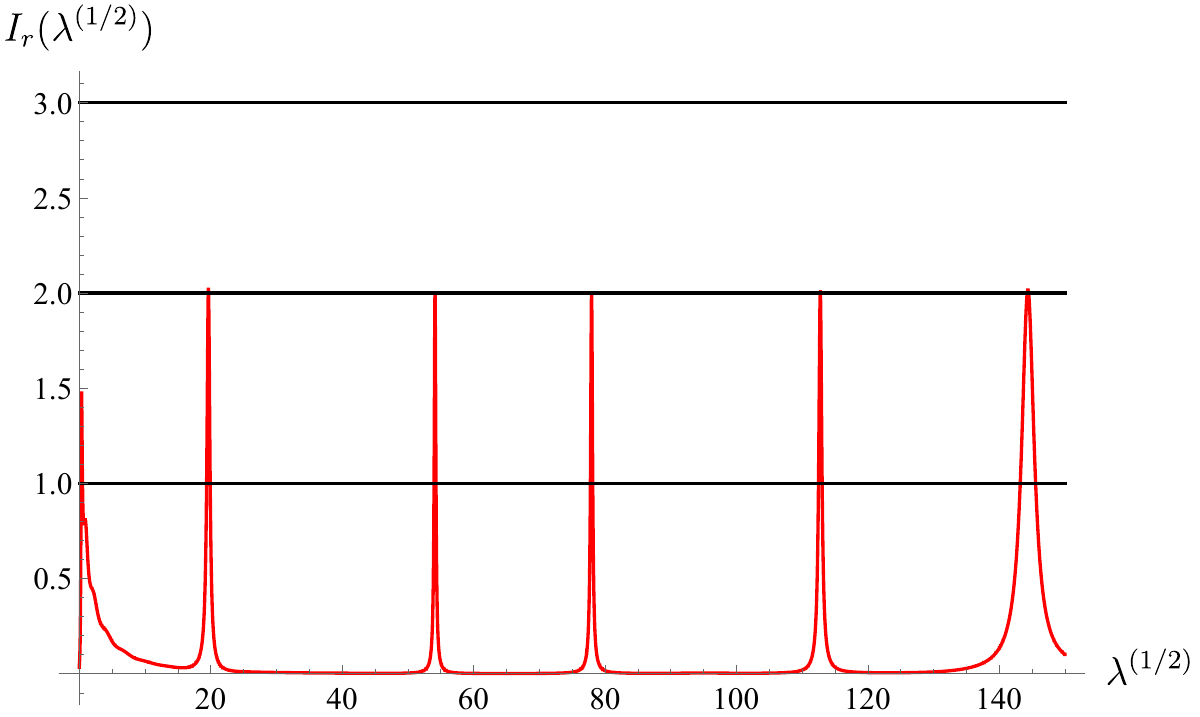} }}%
    \qquad
    \subfloat[\centering ]{{\includegraphics[width=12cm]{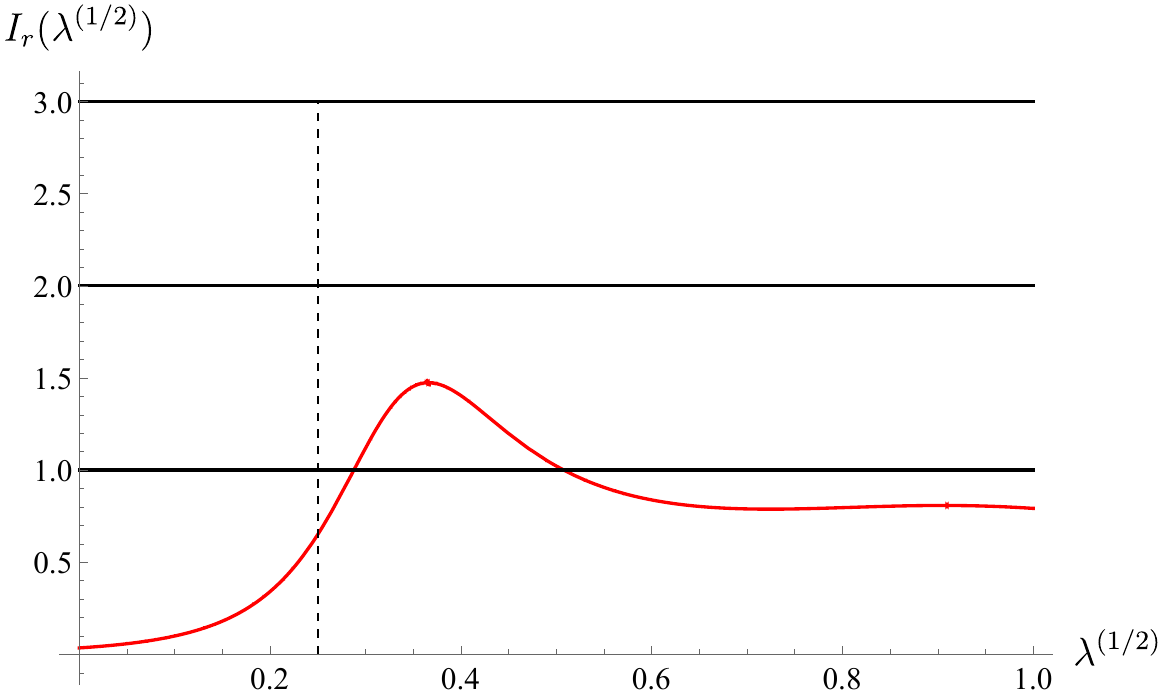} }}
    \caption{(A) Spectral exclusion plot for the 1-Laplacian spectrum of the hyperbolic triangle [0;3,3,5] with $m=4$,$n=24$ and $L=10.9$. The leftmost peak does not imply the existence of any eigenvalues. The curve is below 1 at $\lambda^{(1/2)}=1/4$ therefore it does not indicate the existence of a harmoinc spinor. The peak is also below 2 and by Kramers degeneracy it does not represent a non-harmonic eigenfunction. (B) Spectral exclusion plot for the 1-Laplacian spectrum of the hyperbolic triangle [0;3,3,5] zoomed into the interval from 0 to 1. The dashed vertical line is located at $\lambda^{(1/2)}=1/4$.}
    \label{fig:peaks}%
\end{figure}

\section{Results: bounds from semi-definite programming}\label{sec:results}
In this section we present the proofs of the theorems mentioned earlier, as well as write down further bounds that we obtain using linear and semidefinite programming.  
\subsection{Proofs of theorems using Semi definite programming}


\begin{proof}[Proof of Theorem \ref{thm:1*}]
    For genus $1$ or more, the above bound follows the results in \cite{Kravchuk:2021akc}. For genus $0$ hyperbolic spin orbifold, we use the fact $\ell_{1/2}=0$. Since it is a spin orbifold, the orbifold orders are odd i.e $k_i\geq 3$ and there has to be at least $3$ orbifold points for the surface to be spin. Using Riemann-Roch, we then deduce $\ell_{3}\geq 1$ and $\ell_{3/2}\geq 1$. Furthermore, we have $\lambda^{(1/2)}\geq 1/4$ for all spin orbifolds. Now we leverage \texttt{SDP}\eqref{linpro1} with $n=3$ and choosing $\lambda_*^{(1/2)}=1/4$.  We have verified the bound using rational arithmetic as done in \cite{Kravchuk:2021akc}.
\end{proof}




\begin{proof}[Proof of Theorem \ref{thm:2*}]
   We note that for a hyperbolic spin orbifold $X$ admitting harmonic spinor, we have $\ell_{1/2}\geq 1$ and hence $\ell_1\geq 1$. Furthermore, we have $\lambda^{(1/2)}\geq 1/4$ for all spin orbifolds. Now we  leverage \texttt{SDP}\eqref{linpro1} with $n=1$ and choosing $\lambda_*^{(1/2)}=1/4$.  We have verified the bound using rational arithmetic as done in \cite{Kravchuk:2021akc}.
\end{proof}



\begin{proof}[Proof of Theorem \ref{thm:30*}]
   We leverage \texttt{SDP}\eqref{linpro1} with $n=3$ and choosing $\lambda_*^{(0)}=3/2$.
\end{proof}


\begin{proof}[Proof of Theorem \ref{thm:analytic}]
    Given a compact orientable hyperbolic hyperelliptic spin manifold $X$ of genus $g>2$, we have $\ell_{1/2}=\lfloor (g+1)/2\rfloor>1$. We consider the spectral identities derived from  $\langle \mathscr{O}_{\frac{n}{2}} \mathscr{O}_{\frac{n}{2}} \widetilde{\mathscr{O}}_{\frac{n}{2}}\widetilde{\mathscr{O}}_{\frac{n}{2}} \rangle_{\texttt{SYM}}$ with $n=1$. They are given by (see Proposition \ref{crossEqnMultiple}) 

\begin{equation}
 \begin{aligned}
 \strangefont{S}_{2m+n;n/2}&=\sum_{k=1}^\infty \strangefont{T}_{k;n/2}\ \mathcal{F}_{2m}^{(n/2,n/2,n/2,n/2)}(\lambda_k^{(0)}),\\
  - \strangefont{A}_{2m+1+n;n/2}&=\sum_{k=1}^\infty \strangefont{T}_{k;n/2}\ \mathcal{F}_{2m+1}^{(n/2,n/2,n/2,n/2)}(\lambda_k^{(0)}),\\
\strangefont{S}_{2m+n;n/2}&=\mathcal{F}_{2m}^{(n/2,n/2,n/2,n/2)}(0)+\sum_{k=1}^{\infty}\left(\strangefont{Q}_{k;n/2}^2-\frac{1}{\ell_{n/2}}\strangefont{T}_{k;n/2}\right)\mathcal{F}_{2m}^{(n/2,n/2,n/2,n/2)}(\lambda_k^{(0)}),\\
\strangefont{A}_{2m+1+n;n/2}&=\mathcal{F}_{2m+1}^{(n/2,n/2,n/2,n/2)}(0)+\sum_{k=1}^{\infty}\left(\strangefont{Q}_{k;n/2}^2-\frac{1}{\ell_{n/2}}\strangefont{T}_{k;n/2}\right)\mathcal{F}_{2m+1}^{(n/2,n/2,n/2,n/2)}(\lambda_k^{(0)}),\\
\end{aligned}
\end{equation}
To proceed, we consider the set of identities coming from $m=0$, from which we can derive
\begin{equation}
\begin{aligned}
    0&=\mathcal{F}_{0}^{(1/2,1/2,1/2,1/2)}(0)+\sum_{k=1}^{\infty}\left(\strangefont{Q}_{k;n/2}^2-\left(1+\frac{1}{\ell_{1/2}}\right)\strangefont{T}_{k;1/2}\right)\mathcal{F}_{0}^{(1/2,1/2,1/2,1/2)}(\lambda_k^{(0)})\\
    0&=\mathcal{F}_{1}^{(1/2,1/2,1/2,1/2)}(0)+\sum_{k=1}^{\infty}\left(\strangefont{Q}_{k;n/2}^2+\left(1-\frac{1}{\ell_{1/2}}\right)\strangefont{T}_{k;1/2}\right)\mathcal{F}_{1}^{(1/2,1/2,1/2,1/2)}(\lambda_k^{(0)})
    \end{aligned}
\end{equation}

Plugging in the values of $\mathcal{F}$, we obtain from the above
\begin{equation}
0=\sum_{k=1}^{\infty} \left[\lambda^{(0)}_k \strangefont{Q}_{k;n/2}^2  + \left(\lambda^{(0)}_k  \frac{ (\ell_{1/2}-1)}{\ell_{1/2}}-1\right) \strangefont{T}_{k;1/2}\right]\,.
\end{equation}

The above implies that 
\begin{equation}
    \lambda^{(0)}_1<\frac{\ell_{1/2} }{\ell_{1/2}-1}\,.
\end{equation}
Plugging in $\ell_{1/2}=\lfloor (g+1)/2\rfloor$, the theorem follows.
\end{proof}

\begin{thm}\label{thm:3}
    Given a compact orientable hyperbolic spin 2-orbifold $X$ of genus $g$ admitting $\ell_{1/2}>0$ harmonic spinors,  the first non-zero eigenvalue of the Laplacian operator, $\lambda_1^{(0)}(X)$ satisfies the bound, recorded in Table \ref{table:1}.
\end{thm}

\begin{proof}[Proof of theorem \ref{thm:3}]
    On a hyperbolic spin orbifold $X$ of genus $g$, we have $\ell_{1}=g$ and $\ell_{1/2}\leq \lfloor (g+1)/2\rfloor$. We use the fact that $\lambda^{(1/2)}\geq 1/4$ for all spin orbifolds.  We use \texttt{SDP} (\ref{linpro3}) with $n=1$ and $\lambda_*^{(1/2)}=1/4$. We have verified these bounds using rational arithmetic as done in \cite{Kravchuk:2021akc}.
\end{proof}

\subsection{Exclusion plots}
 As mentioned in the introduction, an \textit{exclusion plot} refers to a region $D\subset (0,\infty)\times (1/4,\infty)$ such that $(\lambda_1^{(0)}(X),\lambda_{1}^{(1/2)}(X))\notin D$, where $X$ can be any compact connected orientable hyperbolic spin surface or orbifold. We may also choose to consider choices of $X$ with prescribed additional properties like genus, number of harmonic spinors it can carry etc. Here $\lambda_1^{(0)}(X)$ is the first nontrivial eigenvalue of the Laplace operator on the compact hyperbolic spin orbifold, while $\lambda_1^{(1/2)}(X)$ is the first nontrivial eigenvalue of the weight-$1$ automorphic Laplacian. In order to obtain the exclusion plots, we perform the following algorithm: 
\begin{enumerate}
    \item Use \texttt{SDP} (\ref{linpro1}) or (\ref{linpro2}), pick a value for $\lambda_*^{(1/2)}$ and search for a functional with the desired properties, leading to a bound $\lambda_1^{(0)}<\lambda_*^{(0)}$ subject to the condition $\lambda_1^{(1/2)}>\lambda_*^{(1/2)}$.
    \item Repeat the above step for various values of $\lambda_*^{(1/2)}$. 
\end{enumerate}
or the following one: 
\begin{enumerate}
    \item Use \texttt{SDP} (\ref{linpro1}) or (\ref{linpro2}), pick a value for $\lambda_*^{(0)}$ and search for a functional with the desired properties, leading to a bound $\lambda_1^{(1/2)}<\lambda_*^{(1/2)}$ subject to the condition $\lambda_1^{(0)}>\lambda_*^{(0)}$.
    \item Repeat the above step for various values of $\lambda_*^{(0)}$. 
\end{enumerate}


Using the aforementioned algorithm involving \texttt{SDP}, we find the disallowed region $D$, which is 
everything except the shaded region (yellow or pink) in the \textit{exclusion} plots. 
We obtain Figure \ref{fig:3,3/2g} by using \texttt{SDP} (\ref{linpro1}) with $n=3$ (for the union of the pink and yellow shaded region) and $n=1$ (for the pink shaded region). The corner point of the yellow shaded region is approximately at $(12.13629,19.67)$, while the corner point of the pink-shaded region is at $(4.7611, 8.28)$, both these points being disallowed. Now, using the Selberg trace formula, we can compute $(\lambda_1^{(0)}(X),\lambda_{1}^{(1/2)}(X))$ for $X=[0;3,3,5]$ (this orbifold has $\ell_{3/2}=\ell_{3}=1$). It is given by $(12.1362327 \pm 10^{-7},19.669\pm 0.03)$. This point is depicted in red on Figure \ref{fig:3,3/2g} and evidently, lies very close to the kink. Figure \ref{fig:3,3/2zoomWithoutRes} zooms onto the kink and shows that the red dot is in the allowed region (up to uncertainty, predicted by the Selberg trace formula). Similarly, using the Selberg trace formula, we can compute $(\lambda_1^{(0)}(X),\lambda_{1}^{(1/2)}(X))$ for $X=[1;3]_{sym}$, the most symmetric point in the moduli space of $[1;3]$, equipped with an odd spin structure such that $\ell_{1/2}=\ell_{1}=1$. In particular we find that $\lambda_1^{(0)}\in (4.7609, 4.7654) $ and $\lambda_{1}^{(1/2)}\in (8.255, 8.298)$. See the table in Appendix \ref{app:tableOfLambdaOne} for a refined interval. For the Laplacian eigenvalue, we can use \texttt{FreeFEM++} as well, this provides the estimate $\lambda_1^{(0)}([1;3]_{sym}) \approx 4.7609$.


Figure \ref{fig:3,3/2} is obtained by using \texttt{SDP} (\ref{linpro2}) with $n=3$. Here we restrict to the case where $\ell_3=\ell_{3/2}=1$, as opposed to using \texttt{SDP} (\ref{linpro1}) with $n=3$, where $\ell_3\geq 1$, $\ell_{3/2}\geq 1$. While Figure \ref{fig:3,3/2} looks almost similar to Figure \ref{fig:3,3/2zoomWithoutRes}, zooming near the kink reveals the differences. Figure \ref{fig:3,3/2zoom} is Figure \ref{fig:3,3/2}, zoomed in near the kink. The corner point is approximately at $(12.13623353125,19.673)$ and disallowed.


Finally, Figure \ref{fig:1/2,1} is obtained by using a simplified version of \texttt{SDP} (\ref{linpro3}) with $n=1$, $\ell_{1}=2$, but $\ell_{1/2}=1$. The red dot here corresponds to the Bolza surface, equipped with an odd spin structure. See the table in Appendix \ref{app:tableOfLambdaOne} for the precise coordinates of the red dot. Figure \ref{fig:1/2,1zoom} is the zoomed version of the above.



\begin{figure}[!ht]
    \centering
    \includegraphics[scale=1.1]{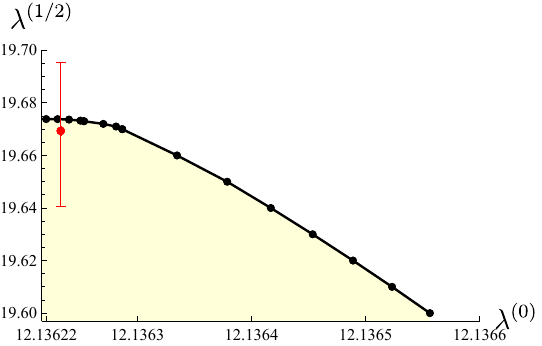}
    \caption{Bounds from the $\ell_{3/2}\geq 1$, $\ell_{3}\geq 1$ system; zooming onto the kink. Everything except the shaded region is disallowed.  The red dot near the corner corresponds to $[0;3,3,5]$.  }
    \label{fig:3,3/2zoomWithoutRes}
\end{figure}

\begin{figure}[!ht]
    \centering
    \includegraphics[scale=1.1]{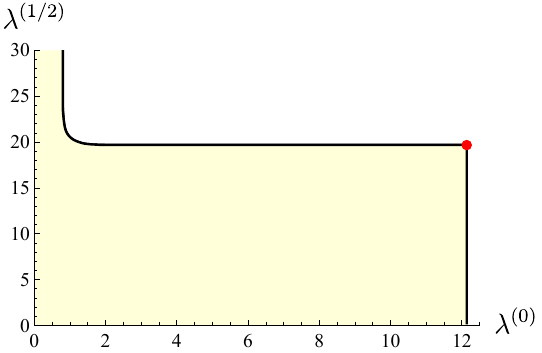}
    \caption{Bounds from the $\ell_{3/2}=\ell_{3}=1$ system. Everything except the yellow shaded region is disallowed for hyperbolic spin orbifolds with $\ell_{3/2}=\ell_{3}=1$.  The red dot in the corner corresponds to $[0;3,3,5]$. }
    \label{fig:3,3/2}
\end{figure}

\begin{figure}[!ht]
    \centering
    \includegraphics[scale=1.1]{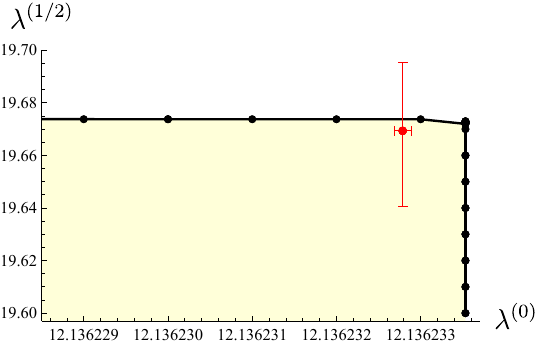}
    \caption{Bounds from the $\ell_{3/2}=\ell_{3}=1$ system; zooming onto the kink. Everything except the shaded region is disallowed.  The red dot near the corner corresponds to $[0;3,3,5]$.  }
    \label{fig:3,3/2zoom}
\end{figure}

\begin{figure}[!ht]
    \centering
    \includegraphics[scale=1.1]{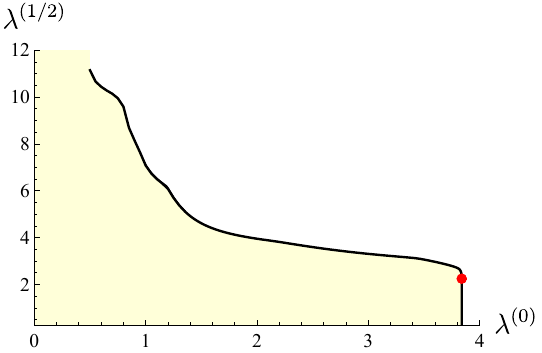}
    \caption{Bounds from the $\ell_{1/2}=1,\ell_{1}=2$ system. Everything except the yellow-shaded region is disallowed.  The red dot in the corner corresponds to Bolza surface with odd spin structure. }
    \label{fig:1/2,1}
\end{figure}

\begin{figure}[!h]
    \centering
    \includegraphics[scale=1.1]{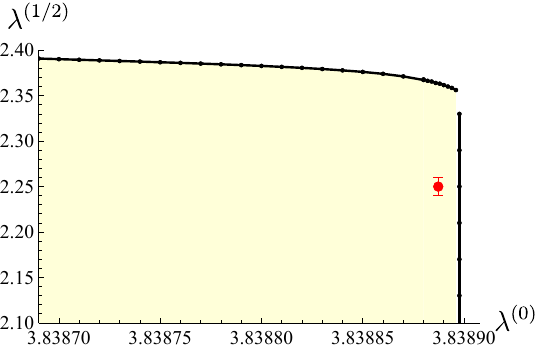}
    \caption{Bounds from $\ell_{1/2}=1,\ell_{1}=2$ system, zooming onto the kink.Everything except the shaded region is disallowed.The red dot corresponds to Bolza surface with odd spin structure. }
    \label{fig:1/2,1zoom}
\end{figure}

\section{Discussion}
\label{sec:discuss}
In this paper, we applied linear programming techniques to constrain the first nonzero eigenvalue of the Laplace and Dirac operators on a compact connected orientable hyperbolic surface or orbifold with a spin structure. The essential tool is the spectral identities coming from the associativity of the product of functions on $\Gamma\backslash \mathrm{SL}(2,\mathbb{R})$. It is the analog of the associativity of the operator product expansion in conformal field theory, which lies at the heart of the conformal bootstrap program.

Our approach is closely related to that of Bernstein and Reznikov \cite{BernsteinReznikov2010},  who investigated the consistency constraints arising from the spectral decomposition of integrals involving a quadruple product of functions within the principal series in $L^2(\Gamma\backslash G)$. See also \cite{Sarnak1994,MichelVenkatesh2010,Nelson2021}. In our current study,  we have demonstrated that integrating these consistency constraints with linear programming techniques, especially when applied to the discrete series, can yield almost optimal bounds on Laplace and Dirac spectra. There is a close analogy of the above with the analytical bootstrap in CFT, in particular, the approach of Bernstein and Reznikov is akin to asymtptotic analysis involving Tauberian theorems, leading to universal behavior of various physical quantities in CFTs \cite{Qiao:2017xif, Mukhametzhanov:2019pzy,Mukhametzhanov:2020swe,Pal:2023cgk}.

Several points could be analysed further. In particular:
\begin{enumerate}
\item It is interesting to note that the bounds we obtain, for example on Laplace eigenvalues, do not always become more stringent as the number of harmonic spinors or genus increases. For example, Table \ref{table:1} shows that the bound at genus 5 with 3 harmonic spinors is less restrictive than the bound at genus 6 with 3 harmonic spinors. We contrast this scenario with that of \cite{Kravchuk:2021akc} where the bound for orbifolds with genus $g_1$ is not weaker compared to that of genus $g_2$ whenever $g_1>g_2$. The origin of monotonicity in the context of \cite{Kravchuk:2021akc} is the fact that the functional used in linear programming for genus $g_2$ can easily be promoted to a functional that is appropriate for genus $g_1$. However, unlike the spectral identities of \cite{Kravchuk:2021akc}, the spectral identities that we bootstrap in this paper have discrete series irreps appearing in the $t$ channel. This constitutes an obstruction to the previous argument. It would be interesting to understand the origin of this non-monotonicity directly and analytically at the level of the spectral identities.

\item In terms of the Dirac spectrum, it is curious that on the exclusion plot of Figure \ref{fig:3,3/2}, the allowed values of $\lambda^{(1/2)}$ seem to blow up quickly as $\lambda^{(0)}$ approaches zero. This raises the question whether our methods could be improved to obtain better bounds in this region of parameter space, and maybe even a universal upper bound on $\lambda^{(1/2)}$.
\item It would also be interesting to study Dirac spectra in higher dimensions using similar spectra identities originating from associativity. In particular the three-dimensional case is quite accessible, extending the work of \cite{Bonifacio:2023ban}.
\end{enumerate}

One can see the results of this paper as a new contribution to an emerging program, started in \cite{Kravchuk:2021akc}, to apply techinques originating from conformal field theory to questions relevant to analytic number theory and hyperbolic geometry. Already in our case, as well as in the three-dimensional case \cite{Bonifacio:2023ban}, these techniques have  been very fruitful. Natural extensions of this work on spinors would be to allow fractional spin structures, or to investigate the case of super-Riemann surfaces. More generally, it would be very interesting to push this program towards quotients of Lie groups of higher rank (for example, $\mathrm{SL}(3,\mathbb{R})$ seems to be a natural next step), as well as the study of non-Archimedean geometries over $\mathbb{Q}_p$. 

Perhaps the most striking feature of our findings is that at least in the cases displayed on Figures \ref{fig:3,3/2zoomWithoutRes} and \ref{fig:1/2,1}, the most symmetric surfaces (i.e. the $[0;3,3,5]$ orbifold, $[1;3]_{sym}$ and the Bolza surface respectively) sit almost exactly at the boundary, or even the kink, of the allowed region in parameter space. The fact that hyperbolic surfaces seem to know about the boundary of the exclusion plot, which purely comes from the constraint equations, seems to suggest that the space of solutions to the constraint equations is closely related to the space of quotients of the form $\Gamma\backslash \mathrm{SL}(2,\mathbb{R})$, where $\Gamma$ is a cocompact lattice. There are well-known results establishing a one-to-one correspondence between abelian algebras and various kinds of geometric spaces. For example, the Gelfand--Naimark theorem \cite{Gelfand:1943} shows that each commutative $C^\ast$-algebra can be thought of as an algebra of continuous functions on a locally compact Hausdorff space. There is an analog of this theorem for von Neumann algebras and measured spaces \cite{Takesaki_2007,pavlov2022gelfand}, and more recently, Connes showed in his reconstruction theorem \cite{Connes:2008vs} that for every spectral triple $(\mathcal{A},\mathcal{H},D)$, with $\mathcal{A}$ commutative, and a few extra conditions, one can realize $\mathcal{A}$ as $C^\infty(X)$, where $X$ is a smooth oriented compact spin$^c$ manifold, and that all such manifolds can be decribed by such a spectral triple. It would be interesting to see whether such reconstruction theorems hold for commutative algebras on which various Lie groups act. The existence of such theorems would likely help understand how closely the space of solutions to the constraint equations described here, and the space of 2-dimensional compact orientable spin-orbifolds, are related.

Finally, we remark on the relevance of our work in context of physics, in particular, conformal field theory. In the present paper, the key role is played by the positivity of the measure of $\widetilde{\Gamma}\backslash \mathrm{SL}(2,\mathbb{R})$. On a similar note we expect that in the context of path integral formulation of conformal field theories, the existence of a $\mathrm{SO}(1,d)$ invariant positive measure along with associativity might lead to the discovery of new constraints on such theories. Additionally, our framework provides a systematic approach for analyzing the de Sitter bootstrap numerically as suggested recently in \cite{Sun:2021thf,Loparco:2023akg,Loparco:2023rug,Penedones:2023uqc}, and for extending this program to include fermions.

\appendix

\section{Dimension of the space of harmonic spinors in low genus}
\label{app:harmonicspinors}

In this work, we obtained bounds on the Laplace and Dirac spectra on a surface or orbifold based on the number of modular forms of various weights on this geometric object. A particularly interesting case has been that of harmonic spinors (i.e. modular forms of weight one), as they carry more than only topological information at genus $g\geq 3$. In this appendix, we summarize the state of the art in the classification of Riemann surfaces of genus $g\geq 3$ in terms of the number of harmonic spinors they can carry. A key role is played by \textit{hyperelliptic surfaces}, which are branched double covers of the Riemann sphere. Note that in genus up to 2, every Riemann surface is hyperelliptic, which informs why the space of harmonic spinors can only retain nontrivial geometric information at genus $g\geq 3$. This appendix cross-references the general results of \cite{Bures1998} with explicit examples of surfaces with many automorphisms from the literature.

\subsection{Genus $3$}

In genus $3$, the situation is particularly nice: either the surface is hyperelliptic, or it admits a representation as a regular quartic in $\mathbb{CP}^2$.

The two cases are defined as follows:
\begin{itemize}
    \item If the surface is hyperelliptic, there exist a choice of spin structure for which the upper bound of two independent harmonic spinors is saturated.
    \item If the surface is a regular quartic, then the maximum number of harmonic spinors is $1$, and it is realized by exactly the odd spin structures on the surface.
\end{itemize}

In genus $3$, the automorphism group of a hyperelliptic surface that has the largest order is $\mathbb{Z}_2\times S_4$. The corresponding surface, given in \cite{https://doi.org/10.48550/arxiv.1711.06599}, has affine equation 
\begin{align}
    y^2=x^8+14x^4+1.
\end{align}
The other candidate special surfaces have automorphism groups $U_6$, $V_8$ (following the notations of \cite{shaska2013subvarieties}), and $\mathbb{Z}_{14}$, see \cite{https://doi.org/10.48550/arxiv.1711.06599} for more details.

On the other hand, the automorphism group of a non-hyperelliptic surface of genus $3$ can be much larger \cite{genus3case}. In particular, there exists a non-hyperelliptic surface of genus $3$ with isometry group $\mathrm{PSL}(2,\mathbb{F}_7)$, that has 168 elements. This surface is called the \textit{Klein quartic}, and has equation 
\begin{align}
    z^3x+x^3y+y^3z=0.
\end{align}
The complete classification of possible automorphism groups of genus $3$ surfaces is given in \cite{genus3case}. 

\subsection{Genus $4$}

In order to understand the situation in genus $4$, we need to introduce the notion of \textit{canonical embedding}. 

\begin{defn}
Let $\Sigma$ be a Riemann surface, and let $\{\omega_1,\dots,\omega_g\}$ be a basis of holomorphic differentials on $\Sigma$. The canonical embedding of $\Sigma$ into $\mathbb{CP}^{g-1}$ is the map
\begin{align}
    \sigma\longmapsto[\omega_1(\sigma),\dots,\omega_g(\sigma)].
\end{align}
\end{defn}
Note that the canonical embedding depends on the choice of basis of holomorphic differentials, however, a change of basis simply amounts to a projective transformation in $\mathbb{CP}^{g-1}$. We have the following general result:

\begin{prop}[\cite{Narasimhan1992}]
The canonical embedding of $\Sigma$ is injective if and only if $\Sigma$ is not hyperelliptic.
\end{prop}

In genus $4$, we have the additional fact:

\begin{prop}[\cite{Bures1998}]
The canonical model of a non-hyperelliptic surface of genus $4$ is the intersection of a unique quadric and a cubic in $\mathbb{CP}^{3}$. Moreover, the quadric is either smooth or a cone.
\end{prop}

It turns out \cite{Bures1998} that this classification in terms of canonical models is the right one to consider to classify the possible dimensions of harmonic spinors. More precisely:

\begin{itemize}
    \item If the surface is hyperelliptic, there exist spin structures for which the surface admits 2 harmonic spinors.
    \item If the surface is not hyperelliptic and the quadric supporting its canonical embedding is a cone, then it admits a unique spin structure with 2 harmonic spinors. 
    \item If the surface is not hyperelliptic and the quadric supporting its canonical embedding is a smooth, then the maximum number of harmonic spinors is $1$, and it is realized by exactly the odd spin structures on the surface. 
\end{itemize}

In genus $4$, still with the notations of \cite{shaska2013subvarieties}, the automorphism group of a hyperelliptic surface that has the largest order is $V_{10}$. The corresponding surface, given in \cite{https://doi.org/10.48550/arxiv.1711.06599}, has affine equation 
\begin{align}
    y^2=x^6-1.
\end{align}
The other candidate special surfaces have automorphism groups $\mathrm{SL}(2,\mathbb{F}_3)$, $U_8$, and $\mathbb{Z}_{18}$, see \cite{https://doi.org/10.48550/arxiv.1711.06599} for more details.

In a similar fashion to the case of genus $3$, there exist non-hyperelliptic surfaces of genus $4$ with much larger automorphism groups. These surfaces are enumerated in \cite{https://doi.org/10.48550/arxiv.2204.01656}. The surface with the largest automorphism group is \textit{Bring's surface}, with $120$ automorphisms. However, the quartic associated to Bring's surface is smooth. 

If we want to impose that the canonical embedding of the surface lies on a cone, the largest possible automorphism group has order $72$. The corresponding surface has equation 
\begin{align}
    z^3y^2=x(x^4+y^4).
\end{align}

A large number of other large groups (but of smaller order) arise, they are described in detail in \cite{https://doi.org/10.48550/arxiv.2204.01656}.

\subsection{Genus $5$}
 
In genus $5$, a surface can carry a maximal number of harmonic spinors if and only if it is hyperelliptic. The automorphism group of a hyperelliptic surface that has the largest order ($120$) is $\mathbb{Z}_2\times A_5$. The corresponding surface, given in \cite{https://doi.org/10.48550/arxiv.1711.06599}, has affine equation 
\begin{align}
    y^2=x(x^{10}+11x^5-1).
\end{align}
The other candidate special surfaces have automorphism groups $W_2$, $U_{10}$, $V_{12}$ and $\mathbb{Z}_{22}$ \cite{shaska2013subvarieties}, see \cite{https://doi.org/10.48550/arxiv.1711.06599} for more details.

If one drops the requirement of hyperellipticity, there is a surface that has more automorphisms: its order is $192$. This surface is described in \cite{Conder2011EQUATIONSOR,https://doi.org/10.48550/arxiv.2204.01656}.

\subsection{Genus $6$}

In the case of genus $6$, there are once again special surfaces with the maximal ($3$) number of harmonic spinors but that are not hyperelliptic. More precisely,

\begin{itemize}
    \item If the surface is hyperelliptic, there exist spin structures for which the surface admits 3 harmonic spinors.
    \item If the surface is a smooth plane quintic, then it admits a unique spin structure with 3 harmonic spinors. 
    \item If the surface is neither hyperelliptic nor a plane quintic, then there are at most 2 harmonic spinors per spin structure.
\end{itemize}

Note that the classification of the remaining cases is subtle. It is known, however \cite{Bures1998}, that trigonal surfaces cannot have more than $1$ harmonic spinor per spin structure.

In genus $6$, the automorphism group of a hyperelliptic surface that has the largest order is $V_{14}$ (following the notations of \cite{shaska2013subvarieties}). The corresponding surface, given in \cite{https://doi.org/10.48550/arxiv.1711.06599}, has affine equation 
\begin{align}
    y^2=x^{14}-1.
\end{align}
The other candidate special surfaces have automorphism groups $U_{12}$, $GL(2,\mathbb{F}_3)$ and $\mathbb{Z}_{26}$, see \cite{https://doi.org/10.48550/arxiv.1711.06599} for more details.

Once again, there exist non-hyperelliptic surfaces of genus $6$ with much larger automorphism groups. These surfaces are enumerated in \cite{Bures1998}. The surface with the largest automorphism group is \textit{Fermat's quintic} \cite{Bures1998}, with $150$ automorphisms. It is a smooth plane quintic of equation 
\begin{align}
    x^5+y^5=z^5,
\end{align}
so interestingly, it also can carry three harmonic spinors.

The surface of genus 6 that \textit{cannot} carry 3 harmonic spinors and has the largest automorphism group is the Wiman sextic, described for example in \cite{https://doi.org/10.48550/arxiv.2204.01656}. 
Its automorphism group is isomorphic to $S_5$, which means that it has $120$ automorphisms.
\newpage
\section{ Numerical estimate of $\lambda^{(0)}_1$ and $\lambda^{(1/2)}_1$ for various orbifolds and surfaces}
\label{app:tableOfLambdaOne}
\begin{table}[!ht]
	\begin{tabular}[t]{ccc}\toprule
Signature  & Interval containing $\lambda_1^{(0)}$ & Interval containing $\lambda_1^{(1/2)}$ \\
\midrule
\multirow{1}{*}{  [0;3,3,5]  } & [12.13623266082, 12.13623279684] & [19.62850299650, 19.70606979308] \\
\cmidrule(lr){1-3}
\multirow{1}{*}{  [0;3,3,7]  } & [6.622512981830, 6.62251303689] & [14.58137931985, 14.70308279925]\\
\cmidrule(lr){1-3}
\multirow{1}{*}{  [0;3,3,9]  } & [4.760935531772, 4.760935540974] & [13.04476106136, 13.20743304307]\\
\cmidrule(lr){1-3}
\multirow{1}{*}{  [0;3,3,11]  } & [3.817638624612, 3.817638645376] & [12.31511817207, 12.56227388521]\\
\cmidrule(lr){1-3}
\multirow{1}{*}{  [0;3,3,13]  } & [3.243870176473, 3.243870434758] & [11.89945757571, 12.22641050708]\\
\cmidrule(lr){1-3}
\multirow{1}{*}{  [0;3,3,15]  } & [2.856060123567, 2.856061719363] & [11.66435497067, 12.00545171009] \\
\cmidrule(lr){1-3}
\multirow{1}{*}{  [0;3,3,17]  } & [2.575237981382, 2.575242811203] & [11.50074150836, 11.87999680844] \\
\cmidrule(lr){1-3}
\multirow{1}{*}{  [0;3,3,19]  } & [2.361748257939, 2.372476650392] & [11.23045200383, 11.84361592566] \\
\cmidrule(lr){1-3}
\multirow{1}{*}{  [0;3,3,21]  } & [2.193498651071, 2.195285976813] & [11.19103656533, 11.78223049525] \\
\cmidrule(lr){1-3}
\multirow{1}{*}{  [0;3,3,23]  } & [2.057131576807, 2.057705367329] & [11.11887135826, 11.73877352328] \\
\cmidrule(lr){1-3}
\multirow{1}{*}{  [0;3,3,25]  } & [1.944123693427, 1.950822253982] & [11.06202801296, 11.70678314765] \\
\cmidrule(lr){1-3}
\multirow{1}{*}{  [0;3,5,5]  } & [5.873575959007, 5.873576043568] & [10.26823286914, 10.46485970804] \\
\cmidrule(lr){1-3}
\multirow{1}{*}{  [0;3,5,7]  } & [4.105916028717, 4.105919429281] & [8.807322253657, 9.080333571881] \\
\cmidrule(lr){1-3}
\multirow{1}{*}{  [0;3,5,9]  } & [3.240661780680, 3.240671025102] & [8.255608604384, 8.583931161697] \\
\cmidrule(lr){1-3}

\multirow{1}{*}{  [0;3,5,11]  } & [2.734102690099, 2.735649379604] & [7.907206673453, 8.351477142335] \\
\cmidrule(lr){1-3}

\multirow{1}{*}{  [0;3,5,13]  } & [2.400941404204, 2.475085745560] & [7.726023042016, 8.311960209309] \\
\cmidrule(lr){1-3}
\multirow{1}{*}{  [0;3,5,15]  } & [2.163736742483, 2.379780806385] & [7.585768436685, 8.286345064939] \\

\cmidrule(lr){1-3}
\multirow{1}{*}{  [0;3,7,7]  } & [3.253194157760, 3.263120231408] & [7.649577360853, 8.068585239517] \\

\cmidrule(lr){1-3}
\multirow{2}{*}{  $[1;3]_{\text{sym}}$  } & \multirow{2}{*}{[4.760933182368, 4.765358782461]} & $[8.255418967910, 8.297252909660]^{*}$ \\
& & [3.108229958351, 3.190576244169]\\

\cmidrule(lr){1-3}
\multirow{3}{*}{  Bolza surface  } & \multirow{3}{*}{[3.838886940769, 3.842772834639]} & $[2.246498128260, 2.259880718024]^{*}$ \\
& & [1.188383272199, 1.192003537383]\\
& & [0.8970133121437, 0.8975387540243]\\

\bottomrule
\end{tabular}
	\caption{$\lambda_1^{(0)}$ and $\lambda_1^{(1/2)}$ of various hyperbolic spin surfaces and spin orbifolds. Orbifolds with signature [1;3] have 4 different spin structures but for the most symmetric point on the moduli space, there are only two different sets of spectra for the Dirac operator. Similarly, the Bolza surface has 16 different spin structures but there are only three different Dirac spectra. The intervals marked with $*$ correspond to spin structures that support a harmonic spinor.}
 \label{table:2}
\end{table}
\newpage
\bibliographystyle{myJHEP}
\bibliography{references}
\end{document}